\documentclass[12pt,makeidx]{amsart}
\makeindex

\catcode`\Ž=\active\def Ž{\'e}
\catcode`\=\active\def {\`e}
\catcode`\ˆ=\active\def ˆ{\`a}

\catcode`\=\active\def {\^{e}}
\catcode`\"=\active\def "{\^{i}}
\catcode`\™=\active\def ™{\^{o}}
\catcode`\ž=\active\def ž{\^{u}}
\catcode`\=\active\def {\c{c}}

\def\un{\underline}

\def\g{\mathfrak{g}}
\def\a{\mathfrak{a}}
\def\b{\mathfrak{b}}

\def\cc{\check}
\def\mm{\mathcal}

\def\bb{\backslash}

\def\F{\bold{F}}

\def\N{\mathbb{N}}

\def\Z{\mathbb{Z}}
\def\R{\mathbb{R}}
\def\C{\mathbb{C}}

\def\si{\sigma}
\def\ep{\varepsilon}

\def\VV{{\check V}}
\def\dd{\delta_P^{1/2}}
\def\vid{\emptyset}

\def\LL{\Lambda}
\def\ll{\lambda}
\def\aa{\alpha}

\def\ste{\par\smallskip\noindent}

\def\beq{\begin{equation}}
\def\eeq{\end{equation}}

\newenvironment{res}
               {\begin{equation}\begin{minipage}{0.85\textwidth}}
               {\end{minipage}\end{equation}}
\def\ber{\begin{res}}
\def\eer{\end{res}}

\newtheorem{theo}{Theorem}[section]
\newtheorem{lem}[theo]{Lemma}

\newtheorem{prop}[theo]{Proposition}

\theoremstyle{definition}
\newtheorem{defi}[theo]{Definition}

\theoremstyle{remark}
\newtheorem{rem}[theo]{Remark}

\numberwithin{equation}{section}

\def\ste{\par\smallskip\noindent}

\title{Constant term of  Eisenstein integrals on a reductive $p$-adic symmetric space }

\author{Jacques Carmona, Patrick Delorme}

\address{               Aix Marseille Universit\'e\ste
CNRS-IML, FRE 3529\ste
163 Avenue de Luminy\ste 13288 Marseille Cedex 09\ste  France.           }
\email{carmona@iml.univ-mrs.fr,  delorme@iml.univ-mrs.fr}
\thanks{ We thank  warmly  the referees for numerous helpful remarks and comments. The second author thanks very much Joseph Bernstein for enlighting 
conversations during the elaboration of this work. 
We also thank   Jean-Pierre Labesse and Bertrand Lemaire for their decisive help on algebraic groups.\\
 The second  author has been supported by the program ANR-BLAN08-2-326851 and by the Institut Universitaire de France  during the elaboration of this work.}


\date{}

\dedicatory{} 


\begin{document} 
\begin{abstract} Let {$H$}  be the fixed point group of a rational involution $\si$ of a reductive $p$-adic group on a field of characteristic different from 2.
Let $P$ be a $\si$-parabolic subgroup of $G $ i.e. such that  $\si(P)$ is opposite to $P$.  We denote  the intersection $P\cap \si(P)$ by $M$.  \\
Kato and Takano on one hand, Lagier on the other hand   associated canonically to an $H$-form, i.e. an $H$-fixed linear form, $\xi$,  on a smooth admissible $G$-module, $V$, a linear form on the Jacquet module $j_P(V)$ of $V$ along $P$ which is fixed by $M\cap H$. We call this operation constant term of $H$-forms. This constant term is linked to the asymptotic behaviour  of the generalized coefficients with respect to $\xi$.

 P. Blanc and the second author defined a  family of  $H$-forms on certain parabolically induced representations, associated to an  $M\cap H$-form, $\eta$,  on the space of the inducing representation. 
 \\ The purpose of this  article is to describe the constant term of  these $H$-forms.
  
Also it is shown that when $\eta$ is discrete,  i.e. when the generalized coefficients of $\eta$ are square integrable modulo the center, the corresponding family of  $H$-forms  on the induced representations is a family of  tempered, in a suitable sense, of  $H$-forms. 
A formula for the constant term of Eisenstein integrals is given. \\
\end{abstract}

\maketitle
\section{Introduction}

\noindent Let $\F$ be a  non archimedean local field of characteristic different from two. The hypothesis on the characteristic of the residue field in the first version of the article has been removed (see Proposition \ref{KT} and section \ref{Lie}).

Let $G$  be the set of $\F$-points of a connected reductive algebraic group defined over 
$\F$.
Let $\si$ be a rational involution of $G$ defined over $\F$ and  let $H$ be equal  to the fixed point group of $ G$ (a slightly weaker assumption is made in the main body of the article).

  If $(\pi,V)$ is a smooth representation of $G$,   an $H$-fixed element  of the  algebraic dual $V'$ of $V$ i.e. $\xi \in V'^H$, will be called an $H$-form on $V$.  We know (cf. [D3], Theorem 4.5) that if $V$ is of finite length the vector space of $H$-forms on $V$ is finite dimensional. We will denote   the   smooth dual of $V$ by $\check{V}$. 
  
  If $(\pi,V)$ is a smooth representation of $G$ and $P$ is a parabolic subgroup of $G$,   we will denote the normalized Jacquet module along $P$ by $(\pi_P, V_P)$ and 
 the canonical projection $V\to V_P$ by $j_P: V \to V_P$. For $v\in V$, we will   also denote  $j_P (v)$ by $v_P$. 

 A parabolic subgroup of $G$ is called a  $\si$-parabolic subgroup if and only $P$ and $P^-:=\si(P)$ are opposite (studied by A. Helminck and Wang, A. Helminck and G. Helminck, cf. [HW] and [HH]). Then 
$M:=P\cap \si(P)$ is the $\si$-stable Levi subgroup of $P$ and $PH$ is open in $G$. Let $U$ be the unipotent radical of $P$. 

Let  $(\pi,V)$ be an admissible representation of $G$ and let $P$ be a $\si$-parabolic subgroup of $G$. Let $\pi'$ be the  dual representation of $\pi$ on $V'$. Independently  Kato-Takano [KT1] and Lagier [L]
have associated to   each $H$-form $\xi$ on $V$ an $M\cap H$-form 
$j_{P^-} \xi$, also  denoted  by $\xi_{P^-}$,  on  $V_P $ such that,   for all $v\in V$, there exists $\varepsilon$ satisfying:
$$\delta^{-1/2}_{P} (a) \langle \xi, \pi(a) v\rangle = \langle( \xi_{P^-}, \pi_P( a) v_P\rangle, a \in A^-_M( \varepsilon)$$
where $A^-_M( \varepsilon)$ is some translate of  the negative  Weyl chamber in the maximal split torus, $A_M$, of the center of $M$  and  $\delta_P$ is the modulus function of $P$. 
This is a characteristic property of $ \xi_{P^-} $ which  describes the asymptotics of generalized  coefficients. The linear form $\xi_{P^-}$ will be 
called the constant  term of $\xi$ along $P$.

By the Second Adjointness Theorem (Casselman,  Bernstein, cf. [C], [Be], [R] VI.9.6) $(\check{V})_{P^-} $ is canonically isomorphic to $(V_P)\check{}$,  
i.e. there  exists a canonical non degenerate pairing between $(\check{V})_{P^-}$ and  $V_P$, denoted by $\langle.,.\rangle _P$.

If $K$ is a compact open subgroup, let $e_K$ be  the normalized Haar measure on $K$ and $ e_K \xi $ be  the smooth linear form defined by:
$$\langle e_K \xi,v \rangle= \langle \xi, \pi(e_K)v \rangle, \> v\in V.$$

We define a notion of  $(P,H)$-good compact open subgroup of $G$ (cf. (\ref{phgood})). 
For $K$ a $(P,H)$-good compact open subgroup of $G$, one has:
$$\langle\xi_{P^-}, v_P\rangle = \langle (e_K\xi)_{P^-}, v_P\rangle_P, v \in V^K. $$
This property justifies the notation $\xi_{P^-}$. 
There exist arbitrary small $(P,H)$-good compact open subgroups of $G$ (Kato-Takano [KT1], Lemma 4.6 and Lagier [L], Theorem 1 (ii)). Hence  $ \xi_{P^-}$ is determined by  the various $(e_K\xi)_{P^-}$.

  P. Blanc and the second author defined  families of   $H$-forms on parabolically induced representations (cf. [BD]). The purpose of this  article is to describe the constant term of these $H$-forms.
  
For this we prepare several tools.
First we  introduce two operations on $H$-forms on induced representations.

The first one is denoted $\cc{j}_{Q^-} \circ$ (cf. Section \ref{ccjq}).
Let $P=MU$ be a parabolic subgroup of $G$, as above. Let $( \delta,E)$ be a smooth representation of finite length of $M$ and  let $i^G_PE $ be the normalized parabolically  induced representation. Our choice is such that the elements of $i^G_PE $ are left covariant $E$-valued functions and  $G$ acts by the  right regular representation. Let 
$(Q,Q^-)$ be a pair of opposite parabolic subgroups  of $M$. Let $P_Q = Q U\subset P, P_{Q^-} = Q^- U\subset P$.
We define a homomorphism of $G$-modules 
$$\cc{j}_{Q^-} \circ: ( i^G_PE) \check{}\to  (i^G_{P_{Q^-} }E_Q)\,\check{} ,$$     
as the composition of four maps:    
$$ (i^G_PE)\,\check{ }\to i^G_P\check{ E}\stackrel{f}{\rightarrow}  i^G_{P_{Q^-}}((\check {E})_{Q^-}) \stackrel{g}{\rightarrow} i^G_{P_{Q^-}}((E_Q)\,\check{})\to
 (i^G_{P_{Q^-}}E_Q)\,\check{},$$
 where the first and the last maps are deduced from  a choice (cf. (\ref{ccigp})) of  an isomorphism between the smooth dual of  a parabolically  induced representation and the parabolically induced representation of the smooth dual , where $f$  is given by the  composition with $j_{Q^-}$ and $g$ is the induced map of the isomorphism between $(\check {E})_{Q^-}$ and $(E_Q)\,\check{}$ given by the Second Adjointness Theorem.
 
 This operation on smooth linear forms easily extends  to  $H$-form:  one associates  to an  $H$-form, $\xi$,  on $ i^G_PE$,  an $H$-form $\cc{j}_{Q^-} \circ \xi  $ on  $ i^G_{P_{Q^-}}(E_Q) $ (cf. Proposition \ref{defccj}).
 
Let $P_1=M_1U_1$ be a parabolic subgroup of $G$ such that $P\subset P_1$, $M\subset M_1$ and such that $P_1H$ is open. The second operation is denoted $\cc{r}_{M _1}$ (cf. Section \ref{two}). Let $\xi$ be an  $H$-form on $i^G_PE$. Inducing in stages,  
$\xi$ is an $H$-form on $i^G_{P_1} E_1$ where  $E_1:=  i^{M_1}_{P\cap M _1}E$.
The linear form $\xi$ might be viewed as an  $E_1$-distribution which is right invariant under $H$ and left covariant under $P_1$. 
The restriction of this distribution   to the open set $P_1H$  is simply a function. Its value at $1$ is denoted $\cc{r}_{M _1}\xiÊ\in (E_1)'^{M_1\cap H}$.

The next tool is what we call the  Generic Basic Geometric Lemma (cf. Proposition \ref{gb}). 
Let $P=MU$ be a $\si$-parabolic subgroup of $G$.
Let $(\delta,E)$ be a finite length  smooth representation of $M$. Let $X(M)_\si$ be  the identity  component of the set of $\si$-antiinvariant elements of $X(M)$. If $\chi \in X(M)_\si$, we denote the space of the representation $\delta_\chi:= \delta\otimes \chi$  by $E_\chi$. If $x\in G$ and $X$ is a subset of $G$, $xXx^{-1}$ will be denoted by $x.X$. Then  $x\delta$  is the representation of $x.M$ on $xE:=E$,  such  that, if $m\in M$,   $x\delta(xmx^{-1})=Ê\delta(m)$. If there is no ambiguity, we will denote  the bijective intertwining operator between $i^G_P E $ and $i^G_{x.P} xE$ defined by
   $\ll(x) v(g)= v(x^{-1}g), g \in G,$. Also we will  also denote  the transpose of the inverse of $\ll(x )$ again   by $\ll(x)$.

Let $P'=M'U'$ be another $\si$-parabolic subgroup of $G$. Then we have the following result that we call the Generic Basic Geometric Lemma (or rather its dual  version, cf. Proposition \ref{betaw}). \\
{\bf Lemma}
{\em 
 For  a good choice $W(M'\backslash G/ M)$ of  a set of representatives of
 $P'\backslash G  /P$, one  defines successively for $w \in W(M'\backslash G/ M)$:
$$X_{\chi,w}:= i^{M '}_{M' \cap w.P} w ((E_\chi)_{M\cap w^{-1}.P'}), $$
$$\tilde{P} _w = (M\cap w^{-1}.P^{'-}) U, \> \tilde{P} '_w= (M' \cap w.P) U'^- AA, $$   and  for $\chi$ in  a suitable open dense subset of $X(M)_\si$:
\\ the transpose of the intertwining integral $^tA(w.  \tilde{P}_w,  \tilde{P}'_w, wj_{M\cap w^{-1}.P'} \delta_\chi) $
 \\ and  $\beta_{\chi, w}: (i^G_PE_\chi )\,\check{} \to (X_{\chi,w})\,\check{}$
  by: $$\beta_{\chi, w} = \cc{r}_{M'} \circ\>  ^tA(w.  \tilde{P}_w,  \tilde{P}'_w, wj_{M\cap w^{-1}.P'} \delta_\chi) \circ  \lambda (w)\circ (\check{j} _{M\cap w^{-1}.P^{'-}} \circ) $$ 
 so that the linear map $ \oplus_{w\in W(M'\backslash G /M)}\beta_{\chi,w}: (i^G_PE_\chi)\,\check{} \to (\oplus_{w\in P'\backslash G /P}X_w)\, \check{}$
 goes through the quotient to  an isomorphism: $$\beta_\chi: ((i^G_PE_\chi)\,\check{}\, )_{P^{'-} } \to \check{X_\chi},$$
 where $X_\chi =\oplus_{w\in W(M'\backslash G /M)}X_{\chi,w}$}. 
 
 The proof of the Generic Basic Geometric Lemma requires the study of bijectivity of intertwining integrals and the notion of infinitesimal character linked to the Bernstein center (cf. Section \ref{pb}). This notion  allows  us to show that,  for $\chi$ in an open dense subset of $X(M)_\si$, two distinct $X_{\chi, w}$ do not  have irreducible subquotients in common.  The end of the proof uses the Basic Geometric Lemma which describes   the graded  object associated to a natural  filtration of the Jacquet module of  a parabolically  induced representation.
 
From our isomorphism  $\beta_\chi: ((i^G_PE_\chi)\, \,\cc{}\,)_{P '^-} \to \check{X_\chi}$,  $X_\chi $ is identified with $(i^G_PE_\chi)_{P '}$, by the   Second Adjointness Theorem.  Hence, if $\xi$ is an $H$-form on $i^G_PE_\chi$, $ \xi_{P^{'-}} $   is identified   with a linear form  denoted  again $ \xi_{P^{'-}} $ on $X_\chi$, with components $ \xi_{P'^-, _w}\in (X_{\chi,w})' $.

By reduction to the  $e_K\xi$, and by unwinding the definitions, one proves a preliminary result (cf. Theorem \ref{xipw}): \\
{\bf Theorem 1} 
{\em For $\chi$ in  an open dense subset of $X(M)_\si$:}
$$\xi_{P^{ '-}, w} = \cc{r}_{M'} \circ\>  ^tA(w.  \tilde{P}_w,  \tilde{P}'_w, w j_{M\cap w^{-1}.P'} \delta) \circ  \lambda (w)\circ \check{j} _{M\cap w^{-1}.P^{'-}} \circ  \xi .$$

To simplify the exposition, we assume, only in the introduction, that there is only one open  $(P,H)$-double coset, $PH$. 
A theorem due to Blanc and the second author (cf. [BD], Theorem 2.8) can be stated as follows. For $ \eta \in E^{' M\cap H}$ ,  there exists a unique  rational map on $X(M)_\si$, $\chi \to \xi(P, \delta_\chi, \eta) \in (i^G_P \delta_\chi)^{'H}$, with  $\cc{r}_M\xi(P, \delta_\chi, \eta)  =\eta$. The proof of this theorem adapts to the weaker hypothesis on $\F$ in this article. 

 Let $Q$ be a $\si$-parabolic subgroup of $G$ whose  $\si$-stable Levi subgroup is equal to $M$. Then one sees (cf. Proposition \ref{Bmat}) that 
 there is a rational map on $X(M)_\si $ (the  analogue of the $B$-matrices of van den Ban, [Ba]),   $\chi \mapsto B(Q,P, \delta_\chi) \in \mathrm{End} (E^{' M\cap H})$  such that:
$$ ^t A(P, Q, \delta_\chi) \xi(P, \delta_\chi, \eta) = \xi (P, \delta_\chi, B(Q,P, \delta_\chi) \eta).$$

Let $(\pi, V) $ be a smooth representation of $G$ and let $\xi$ be an $H$-form on $V$. For $v\in V$, we denote the smooth function on $H\bb G$ defined by:
$$c_{\xi,v} (Hg)= \langle \xi, \pi(g) v \rangle, g \in G$$
 by $c_{\xi,v}$.  
If $\pi$ is of finite length, hence admissible,   $\xi$,  is said to be cuspidal or $H$-cuspidal  if one of the equivalent statements holds (Kato-Takano [KT1]):
\\1) $\xi_{P^-} =0$ for every proper $\si$-parabolic subgroup, $P$,  of $G$.
\\2) For all $v\in V$,   the function   $g \mapsto  \langle \xi, \pi(g) v \rangle  $ has a compact support modulo $HZ$, where $Z$ is the center of $G$.
 
Assuming that $(\pi,V)$ is irreducible and unitary, an $H$-form is said to be   discrete if the $c_{\xi,v} $ are square integrable on $H \backslash G/A_G$. Kato and Takano (cf.  [KT2], Theorem 4.7) have characterized a discrete  $H$-form $\xi$ by a  criterion analogous to a criteria of Casselman ([C], Theorem 4.4.6) for square integrable representations of the group. It involves a condition on $ j_{P^-} \xi$ for each $\si$-parabolic subgroup of $G$, $P$.
We define   the  tempered $H$-forms by a similar property (cf. Definition \ref{xiR}). One has the following result (cf. Theorems \ref{2}, \ref{75}, \ref{33}). \\
{\bf  Theorem 2}{\em
\\ Let $ \eta\in E'^{M\cap H}$. If $ \chi \in X(M)_\si$ we denote  the $H$-form $\xi(P, \delta_\chi, \eta)$ by $\xi_\chi$, when it is defined. \\
(i) Let us assume that  $(\xi_\chi)_{P'^-,w}$ is non identically   zero when $\chi$ varies in $X(M)_\si$. Then one may change  $w$  in  the same double  $(P', P)$-coset  such that:
\\a) $M'\cap w. P$ is a $\si$-parabolic subgroup of $M '$, $M\cap w^{-1}.P'$ is a $\si$-parabolic subgroup of $M$,   $\tilde{P}'_w $, $P_w$  and $w. \tilde{P}_w$ are $\si$-parabolic subgroups of $G$. \\ 
b) The exponents (cf. Definition \ref{expo}) of $\xi_\chi$ along $P'$  are  explicitly  controlled by the  exponents of $\eta$  along  certain $\si$-parabolic subgroups of $M$. 
\\(ii) With these choices of $w$,  if  $ \eta$ is cuspidal, one has  $w.M \subset M'$, $M\cap w^{-1}.P^{'-}=M$, so that ${\tilde P}_w= P$ and   
$$(\xi_\chi)_{P^{'-},w} = \xi (M'\cap w.P,w\delta_\chi,  B(\tilde{P}'_w , w. P, w \delta_\chi) \eta).$$
(iii) If $\eta$ is  discrete and $\chi$ is unitary, then $\xi_\chi $ is tempered. The notion of weak constant term of tempered $H$-forms  is  introduced  and computed for $\xi_\chi$.\\
 (iv) The $B$-matrices  preserve cuspidal (resp., discrete) $M\cap H$-forms.}
 \\{\bf Remark }\\ (a) As expected, when  $\eta$ is cuspidal and $P'$  is too small,  $\xi_{P^{'-}}$ vanishes, by (ii).\\ (b) The analogue of part (iii) for real groups is the main result of  [D1] which plays an important role in the proof of the Plancherel formula (cf [D2], [D3]). The long proof used the description of relative discrete series by Oshima and Matsuki.  Our proof here does not need   such a knowledge. It would be interesting if one could find  a  proof of this result of [D1] avoiding the description of the relative discrete series.  \\ 
 
Let us describe   the two main key lemmas of the article and let us explain how  they lead to  Theorem (i) a)  of the theorem above.\\ {\bf First key lemma} (cf. Lemma \ref{xiPH})\\{\em 
Let $P$ be a parabolic,  not necessarily  a $\si$-parabolic, subgroup of $G$. Let $A_0$ be a $\si$-stable maximal split  torus of $P$, which exists,  and 
let $M$  be a Levi subgroup of $P$ with $A_0\subset M$. Let $\delta$ be a smooth representation of $M$ of finite length. 

Let $(\xi_\chi  ) $  be a smooth family of $H$-forms  on $i^G_P( \delta_\chi)$, with $\chi$ in a neighborhood of $1$ in a  complex subtorus $X$ of  $X(M)$. 
\\Let us assume that  $PH$ is contained and  open   in the support of the family $(\xi_\chi)$.

Then the elements of $X$ are anti-invariant under $\si$. 
Moreover, if  there exists $\chi$ strictly $P$-dominant in $X$, then $P$ is a $\si$-parabolic subgroup of $G$.}

Applying this  to the family $ ((\xi_\chi)_{P^{'-},w})$ (with $G$ replaced by $M'$, and $P$ by a suitable conjugate in $M'$ of $M'\cap w.P$), one sees why, in Theorem 2 (i) a), one can take $M'\cap w. M$ to be a  $\si$-parabolic subgroup of $M'$.  Then it  is possible to refine the choice of $w$ to get (i) a). 

 Let us explain how one gets the assertion  (a) of the above remark and let us assume now that $\eta$ is cuspidal.  If $M\cap w^{-1} .P'$ is different from $M$, it is not possible to see directly that $(\xi_\chi)_{P^{ '-},w}$ is zero. What is easily seen in that case, from the cuspidality of $\eta$,    is that 
  $\check{j}_{M\cap w^{-1}. P'} \circ \xi_\chi$ vanishes on the open $(\tilde{P}_w, H)_\chi$-double cosets, hence its support has an empty interior.  
  
   The second key lemma (cf. Lemma \ref{skey}) describes the effect of intertwining integrals on the support of certain  $H$-forms. It was suggested by  a Lemma of Matsuki (cf. [M], Lemma 3) on orbit closures of orbits of parabolic subgroups on a real reductive symmetric space. Then one sees, using the second key lemma (see the proof of Theorem \ref{2} (iii)), that for $\chi$ in an open dense subset of $X(M)_\si$,  the element 
 $$^tA(w.  \tilde{P}_w,  \tilde{P}'_w, w j_{M\cap w^{-1}.P'} \delta) \circ  \lambda(w)\circ \check{j} _{M\cap w^{-1}.P^{'-}} \circ  \xi_\chi $$ of $(i^G_{\tilde{P}'_w}
 (w j_{M\cap w^{-1}.P'} \delta_\chi))'$ has a support with an empty interior. 
 
Applying $\cc{r}_{M '}$ to something which vanishes  on all open $(\tilde{P'_w} , H)$-double cosets, you get something which vanishes on  all open $(M'\cap w.P, M'\cap H)$-double cosets.
 The last ingredient  to prove the assertion (a) of the above remark is  that for $\chi$-generic, an $H$-form on $i^G_PE_\chi$ is determined by its restriction to open $(P,H)$-double cosets (cf. [BD]). This is applied to $M'$ instead of $G$. 
 
 The computation of $( \xi_\chi)_{P^{'-},w}$, when $M\cap w^{-1} .P'=M$   follows  from Theorem 1. The last statement, which says that $B$-matrices preserve cuspidal $M\cap H$-forms, comes from the hereditary properties of $j_{P^{'-}}$ and the fact observed in the remark, that for $P'$ small enough
 $j_{P^{'-} }\xi =0$.
 
 Theorem 2, together with the determination of  part of the Casselman pairing for parabolically  induced representations (cf. Lemma \ref{partcass}), leads  to the determination of the constant term and the weak constant term of Eisenstein integrals for $p$-adic reductive symmetric spaces  in terms  of the  corresponding $C$-functions (cf. Theorem \ref{cteis}).
 
 The role of this type of results might be seen in the $p$-adic case in [Wal] for the work of Harish-Chandra on the Plancherel formula for $p$-adic groups, and [D4], [D5] for  Whittaker functions.
 
  \section{Notation}
\setcounter{equation}{0}
\subsection{Reductive $p$-adic groups}
If $E$ is a vector space, $E'$ will denote its dual. If $E$ is real, 
$E_\C$ will denote its complexification.

 If $G$ is a group,
$g\in G$ and  $X$ is a subset of $G$, {$g.X$} \index{$g.X$}   will denote $gXg^{-1}$. If $J$ is a subgroup of $G$, $g\in G$
and $(\pi, V)$ is a representation of $J$, $V^J$ will denote the space of
invariant elements of $V$ under $J$ and $(g\pi, gV)$ will denote the representation of $g.J$ on $gV:=V$ defined by: $$(g\pi)(gxg^{-1}): = \pi(x), x\in J.$$ We will denote  the dual representation of
 a representation $(\pi, V)$ of  $G$ in the algebraic dual vector space $V'$ of $V$ by $(\pi', V')$ \index{$(\pi',V')$}. If $V$ is a vector space of vector valued functions  on $G$ which is invariant under right (resp., left ) translations, we will denote  the right  regular representation of $G$ in $V$ by $\rho$ (resp., $\ll$).
 
  If $G$ is locally compact, $d_lg$ will denote a left invariant Haar measure on $G$ and $\delta_G$ \index{$\delta_G$, modulus function} will denote the modulus function. 
  
   Let  $\bold{F}$ \index{$\bold{F}$, non archimedean local field}
be a non archimedean local field. Unless specified we assume: 
\ber \label{22}  The characteristic of $\F$ is different from 2. \eer
Let $\vert.\vert _\bold{F}$ be the  absolute value  of $\bold{F}$.

 One considers various algebraic groups defined over $\bold{F}$,  and a sentence
like:  
 \setcounter{equation}{0}
\ber \label{corse}`` let $A$ be a split torus 
 "  will mean ``let 
$A$ be the  group of $\bold{F}$-points   of a torus,  $\underline A$,   defined and split  over $\bold{F}$ ".\eer 
With these  conventions, let $G$ be a connected reductive linear algebraic
group.

Let $A$ be a split  torus of $G$. Let  $X_{*}(A)$ \index{$X_{*}(A)$, group  of one-parameter subgroups of  
 $A$} be the group  of one-parameter subgroups of  
 $A$. This is a free abelian group of finite type. Such a group will be
called a lattice. One fixes a uniformizing element
 $\varpi$ \index{$\varpi$} of  $\bold{F}$. \ber \label{LLA} One denotes by  $\Lambda (A)$ \index{$\Lambda(A)$} the image of  $X_{*}(A)$ in 
$A$ under   the morphism of groups  $\underline \lambda \mapsto \underline
\lambda (\varpi)$.\eer  
Under  this morphism $\LL(A)$ is isomorphic to   $X_{*}(A)$. 

If $(P, P^-)$ are two opposite parabolic subgroups of $G$, we will
 denote  their common Levi subgroup by $M$
and  the maximal split torus of its center by $A_M$ \index{$A_M$} or $A$.   We denote the   unipotent radical of $P$ (resp., $P^-$) by $U$
(resp., $U^-$) .   

Let   $A^-$ (or  $A^{-}_P$) \index{$A^-$, closed negative Weyl chamber} be the set  of $P$-antidominant elements in $A$. More precisely, if  $\Sigma(P, A)$ \index{$\Sigma(P, A)$} is  the
set of roots of $A$ in the Lie algebra of $P$, and  $\Delta (P,A)$  \index{$\Delta (P,A)$} is the
set of simple roots, one has:
$$A^-\>= \{a \in  A \vert \vert \alpha(a )\vert _{\bold{F}} \leq 1,\>\> 
 \>\> \alpha \in \Delta (P,A)\}.  $$
Also one defines  for $\ep>0$: 
.$$ A^-(\ep)= \{a \in  A \vert \vert \alpha(a)\vert _{\bold{F}}  \leq \ep , \>\> \alpha
\in \Delta (P,A) \index{$A^-(\ep)$}
\}.$$


If  $J$ is an algebraic group, one denotes by $\mathrm{Rat} (J)$  \index{$\mathrm{Rat} (J)$, group of rational characters}
the  group of its rational characters defined over $F$.   Let us define:
\beq \label{ag}\a_{G}= \mathrm{Hom}_{\Z} (\mathrm{Rat} (G),\R).  \index{$\a_{G}$ }\eeq  The restriction of rational characters 
 from  $G$ to  $A_{G}$  induces an  isomorphism: 
\beq \label{iso} \mathrm{Rat}(G)\otimes _{\Z}\R \simeq \mathrm{Rat}(A_{G})\otimes _{\Z} \R .\eeq
Notice that $\mathrm{Rat}(A_G)$ appears as a  generating lattice in the dual space $\a'_G$ of $\a_G$ and :\beq \label{aotimes} \a_G' \simeq \mathrm{Rat}(G) \otimes _\Z \R. \eeq   
One has the canonical map  $H_{G}: G \rightarrow \a_{G}$ \index{$H_G$}
 which is defined by: 
\beq \label{H} e^{\langle H_{G}(x), \chi\rangle}= \vert \chi (x)\vert_{\bold{F}}, \> x\in G,
\chi
\in \mathrm{Rat} (G).\eeq The kernel of  
$H_{G}$, which is denoted by $G^1$,   is the intersection of the kernels of the characters of
$G$,  $\vert \chi \vert_{\bold{F}}$, $\chi \in \mathrm{Rat}
(G)$. One defines
\beq \label{x(g)} X(G)= \mathrm{Hom} (G/G^{1}, {\C}^*),\index{$X(G)$} \eeq  which is the group  of unramified characters of $G$.  One will
use similar notations for Levi subgroups of $G$. One denotes by  $\a_{G, \F}$, resp.,  ${\tilde
\a}_{G, \F}$  the image of  $G$, resp.,  $A_{G}$, under  $H_{G}$. Then 
$G/G^{1}$ is  isomorphic to the lattice 
$\a_{G,\F}$.

There is a surjective map: 
\beq  \label{surjection}(\a'_{G})_{\C}\rightarrow X(G)\rightarrow
1 .\eeq   denoted by $ \nu \mapsto \chi_\nu$ which associates to 
$\chi\otimes s$, with $\chi  \in \mathrm{Rat} (G) $, $s\in \C$,    the character  $g\mapsto \vert \chi (g)\vert_\F 
^{s}$ (cf. [Wal], I.1.(1)).  In other words: 
\beq  \label{cnu}\chi_\nu (g) =   e^{ \langle \nu, H_G (g)\rangle}, g \in G, \nu \in(\a'_{G})_{\C} \index{$\chi_\nu $ } .\eeq
 The kernel is a lattice and it defines a structure of a
complex algebraic variety on  
$X(G)$ of dimension   
$\mathrm{dim}_{\R}\a_{G}$.  Moreover $X(G)$ is an abelian  complex Lie group whose Lie algebra is equal to $(\a'_{G})_{\C}$. 

 If  $\chi$ is an element of $ X(G)$, let  $\nu$ be an element of $
\a_{G,\C}'$ such that $\chi_\nu= \chi $.  The real part  $\mathrm{Re}(\nu) \in \a_{G}'$ is independent from the choice of $\nu$.  We will denote it  by $\mathrm{Re}(\chi)$. If  $\chi \in \mathrm{Hom} (G,
\C^{*}) $ is continuous, the character $\vert \chi \vert$  of $G$ belongs to  $X(G)$.  One sets $\mathrm{Re}( \chi)=  \mathrm{Re}( \vert \chi\vert )$. Similarly,  if  $\chi \in \mathrm{Hom}(A_{G},
\C^{*})$ is continuous, the character $\vert \chi \vert$ of $A_G$ extends uniquely  to  an element of  $X(G)$  with values in  $\R^{*+}$, that we will denote again by 
  $\vert
\chi \vert$ and one sets  $\mathrm{Re}( \chi)=  \mathrm{Re}( \vert \chi
\vert)$.

 If $P$ is a parabolic subgroup of $G$ with Levi subgroup $M$
we set:
$$\a_{P}= \a_{M}, H_{P}=H_{M}.$$   
 The inclusions
$A_{G}\subset A_{M}\subset M\subset G$  determine a surjective  morphism    $\a_{M, \F}\rightarrow  \a_{G, \F}$, resp., an injective  morphisme 
 ${\tilde
\a}_{G, \F}   \rightarrow {\tilde
\a}_{M, \F}$, which extends uniquely to a surjective linear map between  $\a_{M}$ and  $\a_{G}$, resp., injective map  between
$\a_{G}$ and 
$\a_{M}$.
The second map allows us to identify  $\a_{G}$ with a subspace of 
$\a_{M}$ and the kernel of the first one, $\a^{G}_{M}$,  satisfies: 
\ber $$\label{oplus} \a_{M}= \a^{G}_{M}\oplus \a_{G}. \index{$\a^{G}_{M}$} $$
\eer

Let us denote  the set of restrictions to $M$ of elements of $X(G)$ by $X(G\vert M)$ \index{$X(G\vert M)$}. Then $X(G\vert M)$ is the analytic subgroup of $X(M) $ with Lie algebra $(\a'_G)_\C \subset (\a'_M)_\C$. This follows easily from  (\ref{cnu}) and (\ref{oplus}). Moreover:
\ber \label{Xclosed}  $X(G\vert M)$ is a closed subgroup  of $X(M) $. \eer This will  be seen by writing $X(G)= X(G)_ u X(G)^+$ where $X(G)_u$ is the compact group of unitary characters in $X(G)$ and $X(G)^+$ is the group  of elements in $X(G)$ which are strictly positive. The  group $X(G)_u$ is compact and has compact image. The group $X(G)^+$ is isomorphic to a vector subgroup and the restriction, which is a morphism of Lie groups,  determines an isomorphism  of $X(G)^+$ to a connected Lie subgroup of $X(M)^+$: the restriction of an element of $X(G)^+$ trivial on $A_G$ is trivial by what has been said above.   Hence the image of $X(G)^+$  is closed. This implies our claim on $X(G\vert M)$.

One has (cf.  [D3], (4.5)), \ber \label{45} The map $\LL(A_G )\to G/G^1$ is injective and allows us to identify $\Lambda(A_G)$ with  a subset of $\a_G$. \eer  

Let $A$ be a maximal split torus of $G$ and let $M_0$ be the centralizer of $A$ in $G$. We fix a $W$-invariant scalar product on $ \a:= \a_{M_0}$, where $W$ is the Weyl group of $(G, A)$.  Then $\a_G $ is identified with the fixed point set of $\a$ under $W$ and $\a^G$ is an invariant subspace of $\a$ under $W$ which is supplementary to $\a_G$. Hence it is the orthogonal subspace to $\a_G$ in $\a$. The space $ \a'_G $,  might be viewed as  a subspace of $\a'$ by (\ref{oplus}). More generally let $M$ be  a Levi subgroup of $G$ which contains $A$. From (\ref{oplus}) applied to $M$ instead of $G$ and to  $M_0$, instead of $M$, one gets a decomposition $\a= \a^M_{M_0} \oplus \a_M$. From the $W$-invariance of the scalar product one gets:  \ber \label{iden} The decomposition $\a= \a^M_{M_0} \oplus \a_M$ is  an orthogonal decomposition.\\ The space $\a'_M$ appears as a subspace of $\a'$ and, in the identification of $\a$ with $\a'$ given by the scalar product, $ \a_M '$ is identified with $\a_M$. \eer  

 If $\nu\in \a'$, the parabolic subgroup of $G$ whose Lie algebra is equal to the sum of the $A$-weights spaces for  the weights $\aa$ which are  either equal to zero or to a root $\aa$ such that 
$ (\nu, \aa)\geq 0$ is denoted  by $P_\nu$\index{$P_\nu$}. 
Then one has:  \ber \label{p-l} The parabolic subgroups of $G$, $P_\nu$ and $P_{-\nu} $ are opposite. \eer 
If $\rho_P\in \a'$ is the half sum of the $A$-roots of $A$ in the Lie algebra of $P$, then the following is clear:
\ber \label{pl} $P=P_{\rho_P}.$\eer

Let ${\underline G}$ be the algebraic group defined over $\F$ whose group of $\F$-points is $G$. Let $\si$ be a rational involution of ${\underline G}$, defined over
$\F$.  Let 
$H$ be  the group of $\F$-points of an open $\F$-subgroup  of the fixed
point set of $\si$.  We will also denote   the restriction of $\si$ to
$G$ by  $\sigma$.\ber \label{asplit} A split torus of $G$, $A$, is said to be $\si$-split \index{$\si$-split torus} if $A$ is contained in  the set of elements of $G$ which are antiinvariant under $\si$. \eer Notice that  our terminology differs from other authors ([HW], [HH], [KT1]).  
Let $A$ be a $\si$-invariant maximal split torus of $G$.  We fix a scalar product on $\a$ which is invariant under $\si$ and the action of the Weyl group of $G$ with respect to $A$: this is possible because $\sigma$ normalizes $A$, hence its normalizer  and the Weyl group of $G$ with respect to $A$. \berÊ\label{asiasi} Let $A^\si$ (resp., $A_\si$)  \index{$A^\si, A_\si$} 
be the  maximal split torus in  the group of elements of $A$ which are  invariant (resp., antiinvariant) under $\si$.  Then $\a^\si$ (resp., $\a_\si$) is identified   with the set $\a^\si$ (resp.,  $\a^{-\si} $) \index{$\a\si, \a^{-\si}$}
  of invariant  of (resp., antiinvariant) elements of $\a$ under  $\si $ and $A_\si$ is the maximal $\si$-split torus of $A$. \eer 
\ber \label{am'} If $M$ is a $\si$-invariant Levi subgroup of $G$ which contains $A$,  $\a_M$ is a $ \si$-invariant subspace of $\a$ and $\a_M= \a_M^\si\oplus \a_M^{-\si} $ where $\a_M^{\si}= \a_M\cap \a^\si, \a_M^{-\si} = \a_M\cap \a^{-\si}$ . This is an orthogonal decomposition and, in the identification of $\a$ to $\a'$,   $(\a_M)^{-\si }$ is identified with the space $(\a'_M)^{-\si}$ of $\si$-antiinvariant elements of $\a'_M$.  \\
Moreover the Lie algebra of the connected component, $X(M)_\si$, \index{$X(M)_\si$} of the  group of antiinvariant elements of $X(M)$,  with the identification of $\a_M$ and $\a'_M$, is equal to $ \a_M^{-\si}$.\eer
\ber \label{sipara} A  parabolic subgroup of $G$, $P$,  is called a $\sigma$-parabolic subgroup \index{$\sigma$-parabolic subgroup} if  $P$ and
$\sigma(P)$ are opposite  parabolic subgroups. Then $M:=P\cap \sigma(P)$ is the  $\sigma$-stable
Levi subgroup of $P$.    If $P$ is such a parabolic subgroup, $P^-$ will denote $\si(P)$. Then the maximal split torus of the center of $M$, $A_M$ is $\si$-stable. \eer 

The sentence :
 ``Let $P=MU$ be a parabolic subgroup of $G$'' will mean that $U$ is the unipotent radical of $P$ and $M$ is a Levi subgroup of $G$. If moreover $P$ is a $\si$-parabolic subgroup of $G$, $M$ will denote its $\si$-stable Levi subgroup.
 
 One deduces from [HW], Proposition 13.4:
\ber \label{op} $PH$ and $P^-H$ are open in $G$.  Let $M$ be equal to $P\cap \sigma(P)$. The restriction to $H\cap M$ of the modulus function of $P$, $\delta_P$, is trivial as  it is positive and  it is equal to its inverse on $M\cap H$. \eer We also recall  that (cf. l.c. Corollary 6.16):
\ber \label{sip} There are only a finite number of $H$-conjugacy classes of parabolic subgroups of 
$G$. \eer 

Let $M_\vid$ be the $\si$-stable Levi subgroup of a minimal $\si$-parabolic subgroup of $G$. Let $A_\vid$ be the maximal $\si$-split torus of the center of $M_\vid$. 
\begin{defi}  An element $x$ of $G$  is said to be $A_\vid $-good \index{$A_\vid$-good} if and only if  $x^{-1}. A_\vid $ is a $\si$-split torus.
\end{defi} From [BD], Lemma 2.4 one sees: 
\ber \label{WG}There exists a finite  set of representatives of the $( P_\vid, H) $-double cosets open in $G$,  $\mathcal{W}^G_{M_\vid}$,  whose elements are $A_\vid$-good.\\
Moreover if $M$ is  the $\si$-stable Levi subgroup of a $\si$-parabolic subgroup of $G$, with $M_\vid$ (or $A_\vid$)  contained in $M$, there exists  a subset, $\mathcal{W}^G_{M} $,  of $\mathcal{W}^G_{M_\vid}$  \index{$\mathcal{W}^G_{M} $, $\mathcal{W}^G_{M_\vid}$} such that for all   $\si$-parabolic subgroup $P$ of $G$ with Levi subgroup $M$,  $\mathcal{W}^G_{M} $ is a set of representatives of the $(P, H)$-double open cosets.  \eer
\begin{lem} \label{sixP}
Let the notation be  as above. Let $P=MU$ be a $\si$-parabolic subgroup of $G$ with $M_\vid \subset M$ and let $x\in G$ be $A_\vid$-good. Then one has:\\
(i) The group $x^{-1}.P$ is a $\si$-parabolic subgroup of $G$ with $\si$-stable Levi subgroup $x^{-1}. M$. Moreover 
$\si(x^{-1}.P)= x^{-1}.P^-$, where $P^-=\si(P)$. \\(ii) The set   $PxH$ is open in $G$.\\(iii) One defines an involution of $G$, $\si_x$, by:
$$ \si_x(g):=x \si(x^{-1}gx) x^{-1}, g\in G$$
whose fixed point set contains $x.H$. Then $P$ is a  $\si_x$-parabolic subgroup of $G$,  $\si_x (M)=M$  and $\si_x(P)=P^-$.\\
(iv) For all $y \in PxH$, $y^{-1}.P$ is a $\si$-parabolic subgroup of $G$ and $P$ is a $\si_y$-parabolic subgroup of $G$. Let $M_y$ be the $\si_y $-stable Levi subgroup of $P$. 
Then $M_y \cap y.H= P\cap y.H$. Moreover  if $y=px$ with $p\in P$,  one has $ M_y = p.M$.
 \end{lem}
 \begin{proof}
(i) and (ii) follows from [BD], Lemme 2.4.\\
(iii) follows immediately from (i).\\
(iv) If $ y= pxh$ with $h\in H, p\in P$. Then $y^{-1}.P= h^{-1}. (x^{-1}.P)$. Then the first part of (iv) follows from (i). Also one has  $\si_y= \si_{px}$. Hence  one is reduced to prove the second part of (iv) for $h=1$ and $y=px$.
A simple computation shows that $\si_y(P)= p.P^-$. Hence $ \si_y (P) \cap P= p.M$, which proves (iv). \end{proof} 
  
Let $A_\vid$ be a maximal $\si$-split torus of $G$ and let  $A_0$ be a $\si$-stable maximal split torus of $G$ which contains $A_\vid$. Let $P_\vid=M_\vid U_\vid $ be a minimal $\si$-parabolic subgroup of $G$, whose $\si$-stable Levi subgroup is equal to the centralizer, $M_\vid $, of $A_\vid$ in $G$.\\

\begin{prop}  \label{KT}
 There exists  a decreasing  sequence of $\si$-stable compact open subgroups of $G$, $(J_n)_{n\in \N^*}$, 
which forms a basis of neighborhoods of $1$ in $G$ and such that for each $ n \geq 1$, $J:= J_n$ satisfies:  \\
(i)  For every $\si$-parabolic subgroup of $G$, $P=MU$ which contains $P_\vid$, the product map $ J_{U^-}\times J_M \times J_U \to J$ is bijective, where  $J_{U^-} = J \cap U^-$, $J_M=
J\cap M$,  $J_U= J \cap U. $\\
(ii) Let $A\subset A_\vid$ be the maximal $\si$-split torus of the center of $M$ and let $A^-$ be the set of its $P$-antidominant elements.  For all $a$ belonging to  $A^- $, one has $$aJ_U a^{-1}  \subset J_{U},  a^{-1}J_{U^-} a \subset J_{U^-}.$$
(iii) One has $J= J _H J_P$, where $J_H= J \cap H, J_P= J \cap P$. \\(iv) For each $\si$-parabolic subgroup of $G$ which contains $P_\vid$, $P=MU$, the sequence $(J_{n }\cap M)$ enjoys the same properties that $(J_n)$ for $M$, $P_\vid\cap M$. \end{prop}
\begin{rem}If $\F$ is  the characteristic of the residue field of $\F$ is different from 2, the proposition  is due to Katano and Takano (cf. [KT] Lemma 4.3). In that case, their result is stronger, as their $J_n$ satisfy more properties.  
\end{rem}
\begin{proof} We use the terminology and notations  of Section \ref{Lie} (cf. also Lemma \ref{LieG}). 
We fix a basis of the Lie algebra of $G$, $\un{g}$, $(X_j)$, which is the union of a basis $(U_k)$ of $\un{u_\vid}$ made of weight vectors for $A_\vid$, of  a basis of $\un{m_\vid} \cap \un{ h}$,  a basis of the space of $\si$-antiinvariant elements  of $\un{m_\vid}$ and of  the basis $(\si(U_k))$ of $\si(\un{u_\vid})$. We use Lemma \ref{DB} (iv) for $G_1= U, G_2= M, G_3 = \si(U)$ to prove (i).\\ Let $A^0$ be the maximal compact open subgroup of $A$. We remark that the  $A^0$-module $A^-$ is finitely generated. Let $(\ll_l)$ be a finite family of generators. We apply Lemma \ref{DB} (iii) to the family of automorphisms of $G$ induced by the conjugation  by elements of $A^0 \ll_l$. The fact that the  $J_n$ can be choosen $\si$-stable is proven similarlly.
\\(iii) We apply Lemma \ref{DB} to $G_1= U$, $G_2= M$, $G_3=H$. Here $\un{g}'_1= \un{u}$, $\un{g}'_2= \un{m}$ and $\un{g}'_3$ is the subspace generated by the $U_i+ \si(U_i)$, where the $U_i$  are those which belong to $\un{u}$.\\
(iv) results from the fact that $J_M$ is defined like $J$, in view of Lemma \ref{DB} (iv). 
\end{proof}
  \begin{rem}
In [D3], Remark 3.2,   it was stated  incorrectly, although not used,  that in (ii) one could replace $A$ by the maximal split torus of the center of $M$.             
\end{rem}
\ber \label{phgood} If a compact open subgroup of $G$ satisfies the properties  of the Proposition \ref{KT} , it will be said to be to have a $\si$-factorization\index{$\si$-factorization}.
These are the $(P,H)$-good subgroups from the introduction. \eer

  \ber If $(P=MU, P^-=MU^-)$ is a pair of opposite parabolic subgroups of $G$, we will say that a compact open subgroup  $J$ of $G$ has an Iwahori factorization with respect to $(P,P^-)$ if the product map $ J_{U^-}\times J_M \times J_U \to J$ is bijective, where  $J_{U^-} = J \cap U^-$, $J_M=
J\cap M$,  $J_U= J \cap U. $\eer

 \section{Two operations on $H$- forms on induced representations}  \subsection{Second Adjointness Theorem and $H$-forms}
  In the sequel, the smooth representations of $G$ and of its closed subgroups will be with complex coefficients.
  
   Let $(\pi, V)$ be a smooth representation of $G$ and let $P=MU$ be  a parabolic
subgroup of $G$. One
denotes by
  $(\pi_P,V_P)$ the tensor product of the   quotient of $V$ by the
$M$-submodule generated by the $\pi(u)v-v$, $u\in U, v\in V$,  with the representation of $M$ on $\C$ given by $\delta_P^{-1/2}$.  We call it the
normalized  Jacquet module of $V$ along $P$.  We denote the natural projection map from $V$ to $V_P$ by
$j_P$  and sometimes $ \pi_P$ will be denoted $ j_P(\pi)$.
For further reference, we state  the following fact.
\ber \label{xexq} Let   $P=MU$ be  a parabolic subgroup of $G$, let $Q=LV$ a parabolic subgroup of $M$ and  let  $(\delta, E)$ be  a smooth representation of $M$. 
If $x$ is an element of $G$, one has the equality of vector spaces $(xE)_{x.Q}= E_Q $ and the natural representation of $x.Q$ on $(xE)_{x.Q} $ is simply $x. \delta_Q$.  \eer 

 The following result is due  to J. Bernstein for smooth representations (cf.  [Be],  [R], Chapter V.9 , see also [Bu] for the first published proof). We present here a slight reformulation of his result (cf.  [D3], Lemma 2.1). This is a
generalization of a result of W. Casselman for admissible representations (cf. [C]). Let $(P, P^-) $ be a pair of opposite parabolic subgroups of $G$ with common Levi subgroup $M$. Let  $ A_0$ be a maximal split torus of $M$  and let  $P_0$ be a minimal parabolic subgroup $P_0$ such that $A_0\subset P_0\subset P$. We define:  \ber \label{apep} $$\Theta_P:= \Delta(P_0\cap M, A_0)$$
and, for $\ep  > 0$, we set:
 $$A_0^-(P, <\ep):=\{a \in A_0^-\vert \vert \alpha(a)\vert _\F< \ep, \alpha \in \Delta (P_0,A_0) \setminus
\Theta_P\} \index{$A_0^-(P, <\ep)$}. $$
\eer
{\bf Second Adjointness theorem} 
\ber  \label{becas} 
Let $(\pi, V)$ be  a smooth representation of  $G$.  Let $j_P$ (resp., $j_{P^-}$)  denote  the canonical
projection of $V$ (resp., of the smooth dual ${\check  V}$ of $V$) onto $V_P $ (resp.,
$({\check V})_{P^-}$). 
\\ Then there exists a unique non degenerate  $M$-invariant bilinear form $\langle.,. \rangle  _P$
\index{$V_P, j_P, \langle.,. \rangle  _P$} on  $({\check
V})_{P^-} \times V_P$  such that for all compact
open subgroups,
$J$, of $G$,
 there exists $\ep_J< 1$ such that:
$${\delta_P^{1/2}}(a)\langle j_{P^-}( { \check v}), \pi_P(a) j_P(v) \rangle _P=
\langle {\check v}, \pi(a) v  \rangle, $$ for $a \in A_0^{-}(P,<\ep_J), v\in V^J, {\check
v\in {\check V}^J}.$
It is part of  the statement  that $\varepsilon_J$ does not depend on $V$.

In particular there is a canonical isomorphism between $(V_P)\,\cc{}$,  $({\check V})_{P^-}$.
\eer
\ber \label{hform} An   $H$-form  \index{$H$-form} on a smooth module of $G$ is an $H$-fixed linear form on $V$. \eer
One denotes 
by $e_J$ the normalized Haar measure on $J$ \index{$e_J$, normalized Haar measure on $J$}that we view
as  a compactly supported distribution on $G$. Using the same argument as in [L], Lemma 2, one sees:  
\ber  \label{K'}
Let $P=MU $ be a $\si$-parabolic subgroup of $G$.  Let $A$ be the maximal split torus of the center of $M$. Let $A_\si$ be the  maximal $\si$-split torus of the maximal split torus, $A$, of the center of $M$.   
Let $ J$ be a compact open subgroup of $G$ with  a $\si$-factorization with respect to $(P,P^-)$. 
Then 
 for every  smooth module $(\pi,V)$, $\xi\in V^{'H}$:
 $$\langle\xi, \pi(a) v \rangle = \langle e_{J}\xi,\pi(a)  v \rangle  ,
\>\> v\in V^J, \>\> a\in A^-\cap A_{\si} ,
$$
 where $e_{J }\xi $ is the element of $\VV$  defined by:
 $$\langle e_{J }\xi, v \rangle = \langle \xi,\pi( e_{ J }) v \rangle , \>\>
v\in V.$$
  \eer 
  From our hypothesis on $J$, one sees that one can take $J'=J$ in the proof of Lemma 2 of [L].
  
  The following result has been proved independently on one hand  by Kato and Takano  ([KT1], Proposition 6)  and on the other hand by Lagier ([L] Theorem 1 (ii)) for admissible modules. Later, it was remarked (cf. [D3]) that it works  for general smooth modules.  
Let $A_M$, or simply $A$,  be the maximal split torus of the center of $M$.
 For every smooth module $(\pi, V)$ of $G$ and  $\xi$ an $H$-form on $V$,  there exists a unique  $M\cap H$-form  $j_{P^-} \xi$ \index{$j_{P^-} \xi$} on $V_{P}$
 such that   for each compact open subgroup    $J$ of $G$, there exists $\varepsilon'_J > 0$, such that $\varepsilon'_J\leq \varepsilon_J$, depending only on $J$ and not on $V$ and $\xi$,  such that one has:
\beq  \label{xia}\langle \xi, \pi(a)v \rangle =\dd (a) \langle j_{P^-} \xi,  \pi_P(a)
j_P(v) \rangle  ,
\>\>v  \in V^J,\>\> 
a \in A^-(\ep'_J).  \eeq
From (\ref{K'}), one deduces from the above that if $J $ is a compact open subgroup of $G$ with a $\si$-factorization, one has 
$$\langle e_J \xi, \pi(a)v \rangle =\dd (a) \langle j_{P^-} \xi,  \pi_P(a)
j_P(v) \rangle_P ,
\>\>v  \in V^J \>\> , a \in A_\si \cap A^-( \varepsilon'_J).$$
From the Second Adjointness Theorem, one deduces from this, that for $a  \in A_\si \cap A^-( \varepsilon'_J), v\in V^J$ :
$$
 \langle j_{P^-} \xi,  \pi_P(a)
j_P(v) \rangle  = \langle j_{P^-}e_J \xi,  \pi_P(a)
j_P(v) \rangle  .$$
If $\pi$ is admissible, the  two sides of this equality are $A_\si$-finite functions on $A_\si$, hence they are equal. In particular in $a=1$, one gets:
\ber  \label{defj} Let $J$  be a compact open subgroup of $G$ with  a $\si$-factorization  with respect to $(P,P^-)$. Then, if $\pi$ is admissible,  one has:
  $$\langle j_{P^-} \xi, v  \rangle = \langle j_{P^-} ( e_{J} \xi) , v \rangle_P,  v\in V_{P}^{J_M}.$$ \eer 
\subsection{Induced representations} \ber \label{AK} Let $A$ be a maximal split torus of $G$. 
Let $K$ be a maximal compact open  subgroup of   $G$ which is the stabilizer of a special point of the  apartment  of  the extended building of $G$. It fixes a choice of  a left invariant measure measures on each  algebraic closed subgroup of $G$ such that the measure  of its intersection with $K$ is equal to $1$.
It depends on $A$ and $K$.\eer
Let $P=MU$ be a parabolic subgroup of $G$ with Levi subgroup $M$ containing $A$  and with unipotent radical $U$. Let $\chi$ be an element of $ X(M)$. One denotes by  $E_\chi$ the space of the representation  $\delta_\chi:=\delta\otimes \chi$. Let 
$i^G_{P}E_\chi$ \index{$i^G_P$} be the space of maps  $v$ from  $G$ to $E$, right  invariant under a compact open subgroup of $G$ and 
such that  $v(mug)= \delta_P(m)^{1/2}\delta_\chi (m) f(g)$  for all 
$m\in M$, $u\in U$, $g\in G$. Let  $i^G_{P}\delta_\chi$  be the representation of $G$ in 
 $i^G_{P}E_\chi$ by right translations.
 
  If  $(\delta, E)$ is a smooth representation of  $M$, one extends    it to a representation of  $P$ trivial on  $U$, denoted in the same way. 
  One denotes by  $i^K_{P\cap K}E$  the space of maps 
$v$ from $K$ in  $E$, which are right invariant under a compact open subgroup of $K$ and 
such that  $v(pk)= \delta(p)v(k)$ for all $k\in K$ and  $p\in P\cap K$. The restriction of functions to $K$
determines an  isomorphism from  $i^G_{P}E_\chi$ to  $i^K_{P\cap K}E$. We will denote  the representation of  $G$  in $i^K_{P\cap K}E$ deduced from 
$i^G_{P}\delta _\chi$ ``par transport de structure'' by  ${\overline
i}^G_{P}(\delta_ \chi)$. This representation will be called the compact realization   of $ i^G_P\delta_\chi$ in this space independent from $\chi$.
If  $v \in i^K_{P\cap K}E$, one denotes by 
$v_\chi$ the element of   $i^G_{P}E_\chi$  whose  restriction to  $K$ is equal to $v$.    If  $\delta$ is unitary,  one defines a scalar product on $i^K_{P\cap K}E$  by:
  \beq (v, v')= \int_{K}(v(k),v'(k))\> dk, \> v, v' \in i^K_{P\cap K}E .\eeq
   The representation $\overline{ i}^G_P \delta_\chi$ is unitary for this scalar product when   $\chi$  is unitary. Consequently, ``par transport de structure'', $i^G_P\delta_\chi$ is also unitary.\\
If $g \in G$, one chooses $u_P(g) \in U, m_P(g) \in M$ and $k_P(g)\in K$ such that $g= u_P(g)m_P(g)k_P(g)$. Then $\delta_P( m_P(g))$ does not depend on the choice of $m_P(g)$. 
Let $ P^-=MU^-$ be the opposite parabolic subgroup of $P$ with respect to $M$. 
\ber \label{ccigp} We will identify $i^G_P \cc{E}$ to $(i^G_P E)\cc{}$, \index{isomorphism between $i^G_P \cc{E}$  and $(i^G_P E)\cc{}$}by associating to $\cc{v}  \in i^G_P \cc{E} $  the linear form, $\phi$ on $i^G_PE$ defined by 
$$ \langle \phi, v \rangle = \int_K  \langle \cc{v}(k), v(k) \rangle dk, v \in i^G_PE .$$ \eer 

Let $e$ be an element of $E$ and let $J$ be  a compact open subgroup of $G$ such that  $e$ is invariant under $ J \cap P$ under $\delta$.
One defines a map $v _{e, \delta}^{P,J}$ \index{$v _{e, \delta}^{P,J}$} from $G$ to $E$ by:
\ber \label{vsi} $$v _{e, \delta}^{P,J} (pj)= \delta_P^{1/2}(p)\delta(p) e , j\in J, p\in P$$
$$v _{e, \delta}^{P,J}(g)= 0, g \notin PJ,$$\eer
the definition making sense due to our hypothesis on $J$ and $e$. 
 Notice that this hypothesis is satisfied if  $J$ has an Iwahori factorization with respect to $(P, P^-)$ (resp., if $P$ is a $\si$-parabolic subgroup of $G$ and $J$ has a $\si$-factorization for $(P,P^-)$) and $e$  is $J_M$-invariant.
 
 In all cases $v _{e, \delta}^{P,J} $ is invariant on the right by $J$ and defines an element of $i^G_P E$. 

\subsection{The operation $\cc{r}_M$ \index{$\cc{r}_M$}  \label{two}}
Let $P=MU$ be a parabolic subgroup of $G$ such that $A\subset M$ and $(\delta,E)$ a smooth representation of $M$.  Let $\pi=i^G_P\delta$. 

Then with our choice of Haar measures,   the left invariant Haar measure $d_lp$ on $P$ satisfies:
 \beq \label{intpk} \int _{G} f(g) dg= \int _{P\times K}  f(pk)\> d_lp\>  dk, f \in C_c^\infty (G).\eeq 
Let $d_rp = \delta_P d_l p$ which is a right invariant measure on  $P$.
One defines a linear map $M_{\delta,P} $ \index{$M_{\delta,P} $} from $ C_c^\infty (G) \otimes E$ to $i^G_P E$ 
by:  \beq \label{Mdel} (M_{\delta,P}  (f))(g) = \int_{P}\delta_P^{1/2}(p^{-1}) \delta(p^{-1})  f(pg)\> d_r p, f \in  C_c^\infty (G) \otimes E, g\in G ,\eeq
where we have identified $C_c^\infty(G) \otimes E$ with $C_c^\infty(G,E)$. 
  This map goes through the quotient to an isomorphism  between $ H_0( P, C_c^\infty (G) \otimes E))$ and 
 $i^G_P E$ ([BD], Prop. 1.13 (iv)), where $H_0$ stands for the $0$-homology. 
\begin{lem}  \label{tildexi}  We recall our choice of a left invariant measure on $P$ (cf. (\ref{intpk})). A  linear form  $\xi$ on $i^G_PE$ determines   an $E$-distribution on $G$, $\tilde{\xi}, $ which is  defined by \index{$\tilde{\xi}$}
$$ \tilde{\xi} (f)= \langle \xi , M_{\delta,P}  (f)\rangle , f \in  C_c^\infty (G) \otimes E .$$ The distribution $\tilde{\xi}$ is  $P$-covariant for the representation $\pi= \delta\otimes \delta_P^{-1/2}$ (cf. Section \ref{covd} for the definitions) \end{lem} 
\begin{proof} We denote here  the left regular representation of $G$ on $C^\infty(G)\otimes E$ by $\ll$. 
The lemma follows from the obvious equality:$$ \langle \tilde{\xi} ,\ll(p)(\pi(p)f) \rangle=\langle \xi,  M_{\delta,P}( \ll(p )(\pi(p)f)) \rangle, f\in  C_c^\infty (G) \otimes E $$ 
and  from the  equality: $$M_{\delta,P}(  \ll(p) ( \pi( p)f))= M_{\delta,P}(f), f\in  C_c^\infty (G) \otimes E, $$ that we are going to prove. 
Let  $p_0 $ be an element of $P$. Taking into account the equality  $ \delta_P (p_0)^{-1/2}$ with $ \delta_P (p_0)^{1/2} \delta_{P} ^{-1} (p_0)$, one has, for $g\in G$: 
$$ [M_{\delta,P}(  \ll(p_0) ( \pi( p_0)f))](g)= \int_P \delta_P^{1/2} (p^{-1}p_0)\delta (p^{-1}p_0 ) f(p_0^{-1}pg)\delta_P (p_0)^{-1} d_rp$$
Using the definition of $d_rp$, the change of variables $p'=p_0^{-1} p$, 
 leads to the  required identity.\end{proof} 
 The support of $E$-distributions  is given in Section \ref{covd}. 
We define the support of $ \xi$ as follows. 
\ber The support   of $\xi $ is by definition the support of $\tilde{\xi}$.  \eer As $ \tilde{\xi}$ is left $P$-covariant, one has:
 \ber \label{suppxi} $\mathrm{Supp} (\xi)$ is left $P$-invariant and is equal to  the complement of the largest  left $P$-invariant open subset of $G$, $O$, such that $\langle \xi,v \rangle=0$ if the support of $v\in i^G_PE$ is contained in $O$.\eer 
If moreover  $\xi $  is   invariant under  the right action of $H$ then the same is true for the $E$-distribution $\tilde{\xi}$.

Let us assume that $PH$ is open in $G$ and that $\xi$ is right invariant under $H$.
 The group $P\times H$ acts on $PH$ by $$(p,h)  g= pgh^{-1}, g \in PH, p\in P, h\in H .$$
 Hence  $X=PH$ is a homogeneous space under $P\times H$ homeomorphic to $X=(P\times H) / Diag (P\cap H)$ by the map $(p, h)\to ph^{-1}$(cf. [BD] Lemma 3.1 (iii)). 
One remarks that the Haar measure on $G$ restricted to $PH$ is left invariant under $P$ and right invariant under $ H$.  One applies  Lemma \ref{Tcov}  to $P\times  H$ (instead of $G$)  acting on $PH$ and  to the $\delta\otimes \delta_P^{-1/2}$-covariant distribution $\tilde{\xi}$ restricted to $PH$ . From the above, $X$ has a $P\times H$-invariant measure.  The dual representation of $\delta_P$ is $\delta_P^{-1}$. Hence there exists a unique $ H\cap P$-invariant linear form  $\eta$ on  $E$  such that one can define:
\beq \label{xig} f_ \xi(g):= (\delta_P)(p)^{1/2} \delta'(p) \eta, g=ph, p \in P , h \in H \index{$f_\xi$} \eeq 
which verifies \index{$f_\xi$}:
\ber \label{fxi} $ \langle \tilde{\xi}, f \rangle = \int _{P H}\langle f_\xi(g), f(g) \rangle \> dg$ if $f\in C_c^\infty(G)\otimes E$ has its support contained in $PH$.\eer
Let us assume that moreover $f$ has its support contained in $P(K\cap H)$. 
The set of $(p,k) \in P\times K$ such that $pk\in P(K\cap H)$ is equal to $ P\times(K \cap P)(K\cap H)$.
Thus, from (\ref{intpk}), one gets: $$  \int _{P (K\cap H)}\langle f_\xi(g), f(g) \rangle\>  dg= \int_{P\times (K \cap P)(K\cap H)}  \langle f_\xi(pk), f(pk) \rangle \> d_l p\> dk.$$ 
 Then one integrates over $P$, taking into account the covariance property of $\xi$.  Then, as $\delta_P d_lp=  d_rp$, one gets:  
 \beq \label{xivint}  \langle \xi,v \rangle = \int_{(K \cap P)(K\cap H)}  \langle f_\xi(k), v(k) \rangle \> dk , \eeq 
 where $v = M_{\delta,P} (f)$.
 
  One remarks that if $e $ fixed by $J\cap P$, using the notations of (\ref{vsi}), one has:\beq  \label{veM}v _{e, \delta}^{P,J}= \mathrm{Vol}(J\cap P)^{-1}M_{\delta,P}(1_J\otimes e), \eeq where  $\mathrm{Vol}(J\cap P)$ denotes the volume  of  $J\cap P$ for the measure  $d_lp$.
  
  Let  $x$ in $G$ such that  $PxH$ is open in $G$.  
 Applying (\ref{xig}) and (\ref{fxi})   to each $\pi '(x)f$, which  is $xHx^{-1}$-invariant,   one can define a function  on the union $\Omega$ of all $(P,H)$ open double cosets, $f_\xi$,  with values in $E'$, right invariant under $H$, left covariant under  $P$ by $\delta' \otimes \delta_P^{-1/2}$, such that:
 \ber \label{fxiomega}  If $f \in C_c^\infty (G)\otimes E $ has its support contained in $\Omega$, 
 $$ \langle \tilde{\xi},f \rangle = \int _{\Omega} \langle f_\xi(g), f(g)\rangle \> dg.$$\eer
 Moreover if $x \in G$ is such that $PxH$ is open in $G$:
 $$f_\xi(x ) \in E'^{P\cap x.H}. $$  Notice that $f_\xi$ depends on our choice of $dg$. 
 
 Let us assume moreover that $P$ is a $\si$-parabolic subgroup of $G$. Let   $J$ be  a compact open subgroup of $K$ which has  a $\si$-factorization  for $(P,P^-)$. Let $e\in E$ be fixed by $J_M$. Let us prove:
 \beq  \label{xivintbis}\langle \xi, v _{e, \delta}^{P,J} \rangle  = \mathrm{Vol} ((K\cap P)J_H) \langle \eta, e\rangle  .\eeq 
 One can apply (\ref{xivint}) with $v=v _{e, \delta}^{P,J}$ , by taking into account (\ref{veM}). As the support
of $v$ is contained  $PJ= PJ_H $ where $J_H=J\cap H$, one gets from (\ref{xivint}):
$$\langle \xi, v\rangle= \int_{K \cap (PJ_H)}  \langle f_\xi(k), v(k)\rangle \> dk.$$
But   $K\cap (PJ_H)= (K\cap P)J_H$.  So,  one has:
$$ \langle \xi, v\rangle = \int_{(K\cap P)J_H}  \langle f_\xi(k), v(k) \rangle \> dk.$$
The function under the integral sign is left invariant under  $K \cap P$,    due to (\ref{xig}) and to  the properties of the induced representation. It is  right invariant under  $J_H$  due to the fact that $\xi$ is  $H$-invariant and that  $v$ is $J$-invariant. Hence (\ref{xivintbis}) follows.
 \ber \label{crm}
 Let $\cc{v}$ be a smooth (resp., $ \xi $ be an $H$-fixed) linear form on $i^G_PE$, where $P=MU$ is a parabolic subgroup of $G$ (resp., a parabolic subgroup of $G$ such that $PH$ is open in $G$) such that $A\subset M$ and $(\delta,E)$ is a smooth representation of $M$. We have  identified  $\cc{v}$ with an element of $i^G_P \cc{E} $ (cf (\ref{ccigp})) and we will denote its value at 1 by $\cc{r}_M \cc{v}\in \cc{E}$. Similarly  $f_\xi(1) \in E'^{M\cap H}$ is well defined. We will denote it  by $\cc{r}_M \xi$.\\  
 
 \eer  
 \begin{prop} \label{defrM} Let us assume that $A$ is a $\si$-stable maximal split torus which contains a maximal $\si$-split torus.\\
(i)   If  $P=MU$ is a  $\si$-parabolic subgroup of $G$ such that $A\subset M$, for  all  compact open subgroup of $K$, $J$,  with a $\si$-factorization with respect to $(P, P^-)$,  one has:
$$\cc{r}_M (e_J \xi)= e_{J_M} (\cc{r}_M\xi).$$
(ii) Let $Q=LV$ be a  a parabolic subgroup of $G$,   such that $QH$ is open and let $P=MU$ be a parabolic subgroup, such that  $P\subset Q$ and $A\subset M\subset L$. Let  $(\delta, E)$, $\cc{v}$, $\xi$ be  as above.  Applying induction in stages, $i^G_PE$ is isomorphic to $i^G_{Q} (i^L_{P\cap L} E)$. From  (\ref{crm}), one gets an element $\cc{r}_L \cc{v}$ of $(i^L_{P\cap L} E)\,\cc{}$( resp. $\cc{r}_L \xi$ of $(i^L_{P\cap L} E)'^{L\cap H}$).\\
(iii) Let us assume  that $P$ and  $Q$ are   $\si$-parabolic subgroups of $G$. Then one :
 $$\cc{r}_M\xi=  \cc{r_M} ( \cc{r}_L \xi).$$
\end{prop} 
\begin{proof}
(i) One  reduces easily  to compare the evaluation on any element $e $ of $ E^{J_M} $ of both sides of the equality to prove. 
One introduces   $v:= v^{P,J}_{e, \delta}$.  From (\ref{xivintbis}), one gets, on one hand: 
$$\langle \xi,v \rangle  =  \mathrm{Vol} ( (K\cap P)J_H)\langle  f_\xi(1), e \rangle .$$
As $ v$ is $J$-invariant this implies:
\beq \label{ejxi}\langle e_J\xi,v \rangle  =  \mathrm{Vol} ( (K\cap P)J_ H)\langle f_\xi(1), e \rangle .\eeq 
On the other hand: 
$$\langle e_J \xi, v \rangle  = \int_{K} \langle (e_J \xi) (k), v(k) \rangle  \> dk.$$
Again, we use that the support of $v$ is contained in  $PJ= PJ_H$:
 $$\langle e_J \xi, v \rangle = \int_{(K\cap P) J_H} \langle (e_J \xi)(k) , v(k) \rangle \> dk.$$
 The function under the integral sign  is left invariant under  $K\cap P$, by the properties of the induced representations,  and right invariant under $J_H$ as  $v$  and $e_{J} \xi$  are invariant under $J$. So one gets:
 $$ \langle e_J \xi, v \rangle = \mathrm{Vol} ( (K\cap P)J_H)\langle ( e_J\xi)(1), e \rangle .$$
 The equality needed to prove (i) follows from this and from (\ref{ejxi}). \\
 (ii) is a simple consequence of (i). \\
(iii) It is easy to see   that  for all   $ \cc{v} \in (i^G_P E) \,\cc{}$ the equality  $\cc{r}_M  \cc{v}=\cc{r}_M  (\cc{r}_L \cc{v} )$ holds. Then (iii) follows from the last assertion of (i).
\end{proof}
\subsection{The operation $\cc{j}_{Q^-}\circ $ \index{$\cc{j}_{Q^-}\circ $} \label{ccjq}}
Our second operation needs some preparation. \\
Let $P=MU$ be a parabolic subgroup of  $G$. Let $(Q,Q^-)$  be a pair of opposite parabolic subgroups of  $M$ with $L:=Q\cap Q^-$. Let us assume that $A\subset L\subset M$. Let $(\delta, E)$ be a smooth representation of  $M$. We denote   the parabolic subgroup of $G$ equal to $QU$ (resp.,  $Q^-U$) by $P_Q$ (resp., $P_{Q^-}$).  We define a $G$-homomorphism $\cc{j}_{Q^-}\circ $ from $(i^G_PE)\check{ }$ to $(i^G_{P_{Q^-}}  E_Q )\,\check{}$ as follows. Let  $\cc{v}\in (i^G_P E)\,\cc{}$. We identify as in (\ref{ccigp}) $(i^G_PE)\check{ }$ with $i^G_P \cc{E}$ and we denote  the element of $i^G_{P_{Q^-}}((\cc{E})_{Q^-}))$ obtained by composition of $\cc{v} $ with the projection $j_{Q^-} $ by $j_{Q^-}\circ \cc{v}$. By the Second Adjointness Theorem (cf. (\ref{becas})), $(\cc{E})_{Q^-}$ is canonically isomorphic to $(E_Q)\,\cc{}$. Let us denote  the image of $j_{Q^-}\circ  \check{v}  $ under  the induced  isomorphism by $j'_{Q^{-}}\circ \check{v}$. This is  an element of $i^G_{P_{Q^-}}(E_Q)\,\cc{}$. Then one defines $\cc{j}_{Q^-}\circ \cc{v}$ as  the image   of  $j'_{Q^{-}}\circ \cc{v}$  under  the  isomorphism defined by (\ref{ccigp}) of $i^G_{P_{Q^-} }(E_Q)\, \check{} $ with $(i^G_{P_{Q^-}}E_Q)\,\cc{}$.  
Summarizing $\cc{j}_{Q^-}\circ$ appears as the composition of the homomorphisms of $G$-modules:
$$ (i^G_PE)\, \check{ }\to i^G_P\check{ E}\stackrel{f}{\rightarrow}    i^G_{P_{Q^-}}((\cc{E})_{Q^-}) \stackrel{g}{\rightarrow} i^G_{P_{Q^-}}(E_Q)\,\cc{} \to
 (i^G_{P_{Q^-}}E_Q)\,\cc{}$$
 where the first arrow is the  isomorphism (\ref{ccigp}) between $(i^G_PE)\check{ } $ and $i^G_P\check{ E}$,
 the last arrow is the  isomorphim (\ref{ccigp}) between $ i^G_{P_{Q^-}}(E_Q)\,\cc{}$   and $(i^G_{P_{Q^-}}E_Q)\,\cc{}$, the arrow $f$ is given by composition of functions with   the projection $ j_{Q^{-} } $, and the arrow $g$ is the induced morphism from the canonical isomorphism between $ (\cc{E})_{Q^-}$ and $ (E_Q)\,\cc{}$ given by the Second Adjointness Theorem. Hence $\cc{j}_{Q^-}\circ$ is a $G$-module homomorphism. 
 One sees easily that: 
 \ber \label {j1}For $ \cc{v} \in (i^G_P E)\,\cc{}$, one has: $$ (\cc{j}_{Q^-}\circ  \cc{v})(1)  = j_{ Q^-}  (\cc{v}(1)),$$ where in the left hand side of the equality $\cc{j}_{Q^-}\circ  \cc{v}$ is viewed as an element of
  $i^G_{P_{Q^-}} (E_Q)\cc{ }$ and   where in the right side  $\cc{v}$ is viewed as an element of $i^G_P \cc{E}$ (cf. (\ref{ccigp})) and $ j_{ Q^-}  (\cc{ v} (1))$ is viewed as  an element of  $(E_Q)\,\cc{}$ by the Second Adjointness Theorem.
 In other words:$$ \cc{r} _{L} (\cc{j}_{Q^-}\circ  \cc{v}) =  j_{ Q^-} ( \cc{r}_M (\cc{v})) .$$\eer
  \begin{prop} \label{defccj}
  Let  $P=MU$ be a parabolic subgroup of  $G$. Let $(Q,Q^-)$ be a pair of opposite parabolic subgroups of  $M$. Let us assume that $A\subset L:=Q\cap Q^-$. Let $(\delta, E)$ be a smooth representation of  $M$ and let $\xi$ be an $H$-form on  $V=i^G_P E$.\\
  (i) Let $v_1\in V_1: =i^G_{P_{Q^-}} E_Q$.  The number $ \langle \cc{j}_{Q^-} \circ (e_J \xi), v_1 \rangle $ does not depend on the compact open subgroup  $J$ of $G$  such that $v_1$ is fixed by $J$.\\
(ii) This allows us to define  a  linear form on   $V_1$,  denoted by $\cc{j}_{Q^-} \circ \xi$, as follows.\\
If $v_1 \in V_1$ is fixed by the compact open subgroup $J$ of $G$,  one defines
 $$\langle \cc{j}_{Q^-} \circ \xi, v_1 \rangle : = \langle \cc{ j}_{Q^-} \circ (e_J \xi), v_1 \rangle . $$
Then  $ \cc{j}_{Q^-} \circ \xi$  is  an $H$-form on $V_1$.\\
(iii) For every compact open subgroup $J$ of $G$, one has: 
$$e_J (\cc{j}_{Q^-} \circ  \xi )=\cc{j}_{Q^-} \circ ( e_J \xi).$$ \\ (iv) The support of $\cc{j}_{Q^-} \circ \xi$ is contained in the support of $\xi$.\\
\end{prop}
\begin{proof}
(i) Let $v_1 \in V_1$. It is enough to prove that if $J'\subset J$ are two compact open subgroups of $G$ which leave $v_1$ invariant, one has:
 $$ \langle  \cc{j}_{Q^-} \circ( e_{J'} \xi), v_1 \rangle  =  \langle  \cc{j}_{Q^-} \circ (e_J \xi), v_1 \rangle .$$
As $e_J v_1= v_1$, one has :
 $$ \langle  \cc{j}_{Q^-} \circ (e_{J'} \xi), v_1 \rangle =  \langle  \cc{j}_{Q^-} \circ (e_{J'} \xi), e_J v_1 \rangle =   \langle e_J  \cc{j}_{Q^-} \circ (e_{J '} \xi), v_1 \rangle $$
 $$ = \langle  \cc{j}_{Q^-} \circ (e_J \xi), v_1 \rangle,$$
 which proves (i).\\
(ii) Let $\pi$ (resp., $\pi_1$) the induced representation of $G$ on $V=i^G_P E$ (resp.,  $V_1= i^G_{P_{Q^-}}E_Q$). Let $v_1$ be an element of $V_1$ and let $h$ be an element of $H$. One may choose a small enough compact open subgroup $J$ of $G$, such that $h. J$  and $h^{-1}.J$ leave also invariant $v_1$ under $\pi_1$. This implies that $ \pi_1 (h)v_1$ is also fixed by $J$. 
Then,  one has,  from the definition of $\cc{j}_{Q^-} \circ \xi$:
$$\langle \cc{j}_{Q^-} \circ \xi,\pi_1(h) v_1 \rangle = \langle  \cc{j}_{Q^-} \circ ( e_J \xi), \pi_1 (h)v_1 \rangle .$$
By elementary operations one sees that 
$$\langle \cc{j}_{Q^-} \circ \xi,\pi_1(h) v_1 \rangle  = \langle \cc{j}_{Q^-}\circ(\pi '( h^{-1}) e_J \xi), v_1 \rangle = \langle \cc{j}_{Q^-}  \circ (\pi'(h^{-1}) e_J \pi'(h)\xi), v_1 \rangle $$ $$
= \langle \cc{j}_{Q^-}  \circ (e_{h^{-1} .J} \xi), v_1 \rangle.$$
Hence one gets, from the definition of $\cc{j}_{Q^-} \circ \xi$ and the fact that $v_1$ is $h^{-1}.J$ fixed, the equality:
$$\langle\cc{j}_{Q^-} \circ \xi,\pi_1(h) v_1 \rangle = \langle \cc{j}_{Q^-} \circ \xi, v_1 \rangle .$$
which proves  the $H$-invariance of $\cc{j}_{Q^-} \circ \xi$. 

 The linearity is proved in the same way.
This proves (ii).\\
(iii) is an immediate corollary of (ii).\\ 
(iv)  Let $F$ be the support of  $ \xi $ and let $v_1 \in V_1$ whose support, $F_1$,  is contained in the complement of $F$ in $G$. Let us choose a compact open subgroup $J$   of $G$ which fixes $v_1$ and  such that $FJ$ is disjoint from $F_1$, which might be achieved from the compactness of $P_Q \backslash G$. 

Then the support of $e_J\xi $ is disjoint from the support of $v_1$. As the composition by $j_{Q^-}$  do not increase supports, one concludes, from the definition in (ii), that:
$$\langle \cc{j}_{Q^-} \circ \xi, v_1 \rangle = 0.$$
This implies (iv). 

\begin{lem} \label{restri}
We keep the notations of the preceding proposition, but we assume that $P$ is a $\si$-parabolic subgroup with $\si$-stable Levi subgroup $M$ and that $Q$ is a $\si$-parabolic subgroup of $M$ with $\si$-stable Levi subgroup $L$.  We assume moreover that $A$ is $\si$-stable , contained in $L$, and that it contains    a maximal $\si$-split torus $A_\vid$  of $L$.
 Let $x$ be an element of $G$ which is $A_\vid$-good.  Then $\pi'(x) \xi$ is an  $x.H$-form,  $P$, $Q$ are $\si_x$-parabolic subgroups (see  Lemma \ref{sixP}) and  one has :
$$ \cc{r}_{L}(\pi'_1(x) ( \cc{j}_{Q^-} \circ  \xi))=  j_{Q^{-}} (\cc{r}_M (\pi'(x)  \xi) ) .$$
\end{lem}
\proof 
We first treat the case where $x=1$. \\Notice that  $\si(P_{Q^-} )= (P^- )_{Q}$ is opposite to $P_{Q^-}$. Hence  $P_{Q^-}$ is a $\si$-parabolic subgroup of $G$.  \\
Let $e$ be an element of $ E_Q$.  From  Proposition \ref{KT},  one can choose  a compact open subgroup of $K$, $J$,  arbitrary small, which has a $\si$-factorization  with respect to $(P, P^-)$ and $(P_{Q^-}, (P^-)_{Q})$, such that $J_M$ has  a $\si$-factorization with respect to $(Q,Q^-)$. Hence  we can choose $J$ such  that $J_L$ fixes $e$.
One has:  \beq \label{JJ} J = (J\cap P_{Q^-})(J\cap H), \eeq  $$J_M =  (J_M \cap Q^-) (J_M \cap H).$$
 Let us prove  \beq \label{ej'}\langle \cc{r}_{L} ( e_{J} (\cc{j}_{Q^-} \circ \xi)) , e \rangle =\langle \cc{r}_{L} (  \cc{j}_{Q^-} \circ \xi),e \rangle  .\eeq
  Let $P'$ denote  $P_{Q^-}$. As  $e$ is fixed by $J_L$, 
  $v:=  v^{P',J}_{e,j_Q(\delta)}$ is well defined.
  
  One will compute in two ways  $\langle e_{J} \xi', v \rangle $  where  $\xi'=\cc{j}_{Q^-} \circ \xi$. 
 First $v$  is invariant under $J$  so that one has:
  $$\langle e_{J} \xi', v \rangle= \langle \xi', v \rangle  .$$ 
 Using (\ref{xivint} ) and (\ref{veM}),  one deduces from the preceding equality, as in the proof of (\ref{xivintbis}), that: 
$$\langle e_{J} \xi', v \rangle =\int_{(P'\cap K ) J_H} \langle f_{\xi'}(k), v(k) \rangle\>  dk. $$
The function under the integral sign is left invariant under  $K\cap P'$. Moreover, if $j\in J_H$,  $f_{\xi'} (j)=  f_{\xi'} (1)$, by the right $ H$-invariance of $ f_{\xi'} $ and $v(j)= e$ by the right invariance by $J$ of $v$.
So one gets: 
\beq \label{first} \langle e_{J} \xi', v \rangle =\mathrm{Vol}((P'\cap K )J_H)\langle f_{\xi'} (1), e \rangle .  \eeq 
 Our second computation of $\langle e_{J} \xi', v \rangle $ starts with:
$$ \langle e_{J} \xi', v \rangle = \int_K \langle (e_{J}\xi')(k), v(k) \rangle \> dk .$$
As  $v$ is supported on $P'J$, one gets: $$\langle e_{J} \xi', v \rangle  = \int _{(K\cap P')J }\langle (e_{J}\xi')(k), v(k) \rangle \> dk .$$
As the function to integrate is invariant under $P'\cap K$  on the left and by $J$ on the right, one has : 
\beq \label{second} \langle e_{J} \xi', v \rangle =\mathrm{Vol}((P'\cap K )J)\langle (e_{J}\xi')(1), e \rangle  . \eeq 
Notice that: $$ f_{\xi'}(1)= \cc{r} _{L} (\cc{j}_{Q^-} \circ \xi),  (e_{J}\xi')(1)= \cc{r} _{L}( (e_{J}\xi')).$$ 
 Then, taking into account the equality $J=( J\cap P')J_H$, one sees that $(P'\cap K) J= (P'\cap K) J_H$. Then  (\ref{ej'}) follows from these two computations of $\langle e_{J} \xi', v \rangle $ (cf. (\ref{first}) and (\ref{second})) . 
  
From the fact that the composition by $j_{Q^-}$ commutes with right translations by elements of $G$, one sees:
$$e_{J} ( \cc{j}_{Q^-} \circ \xi) = \cc{j}_{Q^-} \circ (e_{J }  \xi )$$ 
hence: 
$$\langle \cc{r}_{L} ( e_{J} ( \cc{j}_{Q^-} \circ \xi)) , e \rangle =  \langle (\cc{r}_{L}( \cc{j}_{Q^-} \circ (e_{J }  \xi )), e \rangle $$
From this and  (\ref{j1}),  one deduces:
$$\langle \cc{r}_{L} ( e_{J} (\cc{j}_{Q^-} \circ \xi)), e \rangle =\langle j_{Q^-}( \cc{r} _{M}({e_{J} \xi})),e \rangle . $$
 As $J$ has a  $\si$-factorization for $(P,P^-)$, one deduces from Proposition \ref{defrM}   (i), that : 
  $$\cc{r}_M ( e_{J} ( \xi)) = e_{J_M} \cc{r}_M ( \xi) .$$
 Replacing in the above equality, one gets: 
 $$\langle \cc{r}_M ( e_{J} (\cc{j}_{Q^-} \circ \xi)), e \rangle = \langle  j_{Q^-}( e_{J _M}  \cc{r}_M \xi), e  \rangle . $$
From (\ref{defj}),  and using the fact that $J_M$ has a $\si$-factorization for $(Q, Q^-)$, 
 this implies: 
$$\langle \cc{r}_M ( e_{J} (\cc{j}_{Q^-} \circ \xi)), e \rangle= \langle j_{Q^-}( \cc{r}_M \xi ), e \rangle . $$ 
Together with (\ref{ej'}), this shows that:
$$  \langle \cc{r}_{L} (\cc{j}_{Q^-}\circ \xi),e \rangle )= \langle  j_{Q^-}(\cc{r}_M \xi) , e \rangle , $$
which proves   the assertion for $x=1$.

Let us treat the general case. 
Then (see Lemma \ref{sixP}),  $P$ is a $\si_x$-parabolic subgroup  of $G$, $M$ is $\si_x$-stable and $Q$ is a $\si_x$-parabolic subgroup of $M$. One may apply the first part of the proof to $ \xi_x:= \pi'(x)\xi$ which is fixed by $x.H$. The result follows from the fact that
$\cc{j}_{Q^-} \circ $ is a $G$-module homomorphism.
\end{proof}
\section{Generic Basic Geometric Lemma}
In the two next subsections, we make no assumptions on the characteristic of the residue field of $\F$.
\subsection{Intertwining integrals} We keep the choices of (\ref{AK}).
Let $P=MU$ be a parabolic subgroup of $G$  such that $A\subset M$, and let $(\delta,E)$ be a smooth representation of $M$ with finite length.
Let $Q=MV$ be another parabolic subgroup of $G$ with $M$ as Levi subgroup. We denote  the set of roots of $A_M$ in the Lie algebra of $Q$ which are not roots of $A_M$ in the Lie algebra of $P$ by $\Sigma(Q,P)$ \index{$\Sigma(Q,P)$}.  Let $P^-=MU^-$ be the opposite parabolic  to $P$ which admits $M$ as Levi subgroup.  
We have a right $V$-invariant measure on $V \cap U\bb V $  which follows from our choice of measures (cf. (\ref{AK})).  We have a canonical identification of $V \cap U\bb V $ with $V \cap U^-$. 
 
 We have   $A_M\subset A$  and we fix a scalar product on $\a'$ which is invariant under the Weyl group of the pair $(G,A)$. It induces a scalar product on $\a_M'$ (cf. (\ref{iden})).

One has (cf. [Wal], Theorem IV.1.1 and Proposition IV.2.1):
\ber  \label{apq}There exists $R>0$ such that, for  all $ \chi \in X(M)$ which satisfies  $$(\mathrm{Re} (\chi),\aa) > R,  \aa\in \Sigma(Q,P),$$ there exists an intertwining operator, $A(Q,P, \delta_\chi)$ \index{$A(Q,P, \delta_\chi)$},  between $i^G_P \delta_\chi$ and $ i^G_Q\delta_\chi$ satisfying:
$$ (A(Q,P, \delta_\chi) f)(g)= \int_{ V\cap U\backslash V}f(vg) \> dv, f \in i^G_P E_\chi,$$
the integral being absolutely convergent.  
This family of operators viewed in the compact realization admits an extension to a rational family in $\chi\in X(M)$ denoted in the same way. More precisely, there is a nonzero polynomial function  $b$ on $X(M)$  such that for all $f$ in $i^K_{K\cap P} E$,  the family $b(\chi) (A(Q,P, \delta_\chi) f )$ is polynomial, in the compact realization. 
\eer 
From this characterization, one deduces: 
\ber \label{acom} 
The intertwining integrals  commute with induced operators from intertwining operators between smooth representations of finite length of $M$. Namely let  $(\delta^1, E^1)$  be another smooth admissible representation of $M$, and let us assume that  $T$ is an intertwining oprator  between  the representations $(\delta, E) $ and $(\delta_1, E_1)$.  For $\chi\in X(M)$, the composition with $T$ determines an induced map $ind(T^P_\chi)$ between $i^G_P E_\chi $ and  $i^G_P E^1_\chi $. With the notations above, one has:
$$  T^Q_\chi\circ A(Q,P, \delta_\chi)= A(Q,P, \delta^1_\chi) \circ T_\chi^P. $$
\eer

Let us show that:
 \ber \label{nonz} When $A(Q,P, \delta_\chi)$ is defined,  this operator is nonzero.
\eer Let $P^-$ be the opposite parabolic subgroup to $P$ with Levi subgroup $M$. 
Let $e \in E$ and let  $J $ be a compact open subgroup with Iwahori factorization with respect to $(P,P^-)$ and such that $e$ is invariant under $J_M$.  Let $f= v^{P, J} _{e, \delta_\chi} $ whose support is $PJ$.  As $ (V \cap U^{-})\cap PJ= V\cap (U^{-}   \cap J)$ one sees that
$$(A(Q,P, \delta_{\chi} ) f) (1)= \mathrm{Vol}(V\cap U^{-}   \cap J) e$$
which proves our claim.  

The following Lemma is an immediate consequence of the induction in stages and of the definitions.
\begin{lem} \label{indop}
Let $ P=MU$ be  a parabolic subgroup of $G$ and let $Q_1=LV_1, Q_2=LV_2$ be two parabolic subgroups of $M$. Let us assume $A\subset L$. Let  us define  $P_{Q_1}: = Q_1U$, $P_{Q_2} :=Q_2U$.
Let $ (\omega,F)$ be  a finite length smooth representation of $L$. By induction in stages, the representation   $ i^G_{P_{Q_i} } \omega$ is identified with  $i^G_P (i^M_{Q_i}  \omega)$.
\\Let  $v \in i^G_{ P_{Q_1} }E_{1} $. Then,  with the identications defined above, one has the equality of rational functions in $\chi\in X(L)$:
 $$A(Q_1, Q_2, \omega_\chi ) \circ v  = A(P_{Q_1} , P_{Q_2}, \omega_\chi) v . $$
   \end{lem} 
From [Wal], IV.1 (11), one has the relation:
\beq \label{tA} \langle A(Q,P, \delta )f, \cc{f} \rangle = \langle f, A(P,Q, \cc{\delta})\cc{f} \rangle  , f \in i^G_P E, \cc{f} \in  i^G_P \cc{E} .\eeq 

Let us prove:
\begin{lem} \label{suppAf} Let $\chi \in X(M)$ such that $A(Q,P, \delta_\chi)$ and $A(P, Q, \cc{\delta_\chi})$ are defined.\\
 (i) Let $f \in i^G_P E_\chi $.Then one has: $$\mathrm{Supp} (A(Q,P, \delta_\chi) f )  \subset cl (V (\mathrm{Supp}(f)) ),$$
 where $cl $ denotes the closure in $G$ and $\mathrm{Supp}$ the support. \\
 (ii) Let $T$ be a linear form on $i^G_Q E_\chi$. Let $T'= T \circ A(Q,P, \delta_\chi)$. Then  one has: $$\mathrm{Supp} ( T') \subset cl( U(\mathrm{Supp}(T)) ) .$$ 
 \end{lem}
\begin{proof}
 (i) Let  $g\notin cl (V( \mathrm{Supp}( f) ))$.  Let us show that  $g$ is not element of the support of $(A(Q,P, \delta_\chi) f )$. One reduces immediately, by holomorphy,  to the case where $A(Q,P, \delta_\chi)$ is defined by a converging integral. If  there exists  $v\in V$ such that $f (vg) $ does not vanish,  $g $ has to be  an element of $ V\mathrm{Supp}(f)$. As this is not true, this implies that $(A(Q,P, \delta_\chi) f)(g)=0  $. This proves (i).\\
  (ii) Let $A:= A(Q,P, \delta_\chi) $. Let $ g$  be an element of $\mathrm{Supp}( T' )$. Then for any  compact open neighborhood $\Omega$ of $g$ in $G$, there exists an $f \in i^G_PE_\chi$ with support   in $P\Omega$, such that $\langle T', f \rangle \not=0$.   Then $\langle T, Af \rangle \not=0$, so that $\mathrm{Supp} \> T\cap \mathrm{Supp} (A f)$ is nonempty. By (i), $\mathrm{Supp} (Af )\subset cl (V \mathrm{Supp}(f))$. So one has:\beq \label{VP} cl (V P \Omega)\cap   \mathrm{Supp} (T)\not=\vid.\eeq Let us show  that if  $X$ is a subset of $G$ and $(\Omega_p)$ is a decreasing sequence of compact open  neighborhoods of $g$  in $G$, whose intersection is reduced to $g$, \beq \label{inter} \cap_{p\in \N} cl( X\Omega_p)= cl (Xg). \eeq  In order to see this,  one can reduce to the case where $g=e$.
    If  $y\in \cap_{p\in \N} cl( X\Omega_p)$,  for all $p$, $y= lim x_{n, p} \omega_{n, p}$ where $\omega_{n, p} \in \Omega_p$ and $x_{n, p} \in X$. Let  $V, V'$ be  symmetric neighborhoods  of $e$ in $G$ with  $V'^2Ê\subset V'$. Let  $p \in \N$ such that $\Omega_p \subset V'$ and let  $n\in \N $ such that  $y^{-1} x_{n, p}\omega_{n, p} \in V'$. Then  $y^{-1}x_{n, p}$ is an element of $V$. Hence $y$ is an element of $cl(X)$, which proves our claim.
    But, as $VP= QU$, $VP\Omega_p$ is left $Q$-invariant and the image of $cl(VP\Omega_p) $ in $Q\backslash G$ is closed as the projection is open. Hence this projection is compact. It is the same for the projection of $\mathrm{Supp} (T)$. Together with (\ref{VP}), an argument of compactness  shows that the intersection  $\cap_{p\in \N} cl(VP\Omega_p) \cap \mathrm{Supp}(T)$ is nonempty.    Together with (\ref{inter}),  this implies: $$cl(VPg) \cap \mathrm{Supp} (T) \not= \vid.$$ Then, one sees that $g \in cl( PV \mathrm{Supp}(T))$: if $(v_np_n g)$ admits  $t\in \mathrm{Supp}(T)$ as a limit,  then  $( v_np_n)$ has  $tg^{-1}$ as a limit and $(( v_np_n)^{-1} t)$ has $g$ as a limit. 
    
       But, by the $Q$-invariance of $\mathrm{Supp} (T)$ and the equality $PV= UQ$, one has: $$PV\mathrm{Supp} (T )= U \mathrm{Supp}(T). $$ Hence $g$ is an element of $cl(U \mathrm{Supp} (T))$, which proves (ii).
    \end{proof}
    \begin{defi}
A $(Q,P) $-subset \index{$(Q,P) $-subset} of $X(M)$ is the  complement in $X(M)$  of a finite union  of sets of the form $\{\chi_ \nu\vert \nu\in (\a'_M)_\C,  (\nu,  \alpha ) = c \}$, where $\alpha$ describes the set  $\Sigma (Q, P)$. 
\end{defi}
It is clear that  such a set is open and dense in $X(M)$.
\begin{lem}Ê\label {hoinv} There exists a $(Q,P) $-subset of $ X(M)$, such that for $ \chi $ element of this  set,  $A(Q, P, \delta_\chi)$ is invertible   and the map $\chi \mapsto A(Q, P, \delta_\chi) $, viewed in the compact realization, is holomorphic on this set.
\end{lem}  \begin{proof}From [Wal] IV.1.1(12) and (14), it suffices to prove the statement assuming that  $P$ and $Q$ are adjacent and opposite, hence maximal. Let $\aa$ be the single element of $ \Sigma(Q,P)$. Recall that  $\a_M =\a_M^G \oplus \a_G$ (cf. (\ref{oplus})). Here $\a_M^G $ is  one dimensional. Hence  the image in $X(M)$ of $ (\a_M^G)'_\C $,  $X(M)^G$,    under the map $\ll\mapsto \chi_\ll$,  is a one dimensional torus.   Thus the family $\chi \to A(Q, P , \delta_\chi)$, depending rationally on $\chi \in X(M)^G $,   it has a finite number of poles $\chi_i$. One remarks that  $A(Q, P , \delta_\chi) $ does  not change if $ \chi$ is multiplied by  an element of $X(G\vert M)$, so  that $A(Q, P , \delta_\chi) $ has poles only along the sets $ \chi_i X(G\vert M)$. This implies  the holomorphy statement.

From [Wal],  IV.3, there exists a rational function on $X(M)$ with values in $\C$, $j$,  such that $A(P, Q,\delta_\chi)A(Q,P, \delta_\chi) $ is equal to the  multiplication by  $j(\chi)$.  Moreover (cf. [S], Theorem 3.2), $ i^G_P \delta_\chi$ is irreducible for $ \chi$ in an open dense subset of $X(M)$. From (\ref{nonz}), one deduces that $j$ is not identically zero. Also it is invariant under $ X(G\vert M)$, by the remark above. Hence,  again,  its poles and zeros  are along a finite number of subtori of $X(M)$ of the form $\chi'_l X(G\vert M) $ where $ \chi'_l \in X(M)^G$. The second part of the lemma follows. \end{proof}  It follows from   (\ref{acom}) and the proof of the preceding lemma that:
\ber \label{ainv}The intertwining integrals  and their inverses commute with induced operators from intertwining operators between smooth representations of finite length of $M$.\eer
Let $P'=M'U'$ be a parabolic subgroup of $G$ with $M\subset M'$, $P\subset P'$. Let $(\delta,E)$ be a finite length smooth representation of $M$. Let $v$ be an element of $i^G_P E$. We denote  the value at 1 of the element of $i^G_{P '}(i^{M '}_{M' \cap P}\delta )$ associated to $v$ by the induction  in stages by $r_{M '}v $ \index{$r_{M'}$}. Thus for $\chi \in X(M)$, it defines a map denoted again by $r_{M'} $: 
\beqÊ\label{rm} r_{M'} : i^G_P E_\chi \to i^{M'}_{P\cap M'} E_\chi . \eeq 
We will identify $i^M_ME$ with $E$.
\begin{lem} \label{QPr}
  There exists a $(Q,P)$-subset of $X(M)$, $O$,  such for every  $\chi\in O$ one has the following property:\\
   For every   $G$-submodule, $V$,  of $i^G_P E_\chi$ the equality  $$r_M (V) = E$$ implies the equality   $$r_M (A(Q,P, \delta_ \chi) V)= E.$$  
  \end{lem}
  \begin{proof}
   Let us take a $(Q,P) $-set as in the previous lemma and let $\chi$ be an element of this $(Q,P)$-set. If $r_M (A(Q,P, \delta_\chi) V)$ is  equal to a strict $M$-submodule, $E_1$,  of $E$, this implies that  $A(Q,P, \delta_\chi) V $ is a submodule of $i^G_Q E_1$. By (\ref{ainv}) one would have $r_M V\subset E_1$. This leads to a contradiction. This proves that the $(Q,P)$-set of the previous lemma has the required property. The lemma follows.      \end{proof} 
     
     \subsection{Generic Basic Geometric Lemma}\label{gblemma}
     Let $P$ (resp., $P'$) be a parabolic subgroup of $G$ with Levi subgroup $M$ (resp., $M'$). Let $A$ (resp., $A'$) be a maximal split torus of $M$ (resp., $M'$). We choose  a maximal compact subgroup $K$ (resp. $K'$) \index{$A,A', K,K'$}  of $G$ which is  the stabilizer of a special point of the apartment associated to $A$ (resp. $A'$) in the extended Bruhat-Tits building of $G$.  Let us show:
      \ber \label{wmm} There exists a set of representatives     of $P'\backslash G/ P$  such that for each of its elements, $w$,  one has   $w. A=A'$ \eer   
       By considering a minimal parabolic subgroup  of $G$, contained in $P$ (resp., $P'$)  and containing $A$ (resp., $A'$),  one can reduce to the case where $P$ and $P'$ are minimal  parabolic subgroups of $G$. Then $ P'= x.P$ for some element $x$ of $G$. As all maximal split tori in a minimal parabolic subgroup are conjugate  by an element of this parabolic subgroup (cf. [BoTi],  Theorem 11.6), one can choose $x$ such that $x.P=P'$ and  $x.A=A'$. On the other hand, by the Bruhat decomposition 
       $G= \cup_wPwP $ where the $w$ normalize $A$. Hence, $ G= \cup_w P' xwP$. Then (\ref{wmm}) follows from the fact that $xw. A= A'$. 
       
       We will say that a maximal split torus of $G$ is maximally $\si$-split if it contains a maximal $\si$-split torus \index{ maximally $\si$-split}. Such a torus is $\si$-stable (cf. [HW] Lemma 4.5 (ii)). 
   \ber \label{wm/g/m}     We will denote  a set of representatives   of $P'\backslash G/ P$  such that for each $w \in W(M' \backslash G /M)$,  $w. A= A'$ by  $ W(M' \backslash G /M)$ \index{$W(M' \backslash G /M)$}(although this set is not unique). \eer  Then $M'\cap w.M$ (resp.,  $M\cap w^{-1}.M'$) contains $A'$ ( resp., $A$ ) and  is the Levi subgroup  of  the parabolic subgroup $M'\cap w.P$ (resp.,  $M\cap w^{-1}.P'$). of $M'$ ( resp., $M$) which contains $A'$ (resp., $A$). 
   
      If $P=MU$, $P'=M'U'$ are $ \si$-parabolic subgroups of $G$, one will assume  that    $A$ (resp., $A'$) is a maximally $\si$-split $\si$-stable maximal split torus of $M$ (resp., $M'$).  
     \begin{prop} \label{ww} Let $P=MU$ and $P'=M'U'$ be   parabolic subgroups of $G$ such that  $A\subset M$, $A'\subset M'$. We denote  the set of  $A_M$-roots in the Lie algebra of $P$  by  $ \Sigma(P)$.
Let $(\delta,E)$ be a smooth  representation of finite length of $M$. 
Let $w, w' $  be two distinct elements of $ W(M' \backslash G /M)$. Let 
$X$ be a complex subtorus of $X(M)$ stable under  complex conjugation. We assume that the Lie algebra of $X$   contains at least an element $\nu $ such that $( \nu,\aa) $ is  strictly positive for 
each   element $\alpha$ of the set $  \Sigma(P) $. This condition is satisfied in particular if $P=MU$  is a $\si$-parabolic subgroup and $X=X(M)_\si$.  Then the following holds:
\\ (i) The set $O_{w,w '}$ of elements  $\chi$ of   $X$   such that:  
\ber  The $M'$-modules  $V_{\chi,w}=  i^{M'}_{M'\cap w.P} (w j_{M\cap w^{-1}.P'} E_\chi)$ and   $V_{\chi,w'}=i^{M'}_{M'\cap w'.P} (w' j_{M\cap w'^{-1} P'} E_\chi)$ have disjoint sets of Bernstein parameters (cf. Section \ref{pb} for the terminology).\eer 
is open and dense in $X$.  If $(\delta,E)$ is unitary, $O_{w,w'} \cap X(M)_u$ is dense in  $X_u:=X\cap X(M)_u$.
\\(ii) If $ \chi$ is an element of the open dense subset $O= \cap_{w, w' \in W(M'\backslash G/M), w\not=w'} O_{w,w'}$ of $X(M)$, the Jacquet module $j_{P'} (i^G_P E_\chi)$ is isomorphic to the direct sum:
$$ \oplus_{w\in W(M' \backslash G/M)}  i^{M'}_{M'\cap w.P} (w j_{M\cap w^{-1}.P'} E_\chi).$$ 
\end{prop}
\begin{proof}
Let $ \{ \Lambda_1,\dots , \Lambda_p\}$ be  the set of   Bernstein's  parameters  of  the representation $( j_{M\cap w^{-1}.P'} \delta_,  j_{M\cap w^{-1}.P'} E)$ of  $M\cap w^{-1}.M'$ , where for every $i$,  $\Lambda_i= (L_i, \omega_i)_{M\cap w^{-1}.M'}$,   $L_i$ is  a Levi subgroup of $M\cap w^{-1}.M'$  which contains $A$ and $\omega_i$ is a cuspidal reprentation of $L_i$ i.e. whose smooth coefficients have a support which is compact modulo the center of $L_i$. \\
We introduce similar data related to $w'$,  $L'_j$, $\omega'_j$.
Then, using (\ref{Bchi}) and  (\ref{indpar}),  one sees: \ber\label{vw} The set of  Bernstein's parameters  of  the finite length $M'$-smooth module $V_{\chi,w}$ is equal to 
$$\{  ( w.L_1, w (\omega_1 \otimes \chi_{\vert L_1 }))_{M'}, \dots ,   ( w.L_p, w( \omega_p \otimes \chi_{\vert L_p  }))_{M'} \}.$$
\eer

Let us prove that the set $Y=X\setminus O_{w,w'} $,  is closed in $X$ and has an  empty interior.
From (\ref{vw}), one sees that  $\chi \in Y$,  if and only if  for some $i, j$, one has:
\ber  \label{conjM}$( w.L_i, w. (\omega_i \otimes \chi_{\vert L_i} ))_{M'}$ is $M'$-conjugate to \\ $( w'.L'_j,  w'. (\omega_j \otimes \chi_{\vert L'_j }) )_{M'}$ .  \eer 
Let $Y_{i,j}$ be  the subset of elements of $Y$ satisfying (\ref{conjM}). Let us assume that $Y_{i,j} $ is nonempty. 
In particular $w.L_i$ is conjugate in $M'$ to  $ w'. L'_j$. \\
These are two Levi subgroups of $M'$ which contain $A'$ and  which are conjugate under $M'$. As two maximal split tori in $w'.L'_j$ are conjugate,  these   two Levi subgroups of $M'$ are conjugate by an element of the normalizer in $M'$ of $A'$. 

By multiplying  $w$ by this element of the normalizer in $M'$ of $A'$, one reduces to the case where these two Levi subgroups of $M'$,  
$w.L_i$, $w'L'_j$ are equal. Let us denote  this Levi subgroup of $M'$ by $L''$. Two cuspidal representations of $L''$, $\omega$, $\omega'$ define the same infinitesimal character for $M'$ if for some $x$ in  the normalizer of $ L''$ in $M'$, $N_{M'} (L'')$,   $x\omega$  is equivalent  to $\omega '$. Hence   $\chi\in Y_{i,j} $ if and only if  for some $x \in N_{M'} (L'')$, that  might be choosen to normalize $A'$,  
\ber \label{equivo}  $xw\omega_i \otimes xw\chi_{\vert L''} $ is equivalent to $ w' \omega'_j  \otimes w' \chi_{ \vert L'' }$. \eer 
For $x$ given, the set $Y_{i,j, x}$ of such $\chi$ is easily seen to be closed because 
\ber \label{fam} The characters of these two families of irreducible representations of $L''$ vary weakly holomorphically in $\chi$. \eer
As $Y_{i,j,x}$ depends only on the right coset $xL''$ and as $ N_{M'} (L'')/ L''$ is finite, this implies that $Y_{i,j}$ is closed.  Hence $Y$ is closed in $X(M)$   and $ O_{w,w'}$  is open in $X $.

Let us assume  that $O$ is not dense. This implies that  $ Y$ has a nonempty interior, hence by Baire's Theorem, there exist $w, w',i,j, x$ as above such that $Y_{i,j, x}$ has a nonempty interior.

By multiplying $w'$ by $x^{-1}$ one may and one will reduce to the case where $x=1$.
From (\ref{fam}), one deduces that for all $\chi \in X$, (\ref{equivo}) holds. In particular it is true for $ \chi= 1$. Denote by $\omega''$  the representation $w\omega_i$ of $L''$. Then one concludes also that for all $\chi \in X$,  $(w'\chi_{\vert L''}^{-1} )(w\chi_{\vert L'' }) $ belongs to the finite set of elements $\chi''$  of $ X(L'')$ such that $\omega''\otimes \chi''$ is equivalent to $\omega ''$. Hence, by connectedness, 
\ber for all $\chi \in X$,  $ w\chi_{\vert L'' } =w'\chi_{\vert L'' }$.  \eer 
Let us denote the Lie agebra of $X$ by  $ \b\subset (\a'_M)_\C$. 
By differentiation, it implies $$w\nu= w' \nu, \nu \in \b.$$ 
This might be written:
$$w'' \nu= \nu, \nu \in \b , $$  where  $w''= w'^{-1}w$ is an element of the normalizer of $A$ in $G$.
From our hypothesis on $\b$, one sees that  $ w''\nu = \nu$ for a strictly $ P$-dominant element of $\b$, $\nu$. But  $w'' $ acts on $\a'_\C$ as an element of the Weyl group of $A$,  which, by the above,   is a product of symmetries with respect to roots orthogonal to $\nu$.  The corresponding roots  have to be roots of $A$ in the Lie algebra of $M$, by our hypothesis on $\nu$.   This implies that $w''$  fixes pointwise 
$(\a_M)'_\C$. Hence $w'' $ is  an element of the normalizer  of $A$ which fixes pointwise $\a_M$. This implies that  it is an element of $M$.  As $w= {w'} ^{-1}w''$, this implies that $w$ and $w'$ represents the same element of $P'\backslash  G/P$. This is a contradiction with our hypothesis. Hence $Y$ has an empty interior and $O$  is dense in $X$. This proves the first statement on $O_{w,w'}$. The proof of the statement for $O_{w,w'} \cap X_u$ is similar.
 \\(ii) By the Basic Geometric Lemma ( cf. [R] VI.5.1) the Jacquet module $j_{PÔ} (i^G_PE_\chi)$ has a filtration whose associated graded object is the direct sum of the statement. (ii) is an immediate consequence of the definition of $O$ and  (\ref{splitsum}). 
 \end{proof}
\begin{lem} \label{V=} We keep the notation and the assumptions of the previous lemma. \\
    (i) Let us assume that $(\delta,E)$ is irreducible (resp. irreducible and unitary). There exists an open dense subset  of $X$ (resp., $X_u$) such that for every $\chi$ in this subset,  $i^G_P \delta_ \chi $ is irreducible. \\
  (ii) Let us assume that $(\delta,E) $  is  a finite length smooth representation of $M$. 
    There exists an open dense subset of $X$, $X'$,  such that for every $\chi\in X'$  and for every   $G$-submodule $V$ of  $i^G_P E_\chi$  such that $r_M V= E_\chi$,
 then $V= i^G_P E_\chi$.
   \end{lem}
    \begin{proof}
  (i)  follows easily  from [S], Theorem 3.2, where no assumption of unitarity on the inducing representation is made, and our hypothesis on  $X$.\\
   (ii) Let $X'$ be an open and dense subset of $X$ such that:
   \\1)  for every irreducible subquotient of $(\delta,E)$, $(\omega, F)$ and  $ \chi \in X'$, $i^G_P \omega_\chi$ is irreducible.
   \\ 2) $X'$ is a subset of the set $O$ of the preceding  lemma,  where we take $P= P', $ $M=M'$.\\ 
   The existence of $X'$ follows from (i)  and from the preceding lemma.\\ 
   We proceed by induction on the length of $E$  to prove that:
   \ber   \label{indu}An open dense subset, $X'$, of $X$ satisfying 1) and 2) above has the properties required by the lemma. \eer
   If $ E$ is of length one and $\chi \in X'$,  $i^G_P E_ \chi $ is irreducible. As $V$ is nonzero, one sees that   the claim is true in that case.

Now let us assume that (\ref{indu}) is true if $E$ is of length $p\geq 1$. Let  $E$  be  a smooth $M$-module of length $p+1$. Let $\chi$ be an element of $X'$.
   Let $ (\pi_1,V_1)$ be  an irreducible $G$-submodule of $V$. As $\chi \in X'$,  $\pi_1$ is  isomorphic to  $i^G_P \omega_\chi$ for some irreducible subquotient  $(\omega,F) $ of $(\delta,E)$.  This determines  a nonzero element, $T$,  of  $ \mathrm{Hom}_G( i^G_P F_\chi, V) \subset \mathrm{Hom}_G(  i^G_P F_ \chi, i^G_P E_\chi) .$ The  latter space  is isomorphic to $
   \mathrm{Hom}_ M ( j_P( i^G_P F _\chi),  E_\chi)$. 
    But from our hypothesis on $\chi$ and the properties of $X'$, $j_P( i^G_P F _\chi)$ splits as a direct sum $\oplus_{w \in W(M \backslash G/M)} i^{M }_{M\cap w.P} w j_{M\cap w^{-1}.P} F_\chi.$ As $\chi$ is an element of $O$, for $w \notin P$,  the set of Bernstein parameters of  $ i^{M}_{M\cap w.P} w j_{M\cap w^{-1}.P} F_\chi $ is disjoint  from the set of Bernstein parameters of  $ F_\chi$.  Hence, one has:
    $$
   \mathrm{Hom}_ M ( j_P( i^G_P F _\chi),  E_\chi )\approx 
    \mathrm{Hom}_{M} ( F_\chi ,E_\chi).$$ From this and from the fact that $T$ is nonzero, it follows  that 
    $\mathrm{Hom}_{M} ( F_\chi ,E_\chi)$ is non reduced to  zero. This  proves that $F$  appears as  a submodule of $E$, that we still denote by $F$.  
   Moreover $T$ is the induced map from an element  of $\mathrm{Hom}_{M} ( F_\chi ,E_\chi)$  and $V_1 $ is equal to $i^G_P F_\chi$. Going through the quotient of $V$ by $i^G_P F_\chi$  and applying the induction hypothesis, one gets the result.     \end{proof} 
   We need some notation.
   \ber \label{deflx} Let $P=MU$ be a parabolic subgroup of $G$, let $(\delta,E)$ be a smooth  representation of $M$ and let $x$ be an element of $G$. If there is no ambiguity, we will denote  the bijective intertwining operator between $i^G_P E $ and $i^G_{x.P} xE$, which associates to $v \in i^G_P E$ the element $\ll(x) v$ \index{$\ll(x)$} of $i^G_{x.P} xE$ defined by
   $$\ll(x) v(g)= v(x^{-1}g), g \in G,$$
   by $\ll(x)$. 
 We will also denote, again by abuse of notation,  the transpose of the inverse of $\ll(x )$  by $\ll(x)$.  It intertwines  $(i^G_P E)' $ and $(i^G_{x.P} xE)'$.
  \eer
   \begin{prop} \label{gb}Generic Basic Geometric Lemma.\\
Let  $P=MU$, $P'=M'U'$  be   parabolic subgroups of $G$ such that  $A\subset M$, $A'\subset M'$.   Let   $ \Sigma(P) $ be   the set of  $A_M$-roots in the Lie algebra of $P$, 
 let $(\delta,E)$ be   a smooth  representation of finite length of $M$  
and  let  $X$ be   a complex subtorus of $X(M)$ stable under  complex conjugation. We assume that the Lie algebra of $X$,   denoted by $ \b\subset (\a'_M)_\C$,  contains at least an element $\nu $ such that $( \nu,\aa) $ is  strictly positive for 
each  $\alpha$ element of the set $  \Sigma(P) $. This in particular satisfied if $P$ is a $\si$-parabolic subgroup and $X=X(M)_\si$. \\
If $ w \in  W(M'\backslash G/ M)$, let us define   parabolic subgroups of $G$  by: $$P_w = (M\cap w^{-1}.P')U\subset P, \>P'_w= (M' \cap w.P) U'\subset P'. \index{$P_w, P'_w$}$$
There exists a dense open subset $O$ of $X$, whose intersection with $X_u$ is dense, such that:
\\(i) For  $ \chi $ belonging to $ O$  and $ w$ belonging to $ W(M'\backslash G/ M)$,  the map $\alpha_{\chi, w}$ \index{$\alpha_{\chi, w}, \alpha_\chi$} is  well defined from $i^G_P E_\chi$ to $V_{\chi,w}:= i^{M'}_{M'\cap w.P} w(j_{M\cap w^{-1}.P'} E_\chi)$ by:
  $$\alpha_{\chi, w}(v)= r_{M'}  [ A(P'_w, w. P_w, wj_{M\cap w^{-1}.P'} \delta_\chi )( \lambda (w)\circ j_{M\cap w^{-1}.P'} \circ v)] ,  \index{$\alpha_{\chi,w}$} $$ for $v \in  i^G_P E_\chi$ . Moreover it goes through the quotient   to  a surjective  morphism of $M'$-modules from $j_{P'} i^G_P E_\chi$ to $V_{\chi,w}$, that we will denote in the same way. Here we use the normalization of measures as in (\ref{AK}) relative to $(A',K')$. \\
(ii) For  $ \chi \in O$, the map $$\alpha_\chi: j_{P'} (i^G_P E_\chi) \to \oplus_{w \in  W(M'\backslash G/ M)} V_{\chi,w}, \index{$\alpha_\chi$}$$ whose components are the $\alpha_{\chi, w}$, 
is an isomorphism  of $M'$-modules. 
 \end{prop} 
 \begin{proof}
Let us denote  the Levi subgroup of $P'_w$ which contains $A'= w.A$ by $M'_w $. Then $M'_w = M ' \cap w. M$.\\  
From the properties of intertwining integrals (cf. Lemma \ref{hoinv}),  $\alpha_{\chi, w} $ is well defined for $\chi $ element of $X$ and such that   $w \chi_{\vert M'_w}$ is element of some  $(P'_w, w.P_w)$-subset $Y_w$ of  $X(M'_w)$. \\
We denote   the set of such $\chi$ by $X_w$.  Such a set is open in $X$. Let us show that it is  dense in $X$. If it was false,   the complement of some    $(P'_w, w.P_w)$-subset of $X(M'_w)$ would contain  the set of    $w \chi_{\vert M'_w}$ when $ \chi $ varies in  a nonempty open subset of $X$.  Thus, by looking to tangent spaces, one would see that $w\b$ should be contained in  the orthogonal subspace to some nonempty  collection of roots,   $\aa$,   of the maximal split torus of the center of $M'_w$, $A_{M'_w}$,  in the Lie algebra  of  $P'_w$ and  which are not root in the Lie algebra of $w.P_w$.  But, by the hypothesis on $X$ in Lemma \ref{ww},  such a root would be  trivial on $w \a_M$, as the roots which are  orthogonal  to $\b$ are trivial on $\a_M$.  Hence it  would  be a root of  $A_{M'_w}$  in the Lie algebra of the  intersection of $P'$ with $w.M$.  On the other hand, from the definition of $P_w$ one sees that:   $$w.M \cap P' \subset w.P_w .$$
Moreover one has:
$$w.M\cap P'= (w.M \cap M')(w.M \cap U' ).$$
From the definition of $P'_w$ one concludes:
$$w.M\cap P'\subset  P'_w.$$ Hence one sees that  $w.M\cap P'$ is a subset of $w.P_w \cap P'_w$ and there is no root having the required property. This proves  that $X_w$ is dense in $X$. Similarly one sees  that $X_w \cap X_u$ is dense in $X_u$. 

Let us denote   the $G$-submodule $\{ \lambda(w) (j_{M\cap w^{ -1}.P'}\circ v)\vert v\in  i^G_P E_\chi \} $ of $i^G_{w.P_w} ( wj_{M\cap w^{-1} .P'}E_\chi) $ by $V$. From  the surjectivity of $j_{M\cap w^{-1} .P'}$ and the surjectivity of $r_M$ from $i^G_P E_\chi $ to $E$,
 one concludes, by ``transport de structure'',  that 
$r_{M'\cap w.M} V$ = $ wj_{M\cap w^{-1} .P'}E_\chi$.  By Lemma \ref{QPr}, one concludes that for $\chi$ belonging to an open dense subset of $X_w$ one has:
$$r_{M'\cap w.M}(  A(P'_w,P_w, wj_{M\cap w^{-1} .M'}E_\chi )V ) =w j_{M\cap w^{-1} .P'}E_\chi.$$
As $r_{M'\cap w.M}=r_{M'\cap w.M}\circ r_{M'} $, one concludes that for $ \chi$ belonging to this open dense subset of  $X_w$, the image of $ \alpha_{\chi, w}$, $\alpha_{\chi, w} (i^G_PE_\chi)$ satisfies:
$$ r_{M'\cap w.M} \alpha_{\chi, w} (i^G_PE_\chi) = w j_{M\cap w^{-1}. P'}E_\chi.$$ 
Then from Lemma \ref{V=}, one deduces that for $ \chi$ belonging to an open dense subset of $X_w$,   the image of $\alpha_{\chi, w}$ is equal to $ i^{M'}_{M'\cap w.P}( wj_{M\cap w^{-1}.P'}E_\chi)$. 
Hence the image of $ \alpha$   admits  $ V_{\chi,w} = i^{M'}_{M'\cap w.P}( wj_{M\cap w^{-1}.P'}E_\chi)$ as a quotient. 
The fact that $\alpha$ goes through the quotient to 
$j_{P '}( i^G_PE_\chi)$  follows   from the fact that, in the definition of $\alpha_{\chi, w}$,  all maps are $G$-morphisms except $r_{M'}$ which is an $M'$-module map. This  proves (i). 

(ii) Using (i) and  Lemma \ref{ww}, one sees that for $\chi$ in a dense open subset of $X$,  $O$, whose intersection with $X_u$ is dense in $X_u$,  $\alpha_{\chi, w}$ is surjective for  every $w\in W(M'\backslash G/M)$ and that the various   $V_{\chi,w}$ have disjoint sets of Bernstein's parameters. This implies (cf. (\ref{splitsum})) that the image is equal to the direct sum of the $V_{\chi,w}$. So $ \alpha $ is a surjective $M'$-module map from $j_{P '} (i^G_PE_\chi)$ to  $\oplus_{w \in W(M' \backslash G/M )} V_{\chi,w}$. On the other hand, by  Proposition \ref{ww},  $j_{P '} (i^G_PE_\chi)$ is an $M'$-module isomorphic   to $\oplus_{w \in W(M' \backslash G/M} V_{\chi,w}$. By looking to the length of modules, one concludes from this that $ \aa$ is bijective.
\end{proof}
\noindent Let $P$ and $Q$  be two parabolic subgroups of $G$, with common Levi subgroup $M$. Let $(\delta , E)$ be a smooth representation of finite length of $M$  such that  the operators $ A(Q,P, \delta)$ and $ A(P, Q, \cc{\delta} )$ are well defined.  Then the restriction of the transposed operator of $ A(Q,P, \delta)$, $ ^t A(Q,P, \delta)$, to the space of  smooth vectors  intertwines $(i^G_Q \delta  )\,\cc{} $ with $ (i^G_P \delta  )\,\cc{} $. Using the isomorphism of 
$i^G_P \cc{E}$ with $(i^G_P E )\,\cc{} $ and $i^G_Q\cc{E}$ with $(i^G_Q E )\,\cc{} $ (cf. (\ref{ccigp})), the restriction of $ ^t A(Q,P, \delta)$ to the space of smooth vectors defines an intertwining operator between $ i^G_Q\cc{E}$ and $ i^G_P\cc{E}$ which, by (\ref{tA}),  is equal to $ A(P, Q, \cc{\delta})$. \\
\begin{prop}\label{betaw} One keeps the notations of  the preceding proposition.
If $ w \in  W(M'\backslash G/ M)$, let us define two parabolic subgroups of $G$   by : $$\tilde{P} _w = (M\cap w^{-1}.P'^-) U\subset P , \> \tilde{P} '_w= (M' \cap w.P)U'^-\subset P'^- . \index{$\tilde{P} _w, \tilde{P} '_w$} $$  
There exists a dense open subset $O'$ of $X$, whose intersection with $X_u$ is dense in $X_u$,  such that:\\
(i) For $ \chi \in O'$  and $ w \in W(M'\backslash G/ M)$   and $\cc{v}\in (i^G_P E_\chi)\,\cc{}$:  $$\beta_{\chi,w} (\cc{v}):= \cc{r}_{M'} \circ\>  ^tA(w.  \tilde{P}_w,  \tilde{P}'_w, w. j_{M\cap w^{-1}.P'} \delta_\chi) \circ  \lambda (w)(\cc{j} _{M\cap w^{-1}.P'^-} \circ \cc{v} ) \index{$\beta_{\chi,w}$} $$  is  a well defined element of $(V_{\chi,w})\,\cc{}$ where $V_{\chi,w}:= i^{M'}_{M'\cap w.P} (wj_{M\cap w^{-1}.P'} E_\chi)$. Moreover  the map $\beta_{\chi,w}$  goes through the quotient   to  a surjective  morphism of $M'$-modules from $j_{P'^-} ((i^G_P E_\chi)\,\cc{}\,) $ to $(V_{\chi,w})\,\cc{}$  that we will denote in the same way. \\
(ii) For $ \chi \in O'$, the map $$\beta_\chi: j_{P'^-}(( i^G_P E_\chi)\,\cc\, ) \to \oplus_{w \in  W(M'\backslash G/ M)} (V_{\chi,w})\,\cc{} \index{$\beta_\chi$} $$ whose components are the $ \beta_{\chi,w} $ 
is an isomorphism  of $M'$-modules.
 \end{prop} 
 \begin{rem}
 Notice that in order to define $\cc{j} _{M\cap w^{-1}.P'^-} \circ $ we use the normalization of measures relative to $(A,K)$ (cf. (\ref{AK})). For the intertwining integrals we use $(A',K')$. When it will be needed, we will make a particular choice of $K' $.
 \end{rem}  
 \begin{proof}
 The   fact that $\beta_{\chi,w}(\cc{v})$ is a well defined element of $(V_{\chi,w})\,\cc{}$  follows from the definitions. 
 The rest of the proof is similar to the proof of the preceding proposition, using the isomorphism of the smooth dual of a parabolically  induced representation with the  parabolically induced representation of the smooth dual of the inducing representation (cf. (\ref{ccigp})). 
  \end{proof}\subsection{Generic Basic Geometric Lemma and $H$-forms}
  We come back  to our assumption that the characteristic  of $\F$ is different from 2.
 Let us keep the notations of the preceding proposition.
Let $ \chi \in O'$.
We set $ V= i^G_PE_\chi$.
The Second Adjointness Theorem   shows that $(i^G_P E_\chi)_{P'}$ is canonically   isomorphic to  $((( i^G_P E_\chi)\,\cc{} )_{P'^-})\,\cc{}$.
From the preceding proposition,  the isomorphism $\beta_\chi$,   determines  an isomorphism \index{$ \gamma_\chi$} $$ \gamma_\chi: j_{P'}(i^G_P E_\chi) \to V_{P',1}:= \oplus _{w\in W(M'\backslash G /M)}  V_{\chi,w}.$$
 We recall that
$\langle., .  \rangle _{P '} $ is the canonical pairing between $ V_{P'}$  and $ (\cc{V})_{P'^-} $.
In other words, the isomorphism $ \gamma_\chi$ is  characterized by 
\beq \label{gamma} \langle  \cc{v}_{P'^-}, v_{P '}    \rangle  _{P'}= \langle  \beta_\chi(\cc{v}), \gamma_\chi(v) \rangle , v \in V, \cc{v} \in \cc{V},  \eeq 
where, in the second member of the equality, the pairing is the natural pairing between  $ (V_{P',1})\,\cc{} $ and $V_{P',1}$. 
We denote  the composition of $ \gamma_\chi$ with the projection onto $V_{\chi,w}$ by $\gamma_{\chi, w} $.
\begin{theo} \label{xipw}
We keep the notations of the preceding proposition. We assume moreover that $P$ and $P'$ are $\si$-parabolic subgroups of $G$. Let $ \chi \in O'$ and let $ \xi$ be an $H$-form on $V:= i^G_PE_\chi$.
We define:
$$ \xi_{P ^{'-} }:=  j_{P'^{-}} \xi. $$
We will denote the components of the  linear form $\xi^{1}: = \xi_{P'^-}\circ \gamma_\chi^{-1} $  on $V_{P',1} =\oplus _{w\in W(M'\backslash G /M)}  V_{\chi,w}$  
 by $ \xi_{P'^-,w}\in( V_{\chi,w})' $, \index{$\xi_{P'^-,w}$} where $V_{\chi,w}:= i^{M'}_{M'\cap w.P} (wj_{M\cap w^{-1}.P'} E_\chi). $\\ 
 Then one has:
$$ \xi_{P'^-,w}=  \cc{r}_{M'} \circ ^t A(w.{\tilde P}_w, \tilde{P'}_w , w. j_{M\cap w^{-1}.P'} \delta_\chi)\circ  \lambda(w)(\cc{j}_{M\cap w^{-1}.P'^-} \circ \xi).$$
\end{theo}
\begin{proof}Let $J$ be a compact open subgroup  of $G$ with   a $\si$-factorization with respect to $(P', P'^-)$. Let $ v$ be an element of $ V^J$. Let us compute $\langle \xi_{P '^-}, v_{P'} \rangle  $.
First one has, from the definition of $ \gamma_\chi$ and $\xi^1$:
\beq  \label{xixi1} \langle \xi_{P '^-}, v_{P'} \rangle =  \langle  \xi^1,  \gamma_\chi( v_{P '}) \rangle . \eeq 
From the definition of $\xi_{P '^-}$  and  of the $\si$-factorization  (cf. (\ref{defj}), Proposition \ref{KT}), one has: 
$$\langle\xi_{P '^-}, v_{P'} \rangle  = \langle (e_{J} \xi)_{{P'}^-}, v_{P'} \rangle  _{P'}.$$ 
From (\ref{gamma}), one sees:
 $$\langle \xi_{P '^-}, v_{P'} \rangle = \langle  \beta_\chi (e_J \xi), \gamma_\chi(v_{P'}) \rangle .$$
 
 Let us denote  the element of  the dual of  $V_{P ', 1}$ whose components are 
 $$\cc{r}_{M'} \circ ^t A(w.{\tilde P}_w, \tilde{P'}_w , w. j_{M\cap w^{-1}.P'} \delta_\chi)\circ  \lambda (w)\circ \cc{j}_{M\cap w^{-1}.P'^-} \circ \xi,$$ 
 by $ \xi^2$.
 Using Proposition \ref{defccj}, Proposition \ref{defrM} (i), and the definition of $\beta_\chi$ in Proposition \ref{betaw},  one sees easily that
 $$\beta_\chi(e_J \xi) = e_{J_{M '}}  \xi^2 .$$
 Hence we get  a second expression for $\langle \xi_{P '^-}, v_{P'} \rangle  $:
 $$\langle \xi_{P '^-}, v_{P '} \rangle = \langle  e_{J_{M'}} \xi^2, \gamma_\chi( v_{P '})   \rangle  , v \in V^J.$$
 Together with (\ref{xixi1}), this implies that $\xi^1$ and $\xi^2$ are equal on $V_{P',1}^{J_{M'}}$. As there are arbitrary small open compact 
subgroups of $G$   with a $\si$-factorization with respect to $(P', P'^-)$, this implies that $\xi^1 = \xi^2$.
This finishes the proof of the theorem. 
\end{proof} 
\section{Two key Lemmas and some of their consequences}
\setcounter{equation}{0}
\subsection{Families of distributions on $PH$, where $P$ is a parabolic subgroup of $G$ }
We keep the notations of the preceding subsection.
Let $O$ be a nonempty   open subset of   $X$.
A map $\chi\mapsto\xi_\chi \in (i^G_P E_\chi)'$ defined on $ O$ is said to be weakly holomorphic if for all $v \in i^K_{K \cap P }E$,  the map 
$\chi \mapsto \langle  \xi_\chi, v_\chi \rangle $ is holomorphic on  $O$. 

We will denote  the map  $C_c^{\infty} (G) \otimes  E \to i^G_PE_\chi $ denoted by  $M_{ \delta_\chi,P} $   in equation  (\ref{Mdel}) by $M_\chi$ and we set $\tilde{\xi}_\chi:= \xi_\chi \circ M_ \chi$.
Let us prove:
\ber \label{finC}
Let  $ f \in  C_c^{\infty} (G) \otimes  E $. Then  $\chi \mapsto \langle  \tilde{\xi}_\chi, f \rangle $  is holomorphic on $O$, in other words $ \chi \mapsto \tilde{\xi}_\chi$ is  a weakly holomorphic family of $E$-distributions on 
$G$. 
\eer 
Let $ J$ be a compact open subgroup of $K$ such that $f$ is right and left invariant under $J$. Then 
$v(\chi): =(M_{\chi} f)_{\vert K} $ has its  values in the finite dimensional space $  (i^K_{K \cap P }E)^{J}$. To see that $ \chi \to v(\chi)$ is holomorphic, it is enough to check that 
for every $k \in K$, $\chi \to (v(\chi))(k)$ is holomorphic. By using left translates by elements of $K$, one can reduce to $k=1$. But $ f \in  C_c^{\infty} (G) \otimes  E$. Hence its restriction to 
$P$ is invariant under a compact open subgroup $J'$ of $P$ and is supported by a finite number of right cosets of $J'$ in $P$, $x_i J'$. Hence, using the definition of $M_\chi$  and the fact that  the unramified characters are trivial
 on compact open subgroups, one sees that: 
    $$(v(\chi))(1)= \sum_i  \chi (x_i^{-1}) (\int _{J'} \delta_P^{1/2} (x_i^{-1} ) \delta(( x_i p)^{-1}) f(x_i) \>  d_r p).$$
Hence $ \chi \mapsto v(\chi)$ is holomorphic. Then (\ref{finC}) follows from the finite dimension of $(i^K_{K \cap P }E)^{J}$ and from our hypothesis on the family $(\xi_\chi)$.

 Let  $ (\xi_\chi)$ be as above and let us  assume   that every $\xi_\chi $ has a zero restriction to  the complement of a closed set $F$ of $X$, which is left $P$-invariant. 
Then  (cf. Section  \ref{covd}) $ \tilde{\xi}_\chi $ induces on $F$ an $E$-distribution denoted by  $ \tilde{\xi}_{F, \chi }\in (C_c^{\infty} (F) \otimes E)'$. 
If  $f \in C_c^{\infty} (F) \otimes E$, let $ f_1 \in C_c^{\infty} (G) \otimes E$ be such that its restriction to $F$ is equal to $f$. Then  \beq \label{xiF} \tilde{\xi} _{F,\chi}  (f)=  \tilde{\xi}_\chi( f_1).\eeq  From this one concludes that:
$ \chi\to \tilde{\xi}_{F, \chi}$ is a weakly holomorphic family of $E$-distributions on $F$.
Similarly, if $\Omega$ is a left $P$-invariant open set of $F$, the restriction of  $ \tilde{\xi}_{F, \chi}$ to $ \Omega$, $( \tilde{\xi}_{F, \chi})_{\vert \Omega}$,  is a weakly holomorphic family of $E$-distributions on $\Omega$. 

 Let us  make two observations.
 By [HH], Lemma 2.4 one has:
 \ber \label{HH24} Each  parabolic subgroup of $G$ contains a $\si$-stable maximal split subgroup of $G$. \eer 
 One has also: 
 \ber  \label{a0am} The morphism of Lie groups from $X(M)$ to $X(A_0)$ given by the restriction has finite kernel. Hence one may view the Lie algebra $ (\a_M)'_\C $ of $X(M)$ as  a subspace of the  Lie algebra $( \a_0)'_\C$ of $X(A_0)$ ( cf also. (\ref{iden})). \eer 
 
 The following Lemma is one of the two  key Lemmas in the article. 
\begin{lem}\label{xiPH}Let $P$ be a parabolic subgroup of $G$. Let $A_0$ be  a $\si$-stable maximal split torus of $G$ contained in $P$. Let $M$ be the Levi subgroup of $P$ which contains $A_0$ and  let $U$ be the unipotent radical of $P$.

Let   $(\delta,E)$ be a finite length smooth representation of  $M$ and let $X$ be a complex subtorus of  $X(M)$.  Let  $O$ be a nonempty open subset of $X$. \\ Let $\chi \mapsto\xi_\chi $ be  a weakly holomorphic family, depending  on $\chi \in O$, of $E$-distributions on $PH$. We assume moreover that the family $(\xi_\chi)$ is non identically zero and that for every $\chi\in O$, $ \xi_\chi$ is $H$-invariant on the right and $\delta_\chi$-covariant under $P$ (cf. Section \ref{covd}). \\
 (i) The elements of the  Lie algebra $\b$ of $X$,  viewed  as  a subspace of $( \a_0)'_\C$  as in (\ref{a0am}), are $\si$-antiinvariant. \\
(ii) Moreover if $\b$ contains a strictly $P$-dominant element, $\nu$ (i.e. such that $( \nu, \aa)>0$ for every  root, $\aa$,  of $A_M$ in the Lie algebra of $P$),  then $P$ is a $\si$-parabolic subgroup of $G$. 
\end{lem}
 \begin{proof}
 There is no restriction to assume that  $O$ is connected and $ \xi_\chi $ is never equal to zero. Also, by translation by an element of $O$, one can  assume that $O$ contains 1.
 The group $P\times H$ acts on $PH$ by $$(p,h)  g= pgh^{-1}, g \in PH, p\in P, h\in H .$$
 Hence  $PH$ is a homogeneous space under $P\times H$ homeomorphic to $(P\times H) / Diag (P\cap H)$ by the map $(p, h)\to ph^{-1}$(cf. [BD] Lemma 3.1 (iii)).
 \\Let us denote  the trivial character of $H$ by $\varepsilon _H$.
 Let us define  the mean value operation $M_{P\cap H} $  which sends $C_c^\infty (P\times H)\otimes E $ to $C_c^\infty (PH)\otimes E $: 
 $$(M_{P\cap H} f)( ph):= \int_{P\cap H} f(px, x^{-1} h) \> d_lx, f\in C_c^\infty (P\times H)\otimes E,$$
  where $d_lx$ is a left invariant Haar measure on $P\cap H$. We define $\xi'_\chi$ by: $$\xi'_\chi :=\xi_\chi\circ M_{P\cap H}.$$   It is a weakly holomorphic family of $E$-distributions on $P\times H$ which are $( \delta_P^{1/2}\otimes \delta_\chi) \otimes \varepsilon_H$-left covariant under $P\times H$.  

 From  Lemma \ref{Tcov}, there exists $\eta_\chi \in  E'$ , such that  for all $e \in E$ and for all  compact open subgroup, $J$,  of $P$ which fixes $e$ under $ \delta$, one has :  $$\langle \xi'_\chi, f \rangle  = \int_{P\times H} \langle \eta_ \chi , \delta(p) f(p, h) e \rangle \> d_lp\> dh
 = vol(J) \langle \eta_\chi ,e \rangle , $$ 
 where 
 $ f = 1_J  \otimes f_1\otimes e \in C_c^\infty (P\times H)\otimes E$ and  $f_1$ is a smooth function on $ H$ with compact support and with integral equal to 1. 

As the family  $ (\xi'_\chi)$ is weakly holomorphic, this implies that:
  \ber  \label{etahol}$(\eta_\chi)$ is a weakly holomorphic family of linear forms on $E$.\eer
  Also from Lemma \ref{Tcov}, one deduces:\beq  \label{deltaeta} (\delta_\chi)' ( p) \eta_\chi = \delta_{P\cap H} ^{-1}Ê(p) \delta_{P} ( p) \eta_\chi , p \in P\cap H . \eeq
 Let  $a\in A_M\subset A_0$ and let $b$ be equal to $ a\si(a)\in A_0 \cap H\subset M\cap H$.
 Then, one has:  $$   (\delta_\chi) '(b)   \eta_\chi= \delta_{P\cap H} ^{-1}Ê(b) \delta_{P} ( b) \eta_\chi. $$ 
 But $ (\delta_\chi) ' (b)=  \chi^{-1} (b)\delta'(b) $ so that one has:
 \beq  \label{delt'} \delta'(b)   \eta_\chi= \delta_{P\cap H} ^{-1}Ê(b) \delta_{P} ( b)  \chi(b) \eta_\chi . \eeq
 Let us consider the parabolic subgroup $M\cap \si(P)$ of $ M$. Its Levi subgroup containing $ A_0$ is equal to $M \cap \si(M)$. 
 Let us prove that: \beq \label{etain} \eta_\chi \in (E _{M\cap \si(P)})' ,\eeq 
 where $E_{M\cap \si(P)}$ is the Jacquet module of $(\delta,E)$ with respect to the parabolic subgroup $M\cap \si(P)$ of $M$.
 
Let  $x\in M\cap \si(U)$. From [BD] Proposition 2.1 (iv), there exists  a unipotent subgroup $V'$ of $P\cap \si(P)$,    $h\in H \cap  V'$  and  $ y \in U\cap \si(P)$  such that  $x=yh$.
Then,  as $ \delta_\chi $ is trivial on $U$, one has:$$ (\delta_\chi)' ( x)\eta_\chi=  (\delta_\chi)' ( yh) \eta_\chi= (\delta_\chi)'  (h) \eta_\chi.$$
As $h \in V'\cap H$ and $V'$ is a unipotent subgroup of $P$, $h$ is an element of a union of compact subgroups of $P\cap H$.
 From the fact that  a continuous  positive character on a topological group is trivial on compact subgroups, one deduces:   
$$ \delta_{P\cap H} ^{-1}Ê(h) \delta_{P} ( h)=1. $$
Similarly, as $ \chi$ is unramified, one has  $$ \chi (h)=1. $$ 
Hence, from  (\ref{deltaeta}),  one gets : $$(\delta_\chi) '(h) \eta_\chi= \eta_\chi.$$
As $\chi$ is an unramified character of $ M$ and $x$ is an element of the unipotent subgroup of $M$, $ M \cap \si(U)$, one has $ \chi(x)=1$. From the previous discussion, one sees:
$$ \delta' (x) \eta_\chi= \eta_\chi, x \in M\cap \si(U),$$
so that one has: $$\eta_\chi \in ( E_\delta' )^{M\cap \si(U)}, $$
which proves  (\ref{etain}).

The Jacquet module $E_{M\cap \si(P)}$ being an $M\cap \si(M) $-module of finite length and $b$ being an element of the center of $M\cap \si(M)$, the number of generalized eigenvalues of $\delta'(b)$  on $(E_{M\cap\si( U)})'$ is finite. 
From (\ref{delt'}), one deduces  that the map  $\chi \mapsto  \chi(b)$ is constant on $O$, and equal to $1$ as $O$ is connected and contains $1$. In other words we have proved:
\beq \label{cas} \chi(a\si(a))= 1, \chi \in X, a\in A_M. \eeq 
\\As $X\subset X(M)$, one  views  an element $\nu$ of $\b$ as an element of $( \a_0)'_\C$, which vanishes  on $ \a^M_0$.  
One deduces from (\ref{cas}) that:
\beq \label{lxsi}\nu(X+\si(X))= 0, X \in \a_M,  \nu \in \b.  \eeq Hence the restriction of $\nu$ to $ \a_M + \si(\a_M)$ is $\si$-antiinvariant. 

Recall that we have choosen a scalar product on $\a_0$ which is invariant under the Weyl group of $(G, A_0)$ and by $\sigma$. Then  $ \nu \in \b \subset (\a'_M)_\C$ is zero on the orthogonal to $\a_M$, hence also on the orthogonal to $ \a_M + \si(\a_M)$. Hence $\nu$  is $\si$-antiinvariant.
This proves (i).

Now let us assume that $ \nu$ is an element of $ \b$ which is  strictly $P$-dominant. Then, with the notations of (\ref{p-l}), one has  $P=P_\nu$. One sees, from the antiinvariance of $\nu$ that $\si(P)= P_{-\nu}$ which is clearly opposite to $P$. Hence $P$ is a $\si$-parabolic subgroup of $G$. 
\end{proof}
\begin{lem}\label{xixP}
Let $P=MU$ be a parabolic subgroup of $G$ and let $(\delta,E)$ be a smooth representation of finite length of $M$. Let $X$ be a complex subtorus of $X(M)$ and let $O$ be a nonempty subset of $X$. Let $\chi \mapsto \xi_ \chi$ be a weakly  holomorphic family  of $ H$-forms on $i^G_PE_\chi$ defined for $\chi \in O$.  
Let $F$ be the union of the supports of the  $ \xi_\chi$, $\chi \in O$. As these supports are left invariant under $P$ and right invariant under $H$ and as there are only a finite number of $(P,H)$-double cosets, $F$ is closed.  We call $F$ the support of the family. 

 Let $A$ be a maximal split torus in $M$. Let $\Omega$ be a $(P,H)$-double coset of $G$ open in $F$.  Then one can choose $x\in \Omega$ such  that $A_x:=x^{-1}.A$ is  a $\si$-stable maximal split torus contained in $x^{-1}.P$. For such an $x$,  $x^{-1}.M$ is  the Levi subgroup of $x^{-1}.P$ which contains $ A_x$. The conjugation by  $x^{-1}$ induces a map $\chi \to x^{-1}\chi$ from $X$ to  the subtorus $x^{-1}.X$ of $X(x^{-1}.M)$. 
\\(i) Then the Lie algebra of $x^{-1}.X$ appears as a subspace of $ (\a_x)'_\C$ made of antinvariant elements by $\si$. \\ (ii)  Moreover if  $X$ contains  a strictly $P$-dominant element, $x^{-1}.P$ is a $ \si$-parabolic subgroup.\\(iii) With the assumption of (ii),
 one can choose $x$ such that $A_x$ is a maximally $\si$-split $\si$-invariant  maximal split torus in the $\si$-stable Levi subgroup of $x^{-1}.P$.
\end{lem}
\begin{proof}Let $x$ be an element of $\Omega$. 
First $x^{-1}.P$ contains a $\si$-invariant maximal split torus of $G$ (cf. [HH] Lemma 2.4). Two maximal split tori of a parabolic subgroup are conjugate by an element of this subgroup (see [BoTi], Proposition 4.7 and Theorem 4.21).  Hence changing $x$ into  $px$, for a suitable $p\in P$, one can assume that $A_x$ is $\si$-invariant.    
The restriction to $\Omega$ of the induced $E$-distribution $ \xi_{F, \chi}$ by $\xi_\chi$  on $F$ (see Section \ref{covd}) will be denoted  by $ \zeta_\chi$. Then, from what follows (\ref{xiF}),    $\ll({x^{-1}})\zeta_\chi  $ (cf. (\ref{deflx}) for the definition of $\ll(x)$) satisfies the hypothesis of the preceding lemma with $P$ changed in $x^{-1}.P$, $A_0$ in $A_x$ and $X$ in $x^{-1}.X$.  Then (i) and (ii)  are an immediate consequence  of the preceding lemma.
 
Let us prove (iii).   We are in the case where $x^{-1}.P$ is a $\si$-parabolic subgroup whose $\si$-stable Levi subgroup contains 
 a maximally $\si$-split $\si$-invariant  maximal split torus of $G$. This implies (iii). \end{proof}
 \subsection{Generic Basic Geometric Lemma for $H$-forms and the role of $\si$-parabolic subgroups }
\begin{prop} \label{mwpsi} Let  $P=MU$ (resp., $P'=M'U'$)  be a  $\si$-parabolic subgroup and let $A$ (resp., $A'$) be a maximally $\si$-split maximal split torus of $M$ (resp., $M'$). Let $A_\si$ (resp., $A'_\si$) be the maximal $\si$-split torus of $A$ (resp., $A'$).  We denote the  identity  component  of the set of  elements of $X(M)$ which are antiinvariant under $\si$ by $X:=X(M)_\si$.   Let $(\xi_\chi)$ be a weakly holomorphic family  of $H$-forms on $i^G_P E_\chi$  defined for $\chi$ in an open subset $O$ of $X$. With the notations of Theorem \ref{xipw}, let $w$ be an element of $W(M'\backslash G/M)$. Let us assume that  $\xi_{P'^-, w}\not=0$ where $\xi= {\xi}_{\chi_0}$ for some element $\chi_0$ of $O\cap O'$, where $O'$ is as in  Theorem \ref{xipw}. Then one may change our choice of $w$ in its class in  $P'\backslash G/P$ in such a way that:\\
(i)  One has  $A':=w.A$.\\
(ii) The group $w.P$ is a $\si$-parabolic subgroup of $G$ with $\si$-stable Levi subgroup $w.M$ and $M' \cap w.P$ is a $\si$-parabolic subgroup of $M'$.\\ 
(iii) The groups $P'_w$   and $w.P_w$ are  $\si$-parabolic subgroups of $G$.
\\(iv)  One has the equality  $w^{-1}. A'_\si= A_\si$.    
\\(v) The group $w^{-1} . P'$ (resp., $w^{-1}.P'^-$) is a $\si$-parabolic subgroup of $G$ with $\si$-stable Levi subgroup $w^{-1}. M'$ and    $M\cap w^{-1}.  P'$ (resp., $M\cap w^{-1}.  P'^-$)  is a $\si$-parabolic subgroup of $M$. \\
The groups  $P_w, P'_w, \tilde{P}_w, \tilde{P}'_w$  are $\si$-parabolic subgroups of $G$. \end{prop}
\begin{proof}
It will be more convenient for the proof of this Proposition to denote $A'$ by $A_1$ in order to avoid  too many $'$. 

First, as $w\in W(M'\backslash G/M)$, one has $w.A=A_1$. Moreover:
\ber \label{liea}  The Lie algebra of $X$ is equal to the space  $ (\a_M)_\C'^{-\si}$ of $\si$-antiinvariant elements of $ (\a_M)_\C '$ that one can view as a subspace of $\a'_\C$ (cf. (\ref{iden}) and (\ref{am'})).\eer
\\We define $X':=\{w\chi_{\vert M' \cap w.M}\vert \chi \in X\} $ which is closed in $X(M'\cap w.M)$ (cf., (\ref{Xclosed})).  By looking to differential, one sees that on   an open neighborhood of $\chi_0$ in $ O\cap O'$,   the map $ \chi\mapsto w\chi_{\vert M' \cap w.M}$,   from $X$ to $X'$, is an isomorphism  whose inverse, defined on $O'' \subset X'$, will be be denoted $\chi' \mapsto w^{-1}\chi'$.  
Let us show  that the family $\chi' \to  (\xi_{w^{-1}\chi'})_{P'^-,w}$ defined on $O'' $ is weakly holomorphic. It is enough to prove that  for all $e$  element of the compact realization of $i^{M '}_{M '\cap w.P} wj_{M\cap w^{-1}.P'} E_\chi$, the map   $\chi\mapsto \langle (\xi_{\chi})_{P'^-,w}, e_{w\chi} \rangle$  is holomorphic on  $O$. We choose a compact open subgroup of $G$, $J$,  with  a $\si$-factorization for $(P',P'^-)$ such that $e$ is fixed by $J_{M '}$.  One starts by using  (\ref{defj}) and then one uses the fact that $(e_J \xi_\chi) $ is an holomorphic family of  $J$-fixed vectors. One deduces from Proposition \ref{betaw}, that $\langle (\xi_{\chi})_{P'^-,w}, e_{w\chi} \rangle $ is weakly holomorphic if  for all $v$ in the compact realization of $i^G_{P} \cc{E}$, the map
 $\chi \mapsto \langle \beta_{\chi,w}(\cc{v}), e_{w\chi} \rangle $ is holomorphic on $O'$,  where $\beta_{\chi,w}(\cc{v}) = (\cc{v}_\chi)_{P'^-,w}$. This  follows easily from the definition of $\beta_{\chi,w}(\cc{v}) $ (cf. Proposition \ref{betaw}) and from the holomorphy properties of the intertwining integrals.  
 
 As $P$ is a $\si$-parabolic subgroup,  $X$ contains strictly $P$-dominant elements and $X'$ contains  a strictly $M'\cap w.P$-dominant element. The conclusion of the preceding lemma, applied to $M'$ instead of $G$ and $X'$ instead of $X$,  asserts that there exists $m' \in M'$   such that $A_2:=m'w.A=m'.A_1$ is a  maximally $\si$-split, $\si$-stable, maximal split torus of $M'$,   $Q:=m'. (M'\cap w.P)$ is a $\si$-parabolic subgroup of $M'$, $ m'.X' $  and hence  $m'w(\a_M)'^{-\si} \subset ({\a}_2)'$ is made of  $\si$-antiinvariant elements, where $(\a_M)'^{-\si}$ is the space of antiinvariant elements of $(\a_M)'$.    Then the half sum of the roots of $A_2$ in the Lie algebra of $Q$, $\rho_Q$, satisfies:
 \beq \label{rhoq} \rho_Q \in (\a_2)'^{-\si} . \eeq \\Two maximally $\si$-split $\si$-stable maximal split tori of $M'$ are conjugate by an element of $M'$ which conjugates their maximal $\si$-split tori, because two maximal $\si$-split tori are conjugate, and two maximal split tori in the centralizer of a maximal $\si$-split torus are conjugate. So one can choose $m''$ in $M'$ such that $m''.A_2=A'$ and $m'' (\a_2)'^{-\si}= (\a_1)'^{-\si}$. Then   from (\ref{rhoq}), one sees that  the parabolic subgroup of $M$, $Q':=(m''m'w.M) \cap P'$,   contains $A_1$ and the half sum of the roots of $A_1$ in its Lie algebra $\rho_{Q'}$ is an element of  $(\a_1)'^{-\si}$ . Using (\ref{p-l}),  one sees that it is a $\si$-parabolic subgroup of $M$. One will change $w$ to $m''m'w$. Hence one has: \ber \label{wam} $w.A=A_1, w (\a_M)'^{-\si} \subset ( \a _1)'^{-\si}$ . \eer
 Let us show that $w$  satisfies   (ii) and (iii).
Let $\rho_P \in \a_M'\subset {\a}'$ be the half sum of roots of $A$ in the Lie algebra of $P$. Similarly we define  $\rho_{w.P_w}$, $\rho_{P'_w}$ and $ \rho_{P'\cap w.M}$ with respect to $A_1$. These are elements of $(\a_1)'$.
One has : \beq \label{rho} \rho_{w.P_w}= \rho_{ P'\cap w.M} +w\rho_P. \eeq  
As $P$ is a $\si$-parabolic subgroup of $G$, one has $\rho_P\in (\a_M)'^{-\si}$. From (\ref{wam}), one deduces that: \ber \label{rhoa} $w\rho_P \in (\a_1)'^{-\si}$ is $\si$-antiinvariant.   \eer
It follows that $w.P$ is a $\si$-parabolic subgroup of $G$.
The group $w.M$ is the Levi subgroup of $w.P$ which contains $A_1=w. A$,  whose Lie algebra is the sum of the $A_1$-weight spaces  for weights $ \aa$ which are equal to zero or to an $A_1$-root $\aa$ such that $(w \rho_P, \aa)=0$. 
As $w \rho_P$ is $\si$-antiinvariant, one sees that $w.M$ is $\si$-invariant.
As $P'$ is a $\si$-parabolic subgroup of $G$, $\si(P')\cap P'= M'$, which implies:
$$\si(P'\cap w.M)\cap P'\cap w.M= M' \cap w.M.$$
As $M' \cap w.M$ is the Levi subgroup of the parabolic subgroup $P'\cap w.M$ of $w.M$ which contains $A'$, this implies that $P'\cap w.M$ is  a $\si$-parabolic subgroup of $w.M$ with $\si$-stable Levi subgroup equal to $w.M \cap M'$. Hence \ber \label{rhop}$\rho_{ P'\cap w.M}$ is $\si$-antiinvariant.\eer From (\ref{rho}), (\ref{rhoa}),  (\ref{rhop}), one sees that $\rho_{w.P_w}\in (\a_1)' $ is $\si$-antiinvariant.
It follows from (\ref{p-l}) that $w.P_w$  is a $\si$-parabolic subgroup of $G$. One sees easily that its $\si$-stable Levi subgroup is $ M' \cap w.M$. \\
Similarly one proves that $P'_w$ is a $\si$-parabolic subgroup of $G$, by using the equality:$$\rho_{P'_w}= \rho_{ M'\cap w.P}+ \rho_{P'}$$
and that $P'$ (resp., $M'\cap w.P$) is a $\si$-parabolic subgroup of $G$ (resp., $M'$), as $w.P$ is  a $\si$-parabolic subgroup of $G$ from (\ref{rhoa}) and   $w.M$ is its  $\si$-stable Levi subgroup.

Altogether we have found a choice of $w$ which satisfies (i), (ii) and (iii). 
 We will  modify our preceding choice of $w$ to get one which will satisfy also  (iv). 
 
  Let $w$ be as above. 
Then  $(w.P) w H= w PH$ is open in $G$. Moreover  $w.P$ is a $\si$-parabolic subgroup of $G$,  $(A_1)_\si $ is a maximal $\si$-split torus contained in $w.P$, hence contained  in its $\si$-stable Levi subgroup  $w.M$. As any $(w.P,H)$-open orbit has a representative which is $(A_1)_\si$-good (cf. (\ref{WG})), there exists $p'=w.p$ with $p\in P$ such that $ w':= (w.p) w= wp$ satisfies $w'^{-1}. (A_1)_\si $ is $\si$-split.  But $w'^{-1}. (A_1)_\si \subset  p^{-1}. A \subset P$. As $w'^{-1}. (A_1)_\si $  is $\si$-split,  it is $\si$-stable hence contained in $M=P\cap \si(P)$. Hence $ w'^{-1}. (A_1)_\si$ is a maximal $\si$-split torus in  $M$.  Then, as all maximal $\si$-split tori in $M$ are conjugate (cf. [HH], Proposition 1.16), we can choose  an element $m$ of $M$,   such that $ w''= w'm$  satisfies ${ w'' }^{-1}.({ A_1}) _\si   =  A_\si$. Then ${w''}^{-1}. A'$ is contained  in the centralizer of $ A_\si $ which is contained in $M$. Hence  one can choose $ m_1$ element of this centralizer such that $ w_1 := w'' m_1$ is such that $w_1^{-1}.A_1= A$ and 
$w_1^{-1}.(A_1)_\si = A_\si $. Hence,  as $\rho_P \in \a'^{-\si}$, $w_1 \rho_P $ is $\si$-antiinvariant and $w_1 .P$ is a $\si$-parabolic subgroup of $G$. As above, one sees also that it implies that $w_1.M$ is its $\si$-stable Levi subgroup.  Then, as above, one sees that $w_1$ satisfies (ii) and (iii). 
As $P'$ is a $\si$-parabolic subgroup of $G$, $ \rho_{P'} $ is an element of  $(\a_1)'^{-\si}$ and  one gets  $ w_1^{-1}\rho_{P '} \in \a'^{-\si}$. One sees, as above,   that $w_1^{-1}. P'$ is a $\si$-parabolic subgroup of $G$ with $\si$-stable Levi subgroup equal to $w_1^{-1}. M'$.  This implies easily that $M\cap w_1^{-1}. P'$ is a $\si$-parabolic subgroup of $M$ with $\si$-stable Levi subgroup equal to $M\cap w_1^{-1} .P'$. 

Then the assertions on $P_{w_1}, P'_{w_1}, \tilde{P}_{w_1}, \tilde{P} '_{w_1}$ follow easily.
Thus $w_1$ has the required properties.
\end{proof}
\subsection{Intertwining integrals and support of families of $H$-forms}
 An ordered pair $(P=MU, P'=MU')$ of parabolic   (resp., $\si$-parabolic subgroups) of $G$, is said to be adjacent (resp., $\si$-adjacent) if  there is a unique reduced $A_{M}$-root (resp., reduced  $(A_M)_\si$-root), $\aa$, which is  positive for $P$ and negative for $P'$.   We denote  by $A^\aa$ the  group of $\F$-points  of the identity  component of the kernel of $\aa$ in $\underline{A}_M$ (resp., $(\underline{A}_M)_\si$) and by $M^\aa$  the centralizer of $A^\aa$ in $G$. The group $P^\aa$ \index{$P^\aa$} generated by $P$ and $P'$ is a parabolic (resp. $\si$-parabolic) subgroup of $G$ with Levi (resp., $\si$-stable Levi subgroup) $M^\aa$ and unipotent radical $U^\aa$ contained in $U\cap U'$. It is easy to see that $\aa$ is $P$-simple.
 
 A minimal string of parabolic (resp.,  $\si$-parabolic) subgroups  of $G$ between two parabolic (resp., $\si$-parabolic) subgroups of $G$, $P=MU$, $P'=MU'$, is a sequence $(P_i)_{i=0,\dots, r} $ of parabolic (resp., $\si$-parabolic subgroups) of $G$,   such that $P_0=P, P_r=P'$ and $(P_i, P_{i+1}) $ is adjacent (resp., $\si$-adjacent) for $i=0, \dots, r-1$. Such a string always exists (cf. [KnSt],  before Theorem 4.2 for  parabolic subgroups and it works in a same manner for $\si$-parabolic subgroups).
 
 The next lemma is the second  key lemma mentioned in the Introduction. It was suggested by  a geometric result of Matsuki (cf. [M], Lemma 3).
\begin{lem} \label{skey} Let $P=MU$  and $Q=LV$ be  two $ \si$-parabolic subgroups of $G$, with $P \subset Q$. Let $(\delta,E)$ be a smooth irreducible representation of $M$. Let $P'=MU'$ be another  $\si$-parabolic subgroup of $G$  such that $(P,P')$ is $\si$-adjacent, and let $\alpha$  be the unique reduced $(A_M)_\si$-root  which is positive for $P$ and negative for $P'$.  One assumes that  the restriction $ \alpha_{\vert a_L^{-\si}}$ of $\aa$,  to  $\a_L^{-\si}$ is nonzero, where $\a_L^{-\si}$ is the subspace of  elements of $\a_L$ antiinvariant under $\si$.\\ We denote   the identity component of the set   of $\si$-antiinvariant elements of $X(L)$ by $ X(L)_\si$.\\ Let $\chi\mapsto \xi_\chi $ be   a weakly  holomorphic family of  $H$-forms on $i^G_P E_ \chi $ defined for $ \chi $ in  an open subset,  $O$, of $X(L\vert M)_\si:=\{\chi_{ \vert  M}\vert \chi \in X(L)_\si\}$. Let us assume   that the support of every  $\xi_\chi$ has an empty interior in $G$.  \\ Then one has the following:\\
(i) The set $O'$ of $\chi \in O$ such that $A(P,P', \delta_ \chi)$ has no pole   is an open and dense subset of $O$. \\ (ii) If $\chi \in O' $,  the support of $\xi_\chi  \circ  A(P,P', \delta_ \chi)$, which is an $H$-form on $i^G_{P'}  E_\chi$,  has an empty interior in $G$.\\
(iii) Let $Q'$ be the $\si$-adjacent $\si$-parabolic subgroup of $G$ determined by the $Q$-simple  $ (A_L)_\si $-root $\aa_{\vert  a_L^{-\si }}$. Then $P'\subset Q'$ and $(Q,Q')$ are $\si$-adjacent.

\end{lem}
\begin{proof}


The fact that $O'$ is dense follows from the  fact that $\aa$ restricted to $\a_L^{-\si}$ is nonzero and from Lemma \ref{hoinv}. This proves (i).  

Let us show that the  union of the $(P,H)$-double cosets $Px_i H$ which are open in the support $F$ of the family  $\xi_\chi, \chi \in O'$ is dense in this support.  In fact if it was false the complementary of the union of the $Px_iH$ would contain a $(P,H)$-invariant subset open in $F$.  Then it contains an open $(P,H)$-double coset (see e.g.  [BD] Lemma 3.1 (ii)).  This is impossible due to our definition of the $Px_iH$.  This proves our claim.

From Lemma \ref{suppAf},  the support of the family ($\xi_\chi\circ A(P,P', \delta, \chi))_{ \chi  \in O'}$ is contained in $cl(P^\aa F)$, hence in the union of $cl(P^\aa x_iH)$  as for $A, B,$  subsets of $G$,  $cl( A cl(B) )= cl(AB)$ and $cl(A\cup B)= cl(A) \cup cl(B)$.   Hence it suffices to show  that:  \ber \label{forallx} For all  $x=x_i$  one has $P^\aa xH$ with empty interior.\eer  Let $A$ be a maximally $\si$-split torus   $\si$-stable maximal split torus in $M$, which is automatically maximally $\si$-split $\si$-stable maximal split torus in   $G$ as $P$ is a $\si$-parabolic subgroup. As any parabolic subgroup contains a $\si$-invariant maximal split torus,  we   may and we will  choose $x$ in its double $(P,H)$-coset  such that $x^{-1}. A$ is $\si$-stable.
The Lie algebra of $X:= X(L\vert M)_\si$ is identified with $ (\a'_L)_\C^{-\si}$. We  can apply   Lemma \ref{xixP} (ii) to the family $(\xi_\chi)_{\chi \in O'}$ and one sees, using (\ref{am'}),  that:  \beq \label{xasub} x^{-1}(\a_L^{-\si})   \subset (x^{-1}\a)^{-\si} . \eeq
 One sets: $P_x=x^{-1}.P, P_x'=x^{-1} .P'$ etc..
Here we  denote  the Lie algebra of the algebraic group $\underline{G}$ such that $G= \underline{G}(\F)$ by $\underline{g}$. By abuse of terminology  we will say that $\underline{g}$ is the Lie algebra of $G$ (notice that this notion is different from the one used in Section \ref{Lie}). We use a similar terminology  for the subgroups of $G$ which are the group of  $\F$-points of a subgroup of   $\underline{G}$ defined over $\F$.

 Let $ \underline{g}(\aa)$ (resp., $\underline{g} (-\aa)$) be the Lie algebra of $U \cap M^\aa$ (resp., $ U'\cap M^\aa$). Similarly let $ \underline{g}(x^{-1}\aa)$ (resp., $\underline{g} (-x^{-1}\aa)$) be the Lie algebra of $U_x \cap M_x^\aa$ (resp., $ U_x'\cap M_x^\aa$).   One has:
$$ \underline{p}_x^\aa= \underline{g}(-x^{-1}\aa) + \underline{p}_x . $$
We fix $ Y_x= x ^{-1} Y\in x^{-1}(\a_L^{-\si})$, where $ Y$ is a strictly $Q$-dominant element in $(\a'_L)^{-\si}$ which is canonically identified to $\a_L^{-\si}$ (cf. (\ref{am'})).  Then an $A_x$-root $\beta $ is such that $ \beta(Y_x)>0$  if and only if the  corresponding root space is contained in $ \underline{v}_x$.  But $ \underline{g} (\aa)Ê\subset \underline {v}$, as $P\subset Q$ and $\alpha_{\vert \a_L} \not = 0$.  So one has $(x^{-1} \alpha)(Y_x)>0$. 
Let us prove that:
\beq  \label{hplusp} \underline{h}+Ê\underline{p}_x= \underline{h}+ Ê\underline{p}_x^\aa. \eeq 
The only thing to prove is that $ \underline{g}(-x^{-1}\aa) \subset  \underline{h}+ \underline{p}_x$. Let $\beta$ be a root of $A_x $ in $ \underline{g}(-x^{-1}\aa)$, so that $\beta(Y_x)<0$ by what has been said above. Let   $Z$ be an element of the corresponding weight space  in $ \underline{g}$. Let $\si\beta$ the $A_x$  be the root $\beta\circ \sigma$. Then $ \si\beta(Y_x)= \beta( \si(Y_x))$. But  as $  x^{-1}(\a_L^{-\si})$  is contained in the space of $\si$-antiinvariant elements (cf. (\ref{xasub})), one has $\si(Y_x)=-Y_x$ and  $\si\beta(Y_x)= -\beta(Y_x)$ is strictly positive. Hence $\si(Z)$ is an element of  $\underline{v}_x\subset \underline{p}_x$. Consequently one has:
$$Z=( Z+ \si(Z))-\si(Z)\in  \underline{h}+ \underline{p}_x, $$
which proves (\ref{hplusp}).

Let us assume that  $P_x^\aa H$ has a nonempty interior. We will prove that it leads to a contradiction. From (\ref{WG}) and Lemma \ref{sixP} (iv), one deduces that  $P_x^{\aa}$ is a $\si$-parabolic subgroup of $G$.
Hence (\ref{a)}) implies that $\underline{P}_x^\aa$ is a $\si$-parabolic of $\underline{G}$  and (\ref{b)}) implies $$ \underline{h} + \underline {p^\aa_x} = \underline{g}.$$  Then, together with (\ref{hplusp}),  it implies that $\underline{h} + \underline {p_x} = \underline{g}$. Hence (cf. Lemma \ref{orbiteouverte}), $\underline{H} \underline{P_x}$ is open. From Lemma \ref{hpouvert},  one sees that $PxH$ would be open in $G$. This is a contradiction with our hypothesis on the support of $\xi$, which implies (\ref{forallx}). We have thus proved (ii).

 The assertion on $Q'$ in (iii) being clear, this proves   the lemma. 
  \end{proof} 
\begin{lem} \label{416}
We keep the  notations  of  the previous lemma, except that  $P'$ is not necessarily adjacent to $P$ . We assume that every $(A_M)_\si$-root which is  positive for $P'$ and negative for $P$ satisfies $ \aa_{\vert \a_{L}^{-\si}}\not=0$. Then the same conclusion as in (i)  and (ii)  of the previous lemma is valid.
\end{lem} 
\begin{proof}
Let $P_0, \dots P_r$ be a minimal string of $\si$-parabolic subgroups between $P$ and $P'$. We will prove by induction on $i$ that for $\chi \in O'$, the support of  $A(P_i, P, \delta_\chi)\xi_\chi$ has an empty interior in $G$ and we will define  $\si$-parabolic subgroups of $G$,  $Q_0=Q, \dots, Q_r$,   with $\si$-stable Levi subgroup $L$, such that: 
\\The family $(\xi^i_\chi): = (\xi_\chi \circ A(P_i, P, \delta_\chi)) $  is such that $(P_i, P_{i+1})$, $Q_i$ satisfies the hypothesis of the previous lemma, by the induction hypothesis. 
The Lemma follows by using this  previous lemma.
\end{proof}
\section{Families of $H$-forms on  representations induced  from $\si$-parabolic subgroups and $B$-matrices}
\setcounter{equation}{0}
\subsection{Families of $H$-forms on representations  induced from $\si$-parabolic subgroups}
Let $P=MU$ be a $\si$-parabolic subgroup of $G$. Let $(\delta,E)$ be a smooth representation of $M$ with finite length. 
Let  ${\mathcal W}_M^G$  be a set of representatives of the open 
$(P,H)$-double cosets  as in  (\ref{WG}): here $A_\si$ is a maximal $\si$-split torus of $M$ .  As remarked in Lemma \ref{sixP}, if $ x\in {\mathcal W}_M^G$, $x^{-1}.P$ is a $\si$-parabolic subgroup of $G$. We define \beq \label{edh} E'(\delta,H)
:=  \oplus _ {x \in {\mathcal W}_M^G} E'^{M \cap x. H}. \eeq 
We have the following mild generalization of [BD], Theorem 2.8,  that one gets in an entirely similar way (see the comments following the statement).
Notice that this statement is true assuming only that the characteristic of $\F$ is different from 2, as it is shown by the examination of the proof.

Let us give a definition.
\ber \label{defb}If $B$ is  a split torus we denote the real vector space  denoted $\a_B$ in (\ref{ag}) by $\b$. \eer
The lattice $\LL(B)$ ( cf. (\ref{LLA}) for its definition) is identified with   a subset of $\b$ (cf. (\ref{45}) and 
 if $B$ is a subtorus of the split torus $A$ one has $\b\subset \a$.

We will  use the identification of the Lie algebra of  $X(M)_\si$ and $(\a_M^{-\si})_\C$ given by a suitable scalar product on $\a_M$ (see (\ref{iden} ) and  (\ref{am'} )).

\begin{prop} \label{defxi}
Let us assume that we are given a $\si$-split torus $B \subset A_M$,  and a complex subtorus  $X$ of $X(M)_\si$.   We assume that  the Lie algebra  of $X$ is identified with  $\b_\C\subset (\a_M^{-\si})_\C$ and  that the lattice $ \Lambda(B)$  contains strictly $P$-dominant elements.  \\(i) Let $\chi\in X$ . 
 Let us denote  the subspace of elements of  $i^G_PE_\chi$ whose support is contained in the union of the open $(P,H)$-double cosets by $J_\chi$.  
\\ There is a canonical linear isomorphism between $E'(\delta,H)$ and $J_\chi'^H$  which associates to 
 $\eta \in  E'^{M \cap x.H}$,  the element $\xi'(P,\delta_\chi,\eta) $ of  $J_\chi'^H$ defined by:
$$\xi'(P,\delta_\chi,\eta)(\varphi) = \int_{(H\cap x^{-1}.M) \backslash H} \langle \varphi(xh),\eta \rangle \> d{h},
  {\varphi \in J_\chi}.
  $$  
  (ii) There exists  an open dense subset,  $O_0$,   of $X$ such that for $\chi\in O_0$,  $\xi'(P, \delta_\chi,  \eta)$ extends uniquely to an $H$-form, $\xi(P, \delta_\chi,  \eta)$,   \index{$\xi(P, \delta_\chi,  \eta)$} on $i^G_PE_\chi$. In particular for $ \chi \in O_0$, every $H$-form on $i^G_P E_\chi$ whose restriction to the  open $(P,H)$-double cosets is equal to zero vanishes.
  \\(iii) Moreover  there exists a polynomial function, $b$, on $X$, such that  for every $ v\in i^K_{K\cap P}E $ and $\eta\in E'(\delta,H)$, $\chi \mapsto b(\chi)\langle \xi(P, \delta_ \chi,  \eta), v_\chi \rangle $ extends to a polynomial function on $X$.
\end{prop}
The main point for this generalization of [BD], Theorem 2.8, is the generalization of [BD], Lemma 2.5 with our $X$. But it is straightforward by changing  $A$ to $B$ in its proof and the projection $p_\a$ to the orthogonal projection $p_\b$ on $\b$. 

Then, we will see that one has,  for a suitable normalization of the measures in the Proposition \ref{defxi}: 
\ber  \label{restxi} For all $\eta= ( \eta_x)_{x \in {\mathcal W}_M^G}$,  for all $x\in {\mathcal W}_M^G$, $\cc{r}_M((i^G_P \delta_\chi )'(x) \xi(P, \delta_\chi,  \eta))=\eta_x$.\eer 
For $x=1$, it follows from section \ref {covd} and  [BD], Equation (2.33).   Now the assertion reduces to the one for $x=1$ applied to the $x.H$-fixed linear form $(i^G_P\delta_\chi)' (x)\xi(P, \delta_\chi,  \eta)$.
\subsection {$B$-matrices} 

\begin{prop} \label{Bmat}
We keep the notations of the preceding proposition. Let $Q$ be a $\si$-parabolic subgroup of $G$ with $\si$-stable Levi subgroup equal to $M$.  There exists  a  rational function 
$\chi \mapsto B(P,Q,\delta_\chi)$ \index{$B(P,Q,\delta_\chi)$, $B$-matrices} on $X(M)_\si$  with values in  $\mathrm{End} _\C (E'(\delta,H))$   such that one has the equality of rational maps on  $ X(M)_\si$: 
$$\xi(Q, \delta_\chi,  \eta) \circ  A(Q, P,\delta_\chi)= \xi(P, \si, B(P,Q, \delta_\chi)\eta), \> \eta \in E'(\delta,H).$$
More precisely,  
let $ b_A $ (resp., $b_\xi$)  \index{$b_A, b_\xi$}be a nonzero polynomial function on $X(M)_\si$, such that for all $v$ in $i^K_{K\cap P} E$,   $\chi \mapsto b_A(\chi) (A(Q,P, \delta_\chi) v )$  (resp., and  for  all $\eta \in E'(\delta,H) $,   $\chi\mapsto b_\xi(\chi) \langle  \xi(Q, \delta_\chi, \eta),v_\chi \rangle  $) is polynomial on $X(M)_\si$.
Then for all  $\eta\in E'(\delta,H)$, the map  $\chi \mapsto (b_A b_\xi (\chi)) B(P,Q, \si_\chi)\eta$ is a  polynomial map  on  $X(M)_\si$ with values in  $E'(\delta,H)$.\\
\end{prop}
\begin{proof}
From Proposition \ref{defxi}, for $ \chi $ element of a dense open subset, $O$, of $X(M)_\si$, there is a unique $\theta(\chi) \in E'(\delta,H)$, such that 
   $$b_\xi ( \chi) \xi(Q, \delta_ \chi, \eta) \circ (b_A (\chi) A(Q, P,\delta_ \chi))=
\xi(P,
\delta  _
\chi, \theta(\chi) ). $$ Let us show that the map 
$\chi \mapsto \theta(\chi)$ is polynomial in  $\chi\in X(M)_\si$.  Let $ \theta(\chi)_x , x \in {\mathcal W}_M^G$ be the components of  $ \theta(\chi)$. 
First, let us prove: 
\ber
\label{toute}  For all  $e\in E$, $\langle \theta(\chi)_1, e\rangle   $  is polynomial in  $\chi\in X(M)_\si$.\eer 
Let   $J$ be  a compact open subgroup of $K$ which has  a  $\si$-factorization for $(P,P^-)$ and such that $e$ is fixed by $J_M$. Then  one has  from  (\ref{xivintbis})
: $$\langle \theta (\chi)_1,e\rangle   = \mathrm{Vol} ((K\cap P)J_H)^{-1} \langle  \xi(Q, \delta _\chi,  \theta(\chi)),v _{e, \delta_\chi }^{P,J}\rangle  $$
$$=   \mathrm{Vol} ((K\cap P)J_H)^{-1} b_\xi ( \chi) \langle \xi(Q, \delta_ \chi, \eta) , b_A (\chi) A(Q, P,\delta_ \chi) v_{e, \delta_\chi}^{P,J} \rangle .$$
It follows from (\ref{vsi}), that  the restriction to $K$ of $v_{e, \delta_\chi}^{P,J}$   is independent of $\chi$.
Hence, from the properties of $b_A $,   one sees that $\chi \to v(\chi):=b_A(\chi) A(Q, P,\delta_ \chi) v_{e, \delta_\chi }^{P,J} $  is polynomial in the compact realization. Hence (\ref{toute}) follows.

Let $x \in {\mathcal W}_M^G$. One applies (\ref{toute} ) to $ (i^G_P\delta_\chi)' (x) \xi( P, \delta_\chi, \eta)$ by changing  $\si$ to $\si_x$.  One concludes that  the map 
$\chi \to \theta(\chi)$ is polynomial in  $\chi\in X(M)_\si .$ 
 Then $$B(Q,P, \si_\chi)\eta:= (b_\xi b_A)^{-1}(\chi) \theta(\chi) $$
 satisfies the required properties.  \end{proof}
  
\section{ Main Theorems} 
\subsection{}
Let us prove
\ber \label{mxh}Let $P=MU$ be a $\si$-parabolic subgroup of $G$ and let $A_\si$ be a maximal $\si$-split torus of $M$.  If   $x, x' \in G$ are  $A_\si $-good and  $PxH=Px'H$ then $ x '= mxh$ with $m \in M$.\eer
One has $x'= muxh$ with $m \in M$, $u\in U$, $h\in H$.  
 As $x'$ is $A_\si$-good, $x'^{-1}. M$ is $\si$-stable,  which implies that $x^{-1}. u^{-1}.M$ is $\si$-stable. 
Let $M'=x^{-1} .M$ which contains $x^{-1} A_\si$. Then $ M'$ is the  $\si$-stable Levi subgroup of the $\si$-parabolic subgroup $P':= x^{-1}.P$ (cf. Lemma \ref{sixP}). Hence, as 
$ x^{-1}. (u^{-1}.M )\subset P'$, one has $ x^{-1}. (u^{-1}.M ) = x^{-1} .M$, which implies that $u^{-1}.M= M$ . Hence,  for all $m \in M$, $u^{-1}mu \in M$  which implies  $ m^{-1} u^{-1} mu \in M$. Then $ m^{-1} u^{-1} mu $ is element of $U\cap M$ hence is equal to 1. Hence $u$ commutes with every element of  $M$. This is possible only if $u=1$. Hence $x'=m xh$, which proves  (\ref{mxh})
\begin{lem} \label{restsup}
(i) Let $P=MU, R=LV\subset P $ be two  $\si$-parabolic subgroups of $G$. Let $Q$ be equal to $R\cap M$, which is a $\si$-parabolic subgroup of $M$ with $\si$-stable Levi subgroup equal to $L$. Let $A_\si$ be a maximal $\si$-split torus of  $M$. If  $x \in M$ and  $\Omega = Q x(M\cap H) $ is open in $M$ then  $RxH $ is open in $G$.\\
(ii) Let $x, x'\in M$ which are $A_\si $-good. If $RxH=Rx'H$, one has $Q x (M\cap H)= Q x'(M\cap H)$. \\
(iii) Let $(\delta,E)$ be a smooth representation of $L$ and let $\xi$ be an $H$-form on $i^G_RE$.  If $\mathrm{Supp} (\xi)\subset G$ has an empty interior, the  same is true for the support of  $\cc{r}_M \xi\in (i^M_Q E)'$.
\end{lem}\begin{proof}
(i) The first  claim on $Q$ is clear as $Q\cap \si(Q)= R\cap \si(R) \cap M \cap \si(M)= L$. 
 
  The map  $ P \times H  \to PH$, $(p,h)\mapsto ph$,  is open  (cf. [BD], Lemma 3.1)   and  $U\Omega $ is open in  $P$. Hence  $U\Omega  H $ is open  in  $PH$ and also in $G$, as $PH$ is open in $G$.  But $U \Omega  H\subset RxH$. It implies that $RxH$ has a nonempty interior, hence is open.  
\\(ii) By (\ref{mxh}) applied to $R$, one has $ x'= lxh $ with $ l\in L, h\in H$. Hence $h$ is element of $M\cap H$ and (ii) follows.
\\(iii) Let us prove that with our hypothesis, 1 does not belong to the support of $\cc{r}_{M}\xi$. The hypothesis implies that $\cc{r_L}\xi =0$.
From Proposition \ref{defrM} (iii), one deduces that $\cc{r}_L{\xi}= \cc{r}_L (\cc{r}_{M} \xi))= 0$, which is equivalent to  1 is not in the support of $\cc{r}_{M} \xi$. This proves our claim. 

Now let $x$ as in (i). Changing of representative in $Q x (H\cap H)$, we may assume (cf (\ref{WG})) that $x$ is $A_\si$-good . Then $P$ is a $\si_x$-parabolic subgroup of $G$  (cf. Lemma \ref{sixP}). Then one applies the first part of the proof to $(i^G_R\delta)'(x) \xi$, which is fixed by $ x.H$, replacing $\si$ by $\si_x$. This implies that $x $ is not element of the support of $\cc{r}_M\xi$. 
\end{proof}

\begin{lem} Let $P=MU$, $R=LV$ be two $\si$-parabolic subgroups of $G$ with $R \subset P$. Let $Q= M\cap R$, which is a $\si$-parabolic subgroup of $M$.  Let $A$  be a  $\si$-stable maximally $\si$-split maximal split torus of $L$. We choose the set  $ {\mathcal W}^G_M$ such that its elements  are $A_\si$-good. Then for all $y  \in {\mathcal W}^G_M$, $y^{-1}.M$ is $\si$-stable,  $y^{-1}.Q$ is a $\si$-parabolic subgroup of $y^{-1}.M$ with  $\si$-stable Levi subgroup equal to $y^{-1}.L$. For all  $y \in {\mathcal W}^G_M$,  we choose ${ \mathcal W}^{y^{-1}.M} _{y^{-1}. L}$ such  that its elements are $y^{-1}. A_\si$-good. Then    $\mathcal{W}:=\cup_{y \in  \mathcal{W}^G_{M} }y\mathcal{W}^{y^{-1}.M} _{y^{-1}. L}$ is made of $A_\si$-good elements, the union being disjoint, and  may  be taken as  $\mathcal{W}^G_{L}$. 
\end{lem}
\begin{proof}
 We will show that $\mathcal{W}$ is a set of representatives of the  open $(R,H)$-double cosets in $G$. 
 
Let $x=yz, x'=y'z' \in \mathcal{W}$ with $y,y'\in  \mathcal{W}^G_{M} , z\in {\mathcal W}^{y^{-1}.M} _{y^{-1}. L}, z'\in\mathcal{W}^{y'^{-1}.M} _{y'^{-1}. L}$
Then $x$, $x'$ are $A_\si $-good. Hence $ RxH$ and $Rx'H$ are open. Moreover  $PxH=PyH$, $Px'H= Py'H$ as $z\in y^{-1}.M$, $z'\in y'^{-1}.M$. Let us show that:
\ber \label{rzh} The equality $RxH= Rx'H$ implies $y=y' $ and $z=z'$  . \eer Our hypothesis implies $PxH=Px'H$. By what we have just seen, it implies $PyH=Py'H$. Hence one has $y=y'$. 

First, let us assume  that $y=y'=1$. 
Then  $z,z'\in M$ are $A_\si$-good, $RzH= Rz'H$, and, by (\ref{mxh}),  $z'=  l zh$ where $l \in L $  and $h\in H$. Hence one has $h \in M\cap H$ . If $Q=R \cap M$,  $z, z'$ determine the same $(Q, M\cap H)$-double coset. This implies that $z'=z$. Hence this proves (\ref{rzh}) if   $y=y'=1$. 

For  the general case,  we apply  this to $y^{-1}.L$ and $ y^{-1}.R$ in order  to prove (\ref{rzh}). 
So $\mathcal{W}$ is a set of representatives of certain open $(R, H)$-cosets. 

Reciprocally, let $RxH$ be an open  $(R,H)$-double coset in $G$. We may assume that $x$ is $A_\si$-good. Let $Q= R\cap M$. Let us show that there is  $x'\in \mathcal{W}$ with $  RxH= Rx'H$. First there exists  $y \in  \mathcal{W}^G_{M} $ such that $PxH=PyH$.   Let us assume that $y=1$. Then, by (\ref{mxh}), one has $x= m h$ with $m\in M$, $h\in H$. Hence one can assume $x\in M$. As $x\in M$ is $A_\si$-good, $Qx(H\cap M)$ is open in $M$ and $Qx(H\cap M)=Qx'(H\cap M)$ with $x' \in \mathcal{W}^M_{L}$. Then $RxH= Rx'H$ as wanted. In general one changes $P$ to $P'= y^{-1}.P$, 
$R$ to $R'= y^{-1}.R$. Hence one has $ P'y^{-1}xH= P'H$, as $PxH=PyH$.  Then one uses our last result.  
\end{proof}

\subsection{ $\si$-exponents of $j_{P^-} \xi$. }
\begin{defi}\label{expo} Let $(\pi, V)$ be a smooth representation of $G$ of finite length. Then it is a finite direct sum of generalized eigenspaces under  $A_{G, \si}:= (A_G)_\si$.  If $ \nu $ is a character of $A_{G, \si}$, let us denote  the corresponding generalized eigenspace of $V$ by 
$V(\nu)$and  the restriction to $V(\nu)$  of any element $\xi$ of $V'$ by $\xi(\nu)$, which might be extended to an element of $V'$ , which is zero on the other generalized eigenspaces also denoted  $\xi(\nu)$. If $Ê\xi\in V'^H$, $\mathrm{Exp}( \xi)$ will denote the subset of $\nu$ such that $\xi(\nu)$ is nonzero.  The elements of $\mathrm{Exp}(\xi)$ are called the $A_{G, \si}$-exponents or $\si$-exponents of $\xi$. \index{$\mathrm{Exp}(\xi)$, set of $\si$-exponents of $\xi$} \end{defi}
 
\begin{theo} \label{2}
Let $P=MU$ (resp.,  $P'=M'U'$)  be a $\si$-parabolic subgroup of $G$  and let  $A$ (resp., $A'$) be a maximally $\si$-split $\si$-stable maximal split torus of $M$ (resp., $M'$).  Let $(P'\backslash G/P)_\si$ be the set of $(P',P)$-double cosets in $G$ having a representative $w$ such that $w.A=A'$, $w.A_\si=A'_\si$. We denote   a set of representatives of $(P'\backslash G/P)_\si$ with this property by $W(M'\backslash G /M)_\si$  \index{$W(M'\backslash G /M)_\si$} and we assume that the set $W(M'\backslash G /M)$ defined in (\ref{wm/g/m}) contains $W(M'\backslash G /M)_\si$.

Let $ (\delta,E)$ be an irreducible smooth represention of $M$, let $\eta$ be an element  of the space  $E'(\delta,H)$ (cf. (\ref{edh}) for the definition of this space).   Let $O'$ be as in Theorem \ref{xipw}, with $X=X(M)_\si$.  In particular it is open and dense in $X(M)_\si$. Let $O_0$ be the open dense subbset of $X(M)_\si$ from Proposition \ref{defxi} (ii) for $X=X(M)_\si$.  Let $\chi$ be an element of the open and dense subset $O '\cap O_0$ of  $X(M)_\si$ and let $ \xi= \xi(P, \delta_\chi, \eta)$.\\
(i)   Let $w\in W(M'\backslash G/M) $. Then   if  $\xi_{P'^-,w}$ is nonzero ,  one has $w\in W(M'\backslash G/M)_\si$.\\
(ii) If $w\in W(M'\backslash G/M)_\si$, then
$Q=M \cap w^{-1} .P'$ is a $\si$-parabolic subgroup of $M$  and    $L= M\cap w^{-1}.M'$ is its $\si$-stable Levi subgroup. We  introduce $\mathcal{W}^G_L$ as  in the preceding lemma.\\
 If $y\in \mathcal{W}^G_M$  and   $z \in \mathcal{W}^{ y^{-1}. M}_{y^{-1}.L}$,   we define $z':= yzy^{-1}\in M $ and $x= yz$.
 
  Then  $\delta '(z ') \eta_y$ is $M\cap x.H$-invariant and $Q$ is a $\si_x$-parabolic subgroup of $M$.  Hence
 $j_{Q^{-}} \delta'(z') \eta_y $ is defined.  \\
 (iii)  
One writes  $$E_{Q}= E^+_Q \oplus E^0_Q,$$ 
where $E^+_Q$ is the sum over  $ y \in \mathcal{W}^G_{M}$ and $ z \in \mathcal{W}^{ y^{-1}. M}_{y^{-1}.L} $   of the $A_{L,\si}$-weights space corresponding to the set of exponents   of $j_{Q^{-}} \delta'(z') \eta_y $, where $E^0_Q$ is the sum over the other weights. 

Then, for $\chi$ belonging to  an open dense subset, $O''$ of $O_0\cap O'$, hence also open and  dense in $X(M)_\si$,  the $A_{M ', \si}$-exponents of  $ \xi_{P'^-, w}$  are of the form  $(w{ \chi^+ })_{\vert A_{M', \si}}$ where $\chi^+ $ is an  $A_{M', \si}$-eigenvalue   of  $(E^+_Q)_\chi$ .
\end{theo}
\begin{proof}
(i)  If  $\xi_{P'^-,w}$ is nonzero, one has $w\in W(M'\backslash G/M)_\si$ by Proposition \ref{mwpsi}. This proves (i). 

Let us prove (ii). From (i),  $w^{-1}.P'$ is a $\si$-parabolic subgroup of $G$, as $w$ is $A'_\si$-good, which contains $A$. Hence $Q$ is a $\si$-parabolic subgroup of $M$.\\ Let $x,y,z$ as in the statement of (ii).
The linear form  $\delta'(z') \eta_y$ on $E$  is $M\cap x.H$-invariant as $\eta_y$ is $M\cap y.H$-invariant and $z'y= yz=x$. By construction of $\mathcal{W}^G_M$, $x$ is $A_\si$-good. Hence $Q$, $Q^-$ are opposite  $\si_x$-parabolic subgroup of $M$ (cf. Lemma \ref{sixP}). This  proves (ii). 

 Let us prove (iii). One defines projections   $ p^+_Q$ and $p^0_Q$ of $E_{Q}$ onto $E^+_Q$ and $E^0_Q$ corresponding to the decomposition $E_Q= E_Q^+\oplus E_Q^0$. 
This defines, by induction,  projections on the space $(i^G_{P_{Q^-}} (E_\chi) _{Q} )'$ that we will denote in the same way. Notice that $P_{Q^-}$ is equal to $\tilde{P}_w$.
With these conventions, we define  for $\chi\in O_0\cap O'$:
$$\xi^+_Q= 
  p^+_Q (\cc{j} _{Q^-} \circ \xi).$$
  Similarly one defines  $\xi^0_Q $. Then
$$\xi^+_{P'^-,w} := \cc{r}_{M'} \circ\>  ^tA(w.  \tilde{P}_w,  \tilde{P}'_w, w. j_{Q} \delta_\chi) \circ  \lambda (w) \xi^+_Q . $$  is  a well defined element of $V_w'$ where $V_w:= i^{M'}_{M'\cap w.P} (wE^+_\chi)$. Similarly one defines  
$\xi^0_{P'^-,w}$ 
 so that  one has: $$\xi_{P'^-,w}= \xi^+_{P'^-,w}+\xi^0_{P'^-,w}.$$
  We will prove that $\xi^0_{P'^-,w}=0$. 
  We first study the restriction of $\xi^0_Q$ on the  open $(M'\cap w.P,M'\cap H)$-orbits. From Lemma \ref{restri}, one sees that for $x \in \mm{W}^G_L$: 
  $$\cc{r}_{L} ( (i^G_{P_{Q^-}}(\delta_\chi)_Q)'  (x)  \cc{j}_{Q^-} \circ \xi)= j_{Q^{-}}( \cc{r}_M ((i^G_P\delta_\chi)' (x)   \xi   )) .$$ 
  But  if $x= yz$ with  $y\in \mathcal{W}^G_M$  and   $z \in \mathcal{W}^{ y^{-1}. M}_{y^{-1}.L}$, one has $x= z'y  $ with $z'=y z y^{-1} \in M $ and, with the notations of (\ref {fxiomega}), one has:$$ \cc{r}_M((i^G_P \delta_\chi)' (x)    \xi)= f_\xi (x)=  \delta'_\chi(z')  f_\xi(y)= \chi(z')^{-1}\delta'(z') \eta_y$$
  Hence, one has: $$\cc{r}_{L} (  (i^G_{P_{Q^-}}(\delta_\chi)_Q)'  (x)   \cc{j}_{Q^-} \circ \xi)=\chi(z')^{-1} j_{Q^-}(\delta'(z') \eta_y)$$ and :  $$\cc{r}_{L} ( i^G_{P_{Q^-}}(\delta_\chi)_Q)'  (x)  \xi^{0}_Q)= 0.$$ 
From the preceding lemma,   one sees that the support of $\xi^0_Q $ has an empty interior.
  By structural transport the same is true for $ \ll(w) \circ \xi^0_Q$. One can apply Lemma  \ref{416} with $L$ replaced by $w.M$, $M$ replaced by $M'\cap w.M$ as it follows  immediately from the definition of $\tilde{P_w}, \tilde{P'_w}$ in Proposition \ref{betaw}. Then by Lemma \ref{416} and Lemma \ref{restsup}, one sees that the support of $\xi^0_{P'^-,w} $ has an empty interior. Let us see that one  can apply   Proposition \ref{defxi} to $M'$, $M'\cap w.P$,  $w.(A_M)_\si$                                   and $X'=\{(w \chi)_{\vert M' \cap w.M}\vert \chi\in X(M)_\si \}$ instead of $G$, $P$,  $B$ and $X$. As $wA_\si=A'_\si$, one sees that $M'\cap w.P$ is a $\si$-parabolic subgroup of $w.M$ with $M'\cap w.M$ as $\si$-stable Levi subgroup.  Then $X'\subset X(M'\cap w.M)_\si$ and its Lie algebra is equal to $w(\a'_M)_\C$ and if $\ll$ is any strictly  $P$-dominant element  in $\LL(A_M)_\si$, $w.\ll$ is a strictly-$(M'\cap w.P)$ dominant  of $\LL(w.(A_M)_\si$. Then Proposition \ref{defxi} implies that $ \xi^0_{P'^-,w}$ is equal to zero for $\chi$ belonging to an open dense subset, $O''$ of $O_0\cap O'$. Hence one has:
 $$\xi_{P'^-, w}= \xi^+_{P'^-, w}.$$
  As  $\xi^+_{P'^-, w} $ is a linear form on $i^{M'}_{M'\cap w.P} w((E_Q^+)_{\chi_{\vert M\cap w^{-1} M'}})$, (iii) follows. 
  \end{proof}
  An $H$-form $\xi$ on a smooth admissible representation of $G$, $(\pi, V)$  is said to be $H$-cuspidal if $j_{Q^-}\xi= 0$ for all proper $\si$-parabolic subgroups of $G$. We denote  the space of cuspidal $H$-forms on $V$ by $V'^H_{cusp}$.\\
 We define  $$E'(\delta,H)_{cusp}:= \oplus _{x \in \mathcal{W}^G_M } (E')^{M \cap x.H}_{cusp}. \index{$E'(\delta,H)_{cusp}$} $$ 
 \begin{theo} \label{75}
 Let  $ \eta\in E'(\delta,H)_{cusp}$.  Let $w \in W(M'\backslash G /M)_\si$. With the notations of the preceding theorem, let $\chi$ be an element  of the open and dense subset $ O''$ of $X(M)_\si$
   and let  $\xi= \xi(P, \delta_\chi, \eta)$.\\ 
  (i) If $  \xi_{P'^-, w }\not=0$, one has $ M \cap w^{-1}.P'= M$. \\
  (ii) If  $ M \cap w^{-1}.P'= M$, one  defines  $ \mathcal{W}^G_{w.M}:= w \mathcal{W}^G_{M}$ whose elements are $w.A_\si=A'_\si$-good.   One can choose   $\mathcal{W}^{M'}_{w.M}$ such that for    $ y'\in \mathcal{W}^{M'}_{w.M}$, there exists a unique $y \in \mathcal{W}^{G}_{w.M}$ and $h\in M'\cap H$ with $ y'=yh$. \\
 If $ \eta \in E'(w\delta, H)$, $p_{M'}\eta:= (\eta_y) _{y'\in \mathcal{W}^{M'}_{w.M}} $ is an element of $(wE)' (w\delta, M' \cap H)$.
 With these notations
  $$  \xi_{P'^-, w } = \xi (M'\cap w.P,  w \delta_\chi,  p_{M'} B (\tilde{P}'_w,w. P, w \delta_\chi) \eta),$$ where the $B$-matrices are defined relative to $ \mathcal{W}^G_{w.M}$. \\  (iii) Let us assume that $M'=M$ and $A'=A$.  Then $B(P', P, \delta_\chi)\eta$ is an element of $  E'(\delta,H)_{cusp}$. Hence $B(P', P, \delta_\chi)$ restricts to an endomorphism of $  E'(\delta,H)_{cusp}$.
    \end{theo}
  \begin{proof}(i) Let us assume that $M\cap w^{-1} P' \not=M$. From the definition of $ E'(\delta,H)_{cusp}$ and of $E^+_Q$ in Theorem \ref{2} (iii) one will  see that $E^+_Q$ is zero. In fact,   with the notation of Theorem \ref{2} (ii), as $\eta_y$ is $M\cap y.H$ cuspidal, one sees by a direct computation that $\delta'(z') \eta_y$ is $M \cap x.H$-cuspidal and $j_{Q^-} \delta'(z') \eta_y= 0$. This implies  that $E^+_Q=\{0\}$. 
From the preceding theorem one has $  \xi_{P'^-, w } =0$.
\\(ii) Let us assume $ M \cap w^{-1}.P'= M$. Then the $\si$-stable Levi subgroup of $P'_w$ is equal to $w.M$. 
  
  If $y'\in M'$ is $A'_\si$-good,  then $ P'_w y'H$ is open and  $P'_w y'H=P'_wyH$ for some element $y$ of $\mathcal{W}^G_{w.M}$. In particular $y$ is  $A'_\si$-good.
From (\ref{mxh}), one sees that $y'= ly h$ with $ l\in M'\cap w.M$, $h\in H$.  Changing $y' $ to $l^{-1}y'$, one may assume that $y'= yh $ for some $h \in M'\cap H$. This allows us to choose $\mathcal{W}^{M'}_{w.M}$  as in (ii).
 From Theorem \ref{xipw}
  $$  \xi_{P'^-, w } = \cc{r}_{M' }\> ^tA( w.P ,\tilde{P}'_w, w (\delta_\chi)) \ll (w) \xi. $$
  One has  choosen $ \mathcal{W}^G_{w.M}:= w\mathcal{W}^G_{M}$ whose elements are $w. A_\si$-good. 
  Then, one sees that $(wE)'(w\delta,H)_{cusp} = E'(\delta,H)_{cusp}$, as $(E')^{M \cap x.H}_{cusp}= ((wE)')^{(w.M) \cap(wx. H)}_{cusp}$.
  Also one has  $\ll(w) \xi(P, \delta_\chi, \eta)= \xi (w.P, w(\delta_\chi), \eta). $
  Hence  from the definition of $B$-matrices  one has:
  $$ \xi_{P'^-, w } = \cc{r}_{M '} \> ^t\xi (\tilde{P}'_w, w(\delta_\chi), B (\tilde{P}'_w,w. P, w (\delta_\chi))\eta ))$$
   In order  to prove the equality of  the theorem for $\chi$ in an open dense subset of $X(M)_\si$, one can use  Proposition \ref{defxi} applied to $M'$ instead of $G$, $M'\cap w.P'$ instead of $P$ , to the complex torus $X= \{w\chi_{\vert M'\cap w.M}\vert \chi \in X(M)_\si\}$ and to $B= w.(A_M)_\si$. Then it suffices  to prove the equality of  the values  of  $\cc{r}_{M '} \xi (P'_w, w(\delta_\chi), B (P'_w,w. P, w (\delta_\chi)\eta ))$ 
   and $\xi (M'\cap w.P,  w \delta_\chi,  p_{M'} B (P'_w,w. P, w \delta_\chi),\eta)$ at every element, $y'$, of  $\mathcal{W}^{M'}_{w.M}$. This is easily seen from the (\ref{restxi}) and from the fact that $y'= yh$ for an element $y$ of $\mathcal{W}^G_{w.M}$ and $h\in H\cap M'$.
  \\ (iii) We take $w= 1$ in (ii). Hence, one has:
   $$\xi_{P'^-,1}= B(P',P, \delta_\chi) \eta\in V(\delta, H)$$
  From (i) and the transitivity of the constant term, one sees that $\xi_{P'^-, 1} \in E'(\delta,H)_{cusp}$ is $M\cap H$-cuspidal.\\
 \end{proof}
   If $P=MU$ is a $\si$-parabolic subgroup of $G$, let us denote   $(A_M)_\si$ by    $A_{M, \si}$. Let $^+ \a'_{M, \si} $ (resp $^{++} \a'_{M, \si} $)
    be the set of $\ll \in \a_{M,\si}'$ which  are linear combinations of roots of $A_{M, \si}$ in the Lie algebra of $U$ with coefficients greater or equal to zero ( resp., greater than zero).
  \begin{defi} \label{xiR}
  Let $(\pi, V)$ be a finite length   smooth representation of $G$ and $\xi$ an $H$-form on $V$. Then $\xi$ is said to be tempered (resp., discrete)
  if and only if for every $\si$-parabolic subgroup of $G$, $P=MU$,  every  exponent, $\chi$,  of $\xi_{P^-}$  is such that $\mathrm{Re} ( \chi)  $ is an element of 
   $^+ \a'_{M, \si} $ (resp $^{++} \a'_{M, \si} $).
   If $\xi$ is a tempered $H$-form we define its weak constant term  \beq\label{wct} \xi_{P^-}^w= \sum_{ \chi \in \mathrm{Exp} (\xi_{P^-}), \mathrm{Re} (\chi)=0} \xi_{P^-}( \chi)\eeq
    \end{defi}
    Hence a discrete $H$-form is a tempered $H$-form such that its weak constant term is zero for all proper $\si$-parabolic subgroups of $G$.
    Notice that Kato-Takano (cf. [KT2]) showed that, if $\pi$ is irreducible and has a unitary central character,  an $H$-form is discrete if and only if  its generalized coefficients are square integrable modulo the center.
  \begin{lem}
    (i) If $\xi$ is a tempered $H$-form on $V$, then $ \xi_{P^-}^w$ is a tempered $M\cap H$-form on $V_P$.\\
    (ii) If $Q$ is $\si$-parabolic subgroups of $M$ and $R=QU$, one has:
    $$ \xi_{R^-}^w = ( \xi_{P^-}^w ) ^w_{Q^-}$$
    \end{lem}
    \begin{proof}
    (i)  and (ii) follow  from the transitivity of the constant term (cf. [D3]).
    \end{proof}
   If $(\delta,E)$ is a smooth unitary irreducible representation of $M$, let $E'^{M\cap H}_{disc}$ be the space of discrete  $M\cap H$-forms on $E$ and let:
    $$E'(\delta,H)_{disc} = \oplus _{x \in \mathcal{W}_M^G }E'^{M \cap x. H}_{disc}. \index{$E'(\delta,H)_{disc}$} $$
    
  \begin{theo}  \label{33}
 With  the notations of Theorem \ref{2}, let us assume  $\delta$ is unitary, $ \eta\in E'(\delta,H)_{disc} $ and $\chi \in O ''\cap X(M)_u$. Then one has:
 \\(i) The $H$-form $\xi(P, \delta_\chi, \eta)$ is tempered.
 \\(ii)  If  $w \in W(M'\backslash G/M)$ is not in $W(M'\backslash G/M)_\si$ or if  $M\cap w^{-1}.P'$ is  distinct from $M$, then $ \xi^w_{P'^-,w}=0$.  Otherwise, with the notations of Theorem \ref{75} (ii):
 $$  \xi^w_{P'^-, w } = \xi (M'\cap w.P,  w \delta_\chi, p_{M'} B (\tilde{P}'_w,w. P, w \delta_\chi) \eta),$$ where the $B$-matrices are defined relative to  $ \mathcal{W}^G_{w.M}$.
 \\(iii) Let us assume that $M'=M$ and $A'=A$.  Then  $B(P', P, \delta_\chi)$ restricts to an endomorphism of $  E'(\delta,H)_{disc}$.
 \end{theo}
 \begin{proof}
 Let us use the notations of Theorem \ref{2} (ii).
 Using the criteria of  of Kato and Takano (see above), one sees by `` transport de structure'' that   $\delta'(z') \eta_y$ is  discrete. Moreover  $Q$ is a  $\si_x$-parabolic subgroup  of $M$.
Then  (i) follows from  Theorem \ref{2} (iii) and from our definition  of discrete forms. 

The proof of (ii)(resp., (iii)) is analogous to the proof of (i) and  (ii) (resp., (iii) ) of the preceding theorem. 
\end{proof}
\section{Constant term of Eisenstein integrals}
 If $f$ is a smooth function on $H \backslash G$ the constant term $f_P$ of $f$ along $P$ has been defined in [D3], Section 3.3. It generalizes the definition in[L], Proposition 2 to general smooth function. It is a smooth function on $M\cap H \backslash M$. 
 
  If $(\pi, V)$ is a smooth  admissible representation of $G$, $\xi$ is an $H$-form on $V$ and $v$ is an element of $V$, let  $c_{\xi, v}$ be the generalized coefficient defined by:
$$c_{\xi,v}(Hg)= \langle \xi, \pi(g)v \rangle$$
Let $P=MU$ be a $\si$-parabolic subgroup of $G$. If $f= c_{\xi, v}$, the constant term $f_P$ of $f$ along $P$ has been defined in [L], Proposition 2 (i) (see [D3] Section 3.3, for a definition for general smooth functions). Then (cf. [L], Proposition 2 (ii)) one has:
\beq \label{cct}(c_{\xi,v})_P= c_{\xi_{P^-},  v_P} , 
\eeq
where $\xi_{P^-}=j_{P^-} \xi, v_P= j_Pv.$
\begin{defi}\label{atemp}
We define $\mathcal{A}_{temp}( H \backslash G)$ to be the set of functions of the type $c_{\xi,v}$ for a finite length smooth representation $(\pi,V)$ of $G$ and a tempered $H$-form on $V$. It is easily seen to be a vector subspace of the space of smooth functions on $H \backslash G$. If $f$ is such a generalized coefficient $f_P^w$ will denote the generalized coefficient $c_{\xi^w_{P^-},  v_P}$. 
It is naturally deduced from the constant term $f_P$ as in the definition of $\xi^w_{P^-}$ (cf. Definition \ref{xiR}), hence it does not depend of the presentation of $f$ as a generalized coefficient.
\end{defi}
\begin{defi}
 Let $P=MU$  be a $\si$-parabolic subgroup of   $G$ and let $(\delta, E)$ be an irreducible smooth representation of $M$. Let $\eta \in E'(\delta,H)$ be such that $\xi(P,\delta, \eta)$ is defined. Then if $v\in  i^G_PE $, one defines 
 an element $ E^G_P(\eta\otimes v)$ of  $ C^\infty (H \backslash G)$ by: 
$$E^G_P( \eta\otimes v)(g) =\langle \xi(P, \delta ,  \eta), (i^G_P\delta )(g) v \rangle, g\in G. $$
Then, from Proposition \ref{defxi}, there exists a nonzero polynomial function $b$ on $X(M)_\si$  such that  for all  $v \in i^K_{K \cap P} E, \eta \in E'(\delta, H)$, $g\in G$ the map    $\chi \mapsto b(\chi) E^G_P( \eta \otimes v_\chi)(g)$ is polynomial  in $\chi \in X(M)_\si$.
By bilinearity, we define  $E^G_P( \phi)$ similarly, where $ \phi $ is an element of $ E'(\delta,  H)\otimes  i^G_{ P}E$. 
\end{defi}
\begin{lem} \label{partcass}
Let $W(M'\vert G \vert M)_\si$ the set of elements of $W(M' \backslash G/ M)_\si ) $ such that $ w.M \subset M'$. Let us assume that $W(M'\vert G \vert M)_\si$ is non empty. Let $s\in W(M'\vert G \vert M)_\si$. We choose  $K'=s.K$ (see the beginning of Section \ref{gblemma}). This choice together with the choice of $K$ determines $\beta_\chi$. 
Let $w$ be an element of $W(M'\vert G \vert M)_\si$.
Let $\gamma(G\vert M')$ be the constant defined in [Wal] after the proof of I. 1.(3),  with the normalizations of measures determined by $K'$. 
For $ \chi$ belonging to an open dense subset of $X(M)_\si$
one has, with the notations of (\ref{gamma}): $$ \gamma_{\chi,w}= \gamma(G\vert M')^{-1}\aa_{\chi,w}.$$
\end{lem}
\begin{proof}

It is enough, by ``transport de structure'',  to prove the result for $s^{-1}.P'$ instead of $P'$, as the Jacquet modules for $P'$ and $s^{-1}.P' $ of $i^G_P \delta_\chi$ are canonically isomorphic. So we may assume that $1\in W(M'\backslash G/ M)_\si$ and $A'=A$, $K'=K$.  
By Lemma \ref{V=} applied to $M'$ instead of $G$, one sees that for $\chi $ belonging to  a dense subset  $O''$ of $X(M)_\si$, $V_{\chi,w} $ is irreducible for all $w \in W(M'\backslash G/ M)_\si$. Hence $\gamma_{\chi,w} = \gamma(P, w, \chi) \aa_{\chi,w}$ for an element $\gamma(P, w, \chi)$ of $\C^*$. 
Notice that in the definition  of $\beta_{\chi,w}$  (resp. $\alpha_{\chi,w}$) the  intertwining integrals  which occur are relative  to  parabolic subgroups of $G$ with  Levi subgroup equal to $M$. Also only the  operations $j_M$ and $\cc{j}_M$  occur and they are equal to the    identity.  Then the 
 proof of the lemma is  identical to the proof of Proposition V.1.1 in [Wal] (see Equations V.1 (2), (3) and (4)), where parabolic subgroups have to be replaced by  $\si$-parabolic subgroups. 
\end{proof}

\begin{theo} \label{cteis}  We keep the notations of Theorem \ref{2}.  Let $\chi $ be an element of the dense open subset $O''$ of $X(M)_\si$. 
 If  $w\in W(M'\vert G\vert M)_\si $, one defines a linear map   $  C(w, P', P, \delta_\chi) $ from  $E'(\delta, H)  \otimes i^G_PE_\chi  $  to 
$ (wE)'(w \delta, M'\cap H) \otimes i^G_{P_w} wE_\chi  $ by : $$ C(w, P', P, \delta_\chi) =\gamma(G\vert M')^{-1}   B({\tilde P'}_w, w.P, w(\delta_\chi))\otimes (A(P'_w, w.P,  w(\delta_\chi) \ll(w)) .$$
Then,  if   $ \phi \in   E'(\delta,H)_{cusp}  \otimes i^G_P E_\chi$ (resp.,  $E'(\delta,H)_{disc}  \otimes i^G_P E_\chi$ and $\delta$ and  $\chi$ are  unitary), one has:\\
$E^G_P(\phi)_{P'}=0  $ (resp.,  $E^G_P(\phi)^w_{P'}=0$) if  $W(M'\vert G\vert M)_\si$  is empty.\\
Otherwise  $E^G_P(\phi)_{P'}$ (resp., $E^G_P(\phi)^w_{P'}$) is equal to: 
 $$ \sum_{w \in W(M' \vert G\vert
M)_\si}E^{M'}_{M'\cap w.P}((p_{M'} \otimes r_{M'}) (C(w, P', P, \delta_\chi) \phi)).$$
\end{theo}
 \begin{proof} 
  By $M'$-equivariance, it is enough to prove the equalities of the proposition evaluated at $m'=1$.  
 Then the result follows from the property (\ref{cct}) (resp.,  , Definition \ref{atemp}), from (\ref{gamma})
  together  with Theorem \ref{75} (resp., Theorem \ref{33}) and Lemma \ref{partcass}. \end{proof}

\section{Appendix}
\subsection{ Covariant distributions on an homogeneous space \index{covariant distributions} }
   \label{covd} Let $X$ be a $l$-space (cf. e.g. [BD], section 3 for the terminology) and let $V$ be a complex vector space. We denote  the space of locally constant functions on $X$ with compact support and with values in $\C$ (resp., $V$) by $C_c^{\infty} (X)$ (resp., $C_c^{\infty} (X,V)$). Notice that $C_c^\infty(X,V)$ is identified with $C_c^{\infty} (X)\otimes V$. We denote  the space of linear forms on $C_c^\infty(X,V)$ which are called $V$-distributions on $X$ by $\mathcal{D}'(X, V) $. 
  The support of a $ V$-distribution $T$ on $X$ is the complement of the largest open subset $O$ such that $T$ restricted to  $C_c^\infty (O)\otimes V$ is equal to zero.
  
 Let  $F$ be a closed subset of $X$ and let  $O$ denote $X\setminus F$.
From the exact sequence:$$0Ê\to C_c^{\infty} (O)\otimes V \to C_c^{\infty} (X)\otimes V\to C_c^{\infty} (F)\otimes V \to 0, $$ one sees that if $T  $ has support contained in $F$, then $T$ defines a $V$-distribution on $F$ which is called the distribution  on $F$  induced  by  $T$. \\
  Let  $(\pi, V)$ be a   smooth  representation of a $l$-group $G$. Recall that  $H_0(G,V)$ denotes the quotient of $V$ by the subspace generated by the elements of the form $\pi(g)v-v, g\in G , v\in V$. The dual of $H_0(G, V)$ is identified with the space of $G$-fixed linear forms on $V$.
  
Let $X$ be a $l$-space  on which $G$  acts continuously on the left. Let $\ll$ be the left regular  representation of $G$ on $C_c^\infty(X)$. A $V$-distribution $T$ on $X$ is said to be $\pi$-covariant if :
$$T(f - (\ll(g)\otimes \pi(g)) f)= 0, f \in  C_c^\infty (X)\otimes V.$$

\begin{lem} \label{Tcov} Let $H$ be a closed subgroup of $G$ and let  $(\pi, V)$ be a smooth representation of $G$. 
Let $T$ be  a $\pi$-covariant distribution on $ G/H$.  Let us denote  a left invariant Haar measure on $G$ (resp $H$) by $d_lg$ (resp., $d_lh$). 
Let us denote the modulus function of $G$ by $\delta_G$. It  satisfies:
$$\int_G f (gg_0)d_lg= \delta_G (g_0) \int_G f(g)\> d_lg, f \in C_c(G), g_0\in G.$$  We define a linear map from $ C_c^\infty (G) \otimes V$ to $ C_c^\infty (G/H) \otimes V$ by:
$$M_H f( gH):= \int_{H}f(gh) \> d_l h, \>  f \in C_c^\infty (G) \otimes V.$$Then  there exists a unique  $\eta\in V'$
such that $$\langle T, M_H f \rangle  = \int_G\langle   \pi'(g) \eta,   f(g) \rangle \>  d_lg, f \in C_c^\infty (G) \otimes V.$$ The linear form $ \eta$  will be called the value at $1$  of $T$  and denoted $ev_1T$ or $T(1)$.   Moreover:
$$ \pi'(h)  \eta = \delta_H^{-1}(h) \delta_G(h )\eta, h \in H.$$
Notice that $T(1)$ depends on the choice of $d_lg, d_lh$.

If $G/H$ has a nonzero left $G$-invariant measure, one has  $$\langle T, f \rangle = \int_{G/H} \langle \pi'(g) \eta, f(gH) \rangle  \> dgH, \> f \in C_c^\infty (G/H) \otimes V. $$ In that case $\eta$ is $H$-invariant and depends only on the choice of the left $G$-invariant measure on $G/H$. 
 
 \end{lem}
 \begin{proof} 
Let us assume first that $H=\{1\}$. 
We remark that $T \in (H_0(G, C_c^\infty( G) \otimes V)))'$ where $G$ acts on $C_c^\infty( G) \otimes V $  by the tensor product  $\ll \otimes \pi$ of the left regular representation with $\pi$.
 From [BD] Prop. 1.13 (iv), one sees: 
\ber \label{hom0}
The map $ f  \in  C_c^\infty( G) \otimes V \mapsto  \int_G \pi(g^{-1})  f(g) \> d_lg$ goes through the quotient  to an isomorphism of $H_0(G, C_c^\infty( G) \otimes V))$ with $V$Ê\eer
 Hence $T$ defines $\eta \in V'$ by  `` transport de structure''.
One sees that $T$ verifies :
 \beq \label{tfi} \langle T,f \rangle = \int _G \langle  \pi'(g)\eta,  f(g)  \rangle \>   d_{l} g, f \in C_c^\infty (G, V).\eeq 
  which proves our claim when $H=\{1\}$. 
  
 In general, we introduce a 
  $ V$-distribution $ \tilde{T} $ on $G$   by $$\langle \tilde{T}, f \rangle = \langle T, M_{H} f \rangle , f \in C_c^{ \infty} (G)\otimes V.$$ 
  Since $T$ $\pi$-covariant implies that $\tilde{T} $ is $\pi$-covariant, we may  apply   the first part of the proof to $\tilde{T}$. 
  Now one has $$M_HR_hf = \delta_H(h) M_H f \in C_c^{ \infty} (G)\otimes V,$$
  which implies   $$\langle \tilde{T},  R_hf \rangle =\delta_H(h)  \langle \tilde{T}, f \rangle , f \in C_c^{ \infty} (G)\otimes V. $$
   Hence, it follows from (\ref{tfi}) applied to $\tilde{T}$ that
  $$\langle \tilde{T},  R_hf \rangle  = \delta_H(h)  \int_G \langle    \eta,\pi(g)^{-1} f(g)  \rangle \> d_l g.$$
  But, using again  (\ref{tfi}), one has:
   $$ \langle \tilde{T}, R_hf \rangle  =  \int_G \langle   \pi'(g) \eta, f(gh)  \rangle d_l g= \delta_G(h) \int_G \langle  \pi' (g)\pi'(h^{-1}) \eta, f(g)  \rangle \> d_l g.$$ From the preceding equalities
  one deduces $$ \delta_H(h)  \int_G \langle    \eta,\pi(g)^{-1} f(g)  \rangle d_l g =\delta_G(h) \int_G \langle    \pi'(h^{-1})\eta,\pi(g)^{-1} f(g)  \rangle \>  d_l g.$$
  Then  (\ref{hom0})  implies:
  $$ \pi'(h)  \eta = \delta_H^{-1}(h) \delta_G(h )\eta, h \in H.$$
  \end{proof}
\subsection{ Bernstein's parameters of finite length smooth modules } \label{pb}
The Bernstein's center [DeliBe], $ZB(G)$, is identified with an  algebra of functions on the set, $\Omega(G)$,  of $G$-conjugacy classes of cuspidal pairs i.e. pairs $(L, \omega)$, where $L$  is a Levi-subgroup of 
$G$ and $\omega$ is a smooth, irreducible cuspidal representation of $L$. Here cuspidal means that the smooth coefficients of the representation are compactly supported modulo the center. 

If $(L,\omega)$ is such a pair,  we denote  its conjugacy class under $G$ by $(L,\omega)_G$. If $\Lambda\in \Omega(G)$,  we denote  the character of $ZB(G)$ given by the evaluation at $\Lambda$  by  $\chi_\Lambda$ and  the kernel of $\chi_\Lambda$ by $I_\Lambda$. It is a maximal ideal of $ZB(G)$.  We  say  that $\chi_\Lambda$ has Bernstein parameter $\Lambda$.

Let $(\pi, V) $ be  a smooth $G$-module of finite length. We  say that    $\{\Lambda_1, \dots, \dots \Lambda_p\} \subset \Omega(G)$  is the set of Bernstein's parameters of $(\pi,V)$,  if $V$ splits as direct sum of $G$-modules $V_1\oplus\dots \oplus V_p$ such that $V_i$ is non reduced to $\{0\}$ and is annihilated by a power of the ideal $I_{\Lambda_i}$ of ZB(G).  Then one sees easily:    \ber \label{Bchi} If $ \chi $ is an unramified character  of $G$  and if  the set  of Bernstein's parameter of $(\pi,V) $  is equal to  $\{ L_1, \omega_1)_G , \dots, (L_p, \omega_p)_G\} $, the set  of Bernstein's parameter of   $(\pi \otimes \chi,V) $ is to equal to $\{ (L_1, \omega_1\otimes \chi_{\vert L_1})_G , \dots, (L_p, \omega_p \otimes \chi_{\vert L_p})_G\} $. \eer
Let $P$ be a parabolic subgroup of $G$ with Levi subgroup $M$ and let $(\delta, E) $ be   a smooth representation of finite length of $M$ with set of Bernstein's parameters 
$\{ (L_1, \omega_1)_M , \dots, (L_p, \omega_p)_M\} $, where  $(L_i, \omega_i)$ is a cuspidal pair for $M$. Then, one has:
\ber \label{indpar}    The set of  Bernstein's parameters  of $(i^G_P E, i^G_P\delta)$  is   $\{ (L_1, \omega_1)_G , \dots, (L_p, \omega_p)_G\}.$\eer 
The following is an immediate consequence of the splitting of the category of smooth modules ( cf. [Be], [R])
\ber \label{splitsum}If  one has a short exact sequence $0\to V_1\to V_2\to V_3\to 0$ of finite length smooth $M'$-modules such that $V_1$ and $V_3$ have disjoint sets of Bernstein's parameters, then $V_2$ is isomorphic to the direct sum of $V_1\oplus V_3$. \eer 
\subsection{Some results on $\si$-parabolic subgroups} \label{5.1}

In this subsection, we change slightly the notations of the main body of the article.
\begin{lem}\label{dimorbite}
Let $G$ be a connected  algebraic group acting over the nonempty   variety $X$.\\
 (i)  Let $x\in X$. Then   $Y=Gx$ is  a smooth locally closed subset of  $X$. 
 (ii) There exists at least one closed orbit.\\
(iii) Let $G_x$ be the stabilizer of $x$ in $G$. Then   $dim Y= dim G - dim G_x$. 
  \end{lem}\begin{proof}
 (i)  and (ii) follows from [Hu], Proposition 8.3.\\
 (iii) The  morphism $G \to G.x$ is dominant: this morphism  is surjective  and $G$ is irreducible, hence $G.x$ is also irreducible and our claim follows from the discussion in the middle of [Bo], Ch. AG. 8.2. Then, the assertion on dimensions follows from [Bo], Ch. AG. 10.1,  with $X=G$, $Y=G.x$,  $W= \{x\}$, $Z$ the identity  component of $G_x$. 
 \end{proof}
\begin{lem} \label{orbiteouverte} Let $G$ be a connected  algebraic group acting over an irreducible nonsingular  variety $X$, with a finite number of orbits.
\\If  an  orbit,  $X'$,  of $G$ in $X$ has  the same dimension than $X$, then, $X'$ is open in $X$. \end{lem}
\begin{proof} 
We use induction on the number of orbits.  If this number is 1, our statement is clear.  Otherwise, if this number is greater than 1,  let $Y$ be a closed orbit in $X$. Then $Y$ is not equal to $X$. As $G$ is connected, $Y$ is irreducible. It follows from [Hu], Proposition 3.2, that  $dim \> Y < dim\> X$.  Then $X\setminus Y$ contains $X'$ and is irreducible as $Y$ is closed  and $X$ is irreducible. Moreover the action of $G$ on $X$ induces an action on $X\setminus Y$.
One applies the induction hypothesis.  \end{proof}
Let $G$ be a connected  reductive group defined over  a local  field $k$ of characteristic different from 2. Let $P$, $Q$ be two parabolic subgroups of $G$ defined over $k$. Let $G_k$ be the set  of $k$-points of $G$. We have similar notations for subgroups of $G$.\\
Let us show: 
\ber \label{a)} The $k$-parabolic subgroups $P$ and $Q$ of $G$ are opposed if and only if $P_k\cap Q_k$ is equal to $M_k$ where $M$ is a common Levi to $P$ and $Q$. \eer
If $P$ and $Q $ are opposed $P\cap Q$ is equal to their common Levi subgroup, $M$, and  it is clear that  $P_k\cap Q_k=M_k$.
Reciprocally if $P_k\cap Q_k = M_k$, then  $P \cap Q$ contains  the Zariski closure  of $M_k$  which is equal to $M$ by [Bo],  Corollary 18.3.
By looking to the $k$-parabolic subgroups of $G$ with Levi subgroup $M$, one sees that only the $k$-parabolic subgroup opposed to $P$ satisfies $P_k\cap Q_k=M_k$. This proves our claim.\\
Hence  if $ \si$ is an involution of $G$ defined over $k$ and $P$ is a parabolic subgroup of $G$ defined over $k$, $P$ is  a  $\si$-parabolic of $G$  if and only  if $P_k$ is a $\si$-parabolic subgroup of $G_k$.\\
Let us show:
\ber  \label{b)}If $ P$ is a $\si$-parabolic subgroup of $G$,  $\underline{p} +\underline{h}= \underline{g}$, where $\underline{g}$ (resp.,  $\underline{p}$) is the Lie algebra of $G$ (resp., $P$) and $\underline{h}$ is the Lie algebra of the  fixed point group  of $\si$. \eer 
As $ P$ is a $\si$-parabolic subgroup of $G$, $\underline{p}+\si(\underline{p})=\underline{g}$. Hence  any $X\in \underline{g}$ is of the form  $Y+\si(Z)$ with $ Y, Z\in \underline{p}$. Hence, one has:
$$X= Y-Z + (Z+\si(Z)).$$ The result follows from the fact that $\underline{h}$ is equal to the fixed point set of $\si$ in $\underline{g}$ (cf. [Ri],  proof of Lemma 2.4).
  Let $H$ be an open subgroup, defined over $k$,   of the fixed points  group of $\si$. We will show:
 \ber \label{kpk}Let $P$ be a $\si$-parabolic subgroup of $G$ defined over $k$.\\
Then:\\ a)$HP$ is open in $G$, \\ b)$H_kP_k$ is open in $G_k$
\eer 
The assertion a)  follows from [HW] Lemma 4.8. The assertion b) reduces to  the case where $P$ is  a minimal $\si$-parabolic subgroup of $G$ defined over $k$. In that case it follows from 
[HW],  Definition 13.1 and  Proposition 13.4.

\begin{lem} \label{hpouvert}
Let $P$ be a $\si$-parabolic subgroup of $G$ defined over k.\\
Let $x\in G_k$.\\
The following conditions are equivalent:\\
(i) $H_kxP_k$ is open in $G_k$.\\
(ii) $HxP$ is open in $G$.\\
(iii) $x \in HP$.\\
(iv) $xPx^{-1}$ is a $\si$-parabolic subgroup of $G$ defined over $k$.
\end{lem}
\begin{proof}
(i) implies (ii): Let $P'$ be a minimal $k$-parabolic subgroup of $G$ contained in $P$. There are finitely many $(H_k, P'_k)$-double cosets in $G_k$ (cf. [HW], Corollary 6.16). Hence $H_kxP_k$ contains an open $(H_k,P'_k)$-double coset, $H_k x'P_k$ with $ x'= x p$ and $p\in P_k$. From [HH], Proposition 3.5, one sees that $x'.P'$ is contained in  a minimal $\si$-parabolic subgroup of $G$, hence $H x'P '$ is open  in $G$ by [HW], Lemma 4.8. Hence $Hx'P= HxP $ is open in $G$. Hence (i) implies (ii).\\(ii) implies (iii): as the union  of the closures of the open  $(H,P)$-double coset in $G$ is  equal to $G$  and as $G$ is connected, hence  irreducible, Êthere is only one open $(H,P)$-double coset in $G$.  From (\ref{kpk}) a), one knows that $HP$ is open. Hence (ii) implies (iii). 
\\ (iii) implies (iv) because the  conjugation  by an element of $H$  preserves the set of  $\si$-parabolic subgroups of $G$. 
\\ (iv) implies (i) follows from (\ref{kpk}) b).
\end{proof}
\section{Compact open subgroups of Lie groups over $\F$} \label{Lie}
Let $G$ be a Lie group over $\F$ in the sense of [Bou] Ch.III.1, Definition 1..and let $\un{g}$ be its Lie algebra.
\\We will use an idea given by Deligne in [DeliBe], top of p.16.  We fix  an analytic bijective  map $\psi: V \to W $  between an  open neighborhood $V$ of 0 in $\un{g}$ to an open neighborhood, $W$,  of 1 in $G$. We assume that its differential at 0  is   the identity. Such a map will be called a good chart at 1 for $G$.
\begin{lem}  \label{DB} 
 Let $\mathfrak{P}$ be the the maximal ideal of the ring of integers, $\mathfrak {O}$ of $\F$. We fix a basis $(X_j)$ of $\un{g}$. \\ (i) Let $\LL_n\un{g}$ (resp. $\LL\un{g}$) be the $\mathfrak{P}^n$-module ( resp $\mathfrak{O}$-module) generated by the $X_i$.  It is a basis of neighborhoods  of $0$ in $\g$. Let $J_n$ be the image of $\LL_n\un{g}$ under $\psi$, which is defined for $n$ large enough.  
\\  Then for $n$ large enough, $J_n$ is a compact open  subgroup of $G$.\\ (ii)  If $J'_n$ is defined with another good chart  $\psi'$,  one has $J_n=J'_n$ for $n$ large enough.\\
(iii) Let $\Theta$ be a family  of automorphisms of the Lie group $G$ whose differential preserves $\LL\un{g}$.  Then it preserves $\LL_n\un{g}$ for all $n \in \N^*$. We assume moreover that $\Theta$ has the structure of compact analytic manifold over $\F$ and that the map $\Theta \times G \to G$, $(\theta, g)\mapsto \theta(g)$ is analytic.  Then for $n$ large enough, $J_n$ is invariant under  every $\theta\in \Theta$.\\
(iv) Let us assume that we are given three closed Lie subgroups $G_1$, $G_2$, $G_3$ of $G$  and  vector subspaces $\un{g'}_i \subset \un{g}_i$, $i=1,2,3$ of $\un{g}$  such  that $\LL \un{g} =(\LL \un{g}\cap \un{g}'_1) \oplus( \LL \un{g}\cap \un{g}'_2) \oplus ( \LL \un{g}\cap \un{g}'_3)$. Then for $n$ large enough one has:
$$J_n= (J_n \cap G_1)(J_n \cap G_2) (J_n \cap G_3). $$
\end{lem}
\begin{proof}
(i) The fact that $J_n$ is compact and open follows from the fact that  $\LL_n\un{g}$ is compact and open. We choose $n$ large enough so that $J_nJ_n\subset W$. \\ 
Let us enote by $x_j$ the $j$-th coordinate map on $\un{g}$. 
We define:  $$\vert X \vert = Sup_j \vert x_j( X) \vert _\F $$
Let $X, X' \in \LL_n\un{g}$, and let us  study  $x_j ( \psi^{-1}( \psi (X) \psi(X') ))$. 
By our hypothesis on $\psi$ and from the formula of the differential of the product in $G$,  one sees that the differential  of this map of $(X, X')$ is simply $(X, X') \mapsto x_j(X)+ x_j(X')$.
The definition of the differential shows that:
$$ x_j ( \psi^{-1}( \psi (X) \psi(X'))) = x_j(X)+ x_j(X')+Sup( \vert X\vert, \vert X'\vert ) \varepsilon( X,X').$$
where $ \varepsilon (X, X')$ tends to zero if $(X, X') $ tends to $(0,0)$. Let $n_0 $  be large enough such that $\vert \varepsilon( X,X')\vert  <1$ for $X, X'$ in $\LL_{n_0}\un{g} $. One deduces from the above equality that for $n \in \N$ larger than $n_0$,  $ jj' \in J_n$ for all $j, j'\in J_n$.
Similarly one sees  that $j^{-1} \in J_n $ if $j\in J_n$ and $n$ is larger than $n_0$. 
Then  (i) follows. \\
(ii) One proceeds as in (i), by considering the map $X \mapsto x_j( \psi^{-1}(\psi( X)(\psi'(X))^{-1}))$, whose differential at $0$ is equal to zero. Arguing as in (i), for $n$ large one sees that $\psi(X)(\psi'(X))^{-1} \in J_{n+1}$.   if $X\in \LL_n \un{g}$.  Hence if $j \in J_n$, we have found $j'\in J'_n$ such that $j(j')^{-1} \in J_{n+1}$.  Using that $J'_n$ is a group, and proceeding inductively, we find a sequence $(j'_p)$ in $J'_n$ such that  $j( j'_p)^{-1} \in J'_{n+p}$. Hence $j'_p$ converges to $j$. But $J'_n$ is compact, Hence $j \in J'_n$ and $J_n \subset J'_n $. The reverse inclusion is proved similarly. This proves (ii).
\\ (iii) By compactness, one is reduced to prove that for every $\theta_0\in \Theta$, there exist a neighborhood $V(\theta_0)$ of $\theta_0$ in $\Theta$  and $n_0\in \N$ such that for all $n \in \N$ greater than $n_0$ and for all $\theta \in V(\theta_0)$  one has $\theta (J_n)\subset J_n$.
 We fix $\theta_0$. 
 We denote again  the differential at 1 of $\theta\in \Theta$ by $\theta$. 
Shrinking $V$ and $\Theta$ if necessary if necessary,  let us consider the analytic map from $(\Theta,V)$ to $\un{g}$, $ ( \theta, X)\mapsto \psi^{-1} [(\theta(\psi(X)) (\psi (\theta(X))^{-1} ]$. By looking to partial derivatives, one sees that  its  differential  at $(\theta_0,0)$ is equal to zero.  One sees as above that it implies that for 
$\theta$ is in a small neighborhood $V(\theta_0)$ of $\theta_0$ and for $n$  large enough one has  
$$(\theta(\psi(X)) (\psi (\theta(X))^{-1} \in J_n, X \in \LL_n \un{g}$$
As $\theta (X)  \in \LL_n\un{g}$, this implies $\theta(\psi(X)\in J_n$. This proves (iii).  
\\(iv)  Let $\psi_i$ be a good chart at zero of $G_i$, $i=1,2,3$. From the first part of the lemma, one can use the following map  $\psi$ to study $J_n$. 
$$\psi(Y_1+Y_2+Y_3)= \psi_{\un{g}_1} (Y_1)\psi_{\un{g}_2} (Y_2)\psi_{\un{g}_3} (Y_3), Y_i \in \un{g}'_i, i=1,2,3.$$
One see easily that  $\LL_n \un{g} =(\LL_n \un{g}\cap \un{g}'_1) \oplus( \LL_n \un{g}\cap \un{g}'_2) \oplus ( \LL_n \un{g}\cap \un{g}'_3)$.
From the definition of $J_n$ it is clear that $$J_n\subset  (J_n \cap G_1) (J_n \cap G_2) (J_n \cap G_3).$$
The reverse inclusion being clear this proves the lemma.  \end{proof}
\begin{lem} \label{LieG} If $\un{G}\subset GL(n)$ is a linear algebraic group define over $\F$,  the group $G= \un{G}(\F)$ has a structure of analytic Lie group, whose Lie algebra $\un{g}$ is the Lie algebra, $\un{g}(\F)$,   of $\F$-points of the Lie algebra of $\un{G}$.  \end{lem} \begin{proof}  

The following fact seems to be well known but  by lack of reference, we give a proof.
We thank 
Joseph Bernstein, Bertrand Lemaire and  Jean-Pierre Labesse for discussions on this fact.
\ber \label{Be}Let $X$  be an affine irreducible variety. Let us assume that $X$ is non singular and defined over $\F$ and let $\Omega$ be the algebraic closure of $\F$. Let us assume that $X$ is a Zariski closed set  in $\Omega^n$. Let $x \in X(\F)$.\\
(a) There exists $f_1, \dots , f_p$ polynomial functions vanishing on $X$ with coefficients in $\F$   such that the rank of their differentials at $x$  is equal to the codimension of $X$ in $\Omega^n$.\\
(b) If $f_1, \dots , f_p$ are as above, let  $V(f_1, \dots , f_p)$ be the set of  their common zeroes. There exists an open subset $O$   in $\Omega^n$ such that $x\in O$ and such that  $X\cap O$ is is equal to $V(f_1, \dots, f_p)\cap O$. 
\eer  
(a) follows from the Definition of nonsingular points  in [Ha] I.5 and from the fact that the ideal f functions vanishing on $X$ is generated by polynomial with coefficients in $\F$.
\\ (b) Let $Y=V(f_1, \dots , f_p)$ and let $O' $ be the open  set of elements of $\Omega^n$ such that the rank of the differential of $f_1, \dots, f_p$ is of rank $p$. It is open in $\Omega^n$ and it contains $x$.
By [Ha], proof of Theorem 5.3, one knows that 
\ber \label{rang}The rank $r$  of the Jacobian matrix, at a point of $y\in Y$,  associated to polynomial functions vanishing  on $Y$ is at most $n-{\rm dim}( Y)$ i.e ${ \rm codim} (Y) \geq r$.
 \eer 
Applying this to $f_1,\dots,f_p$ one has $ { \rm codim} (Y) \geq p $, hence the codimension of any irreducible component of $Y$ is of codimension at greater or equal to  $p$.  
On the other hand, from the definition of $Y$ by $p$ equations  and [Hu] 3.4 Corollary B, every irreducible component of $Y$ has codimension at most $p$. Hence every irreducible component of $Y$ is of dimension $n-p$. 
  Then from the definition of  a non singular point  in [Ha] I.5 and (\ref{rang}), on sees that $x$ is non singular  in $Y$. 
 
 A simple point  in the sense of [Bo]  17. 1 is  exactly a non singular point in the sense  of [Ha], due to [Ha] Theorem 5.1.  Hence, by [Bo] Theorem 17.1, one sees that $x$ belongs to only one irreducible component of $Y$. Let $Z $ be the irreducible component of $Y$ which contains $x$ and let $O''$ the complement in $\Omega^n$ of the (finite) union of the other irreducible components of $Y$. It is open in $\Omega^n$ and contains $x$.

Let us show that $O:=O' \cap O''$ satisfies (b). In fact $X'=X\cap O$ is contained in $Z'= Z \cap O $ and have the same dimension. It  contains $x$. Moreover as $X$ ( resp. $Z$ ) is irreducible, $X'$ and $Z'$ are irreducible. Hence from [Hu], Proposition 3.2, one has $X'=Z'$. This proves (\ref{Be}).

 
  Let $f_1$, ..., $f_p$  be a set of polynomials, with coefficients in $\F$,   as in (\ref{Be}) for $X=\un{G}$.  Then we  consider the map
  $h:\F^n  \to \F^n$, $x \mapsto (f_1(x), \dots, f_p(x), x_1(x), \dots x_{n-p}(x))$, where the affine  coordinates $ (x_1, \dots, x_n)$ in $\F^n$   are equal to zero at the neutral element 1 of $G$ and are choosen   such that the   differential of $ h$ at 1 is bijective.  The  Inverse Function Theorem (cf. [Se], LG.2.10) gives a local analytic inverse to this map, $\phi$. 
  
Let $\psi$ be the map defined on an open  neighborhood $V$ of $ (0, , \dots,0)$  in $\F^{n-p}$ by   $\psi(y)= \phi((0, \dots, 0), y)$.   From (\ref{Be}) and from the fact that $\phi$ is a homeomorphism,  one sees that if $V$ is small enough, $\psi $ is a homeomorphism  from   $V$  onto an open neighborhood of $1$ in $G$. The differential at 0 of $\psi$ is injective and  its image is a subspace of the space of common zeroes in $\F^n$ of the differential at 1of $f_1, \dots, f_p$.  The space  of common zeroes in  of the differential at 1 of $f_1, \dots, f_p$ is the  tangent space at 1 of $\un{G}$, hence is equal to $\un{g}$. For reasons of dimension this implies that  the image of the differential at 0 of $\psi$ is equal to $\un{g}(\F)$. 

 The map $\psi$  gives one chart of $G$ at 1. Then by translation this gives a chart at every point of $G$. The analyticity of the change of charts follows from the analyticity of $\phi$ and its inverse.  Similarly one proves that $G$ has a structure of Lie group. It is even an analytic Lie subgroup of $GL(n, \F)$.  The Lie algebra of $G$ is the image  of the differential at 0 of $\psi$. Hence the Lie algebra of $G$ is  $\un{g}(\F)$. 
 \end{proof} 
 
\section{References}
\noindentÊ[Ba], van den Ban E.P., The principal series for a reductive symmetric space I. $H$-fixed distribution vectors. Ann. Sci. Ecole Norm. Sup. (4) 21 (1988), 359--412. 

\noindent [Be] Bernstein J.,  Second Adjointness Theorem for representations of reductive p-adic groups, unpublished
manuscript, web page of Joseph Bernstein.

\noindent [BD] Blanc P., Delorme P.,  Vecteurs distributions H-invariants de
reprŽsentations induites pour un espace symŽtrique rŽductif $p$-adique
G/H, Ann. Inst. Fourier, 58 (2008),  213--261.

\noindent [Bo] Borel A.  Linear algebraic groups. Second edition. Graduate Texts in Mathematics,
126. Springer-Verlag, New York, 1991.

\noindent [BoTi] Borel A., Tits J., Groupes rŽductifs. Inst. Hautes Etudes Sci. Publ. Math.  27 (1965) 55--150.

\noindent [Bou] Bourbaki, Groupes et algbres de Lie. Chapitres 2 et 3,  ElŽments de MathŽmatique  XXXVII, Herman, Paris 1972.

\noindent [Bu] Bushnell C., Representations of reductive p-adic groups: localization of Hecke algebras and applications. J. London Math. Soc. (2) 63 (2001)364--386.

\noindent [C] Casselman W.,  Introduction  to the theory of admissible representations of $p$-adic reductive groups,
http://www.math.ubc.ca/$\tilde{}$ cass/research.html. 

\noindent [DeliBe] Deligne P.,  Le ``centre'' de Bernstein rŽdigŽ par Pierre Deligne. Travaux en Cours,  Representations of reductive groups over a local field,  1--32, Hermann, Paris, 1984. 

\noindentÊ[D1] Delorme P., IntŽgrales d'Eisenstein pour les espaces symŽtriques rŽductifs: tempŽrance, majorations. Petite matrice $B$, J. Funct. Anal. 136 (1996), 422--509. 

\noindent [D2] Delorme P., The Plancherel formula on reductive symmetric spaces from the point of view of the Schwartz space.   in Lie theory,   135--175, Progr. Math., 230, Birkhauser Boston, Boston, MA, 2005. 

\noindent[D3] Delorme P., Constant term of smooth $H_\psi$-invariant functions,  Trans. Amer. Math. Soc.  362  (2010),  933--955. 

 \noindent[D4] Delorme P., ThŽorme de Paley-Wiener pour les fonctions de Whittaker sur un groupe rŽductif $p$-adique,   to appear in J. Inst. Math. Jussieu.

\noindent[D5] Delorme P., Formule de Plancherel pour les fonctions de Whittaker sur un groupe rŽductif $p$-adique,  to appear in Ann. Inst. Fourier. 

\noindent [Ha] Hartshorne R., Algebraic Geometry, Graduate Texts in Math. 52, Spinger Verlag, New York,Heidelberg, Berlin,  1977.

  \noindent [Hei] Heiermann, V.,  Une formule de Plancherel pour l'algbre de Hecke d'un groupe rŽductif $p$-adique. Comment. Math. Helv. 76 (2001),  388--415.
  
  \noindent [HH] Helminck A.G., Helminck G.F., A class of parabolic $k$-subgroups
associated with symmetric
$k$-varieties.  Trans. Amer. Math. Soc.  350  (1998)
4669--4691.

\noindent[HW] Helminck A. G.,  Wang  S. P., 
On rationality properties of involutions of reductive groups.
Adv. Math. 99 (1993) 26--96.

\noindent [Hu] Humphreys J. E, Linear algebraic groups, Graduate Text In Math. 21, Springer, 1981.
 
\noindent [KT1] Kato S., Takano K., Subrepresentation theorem for $p$-adic symmetric spaces,  Int. Math. Res. Not. IMRN 2008, no. 11

\noindent [KT2] Kato S., Takano K., Square integrability of representations on $p$-adic symmetric spaces. J. Funct. Anal. 258 (2010)1427Ð-1451.
  
  \noindentÊ[KnSt] Knapp A., Stein E., Intertwining operators for semisimple groups. II. Invent. Math. 60 (1980) 9--84.
 
 \noindent [L], Lagier N., Terme constant de fonctions sur un espace symŽtrique rŽductif p-adique,
J. of Funct. An., 254 (2008) 1088--1145. 

\noindent [M] Matsuki, Closure relations for orbits on affine symmetric spaces under the action of minimal parabolic subgroups. Representations of Lie groups, Kyoto, Hiroshima, 1986, 541--559, Adv. Stud. Pure Math., 14, Academic Press, Boston, MA, 1988.

\noindent [R] Renard, ReprŽsentations des groupes rŽductifs $p$-adiques. Cours SpŽcialisŽs, 17. SociŽtŽ MathŽmatique de France, Paris, 2010. 
 
 \noindent [Ri] Richardson, R. W. Orbits, invariants, and representations associated to involutions of reductive groups. Invent. Math. 66 (1982),  287Ð-312.
 
 \noindent [S] Sauvageot, Principe de densitŽ pour les groupes rŽductifs. Compositio Math. 108 (1997) 151--184.

\noindent [Se] Lie algebras and Lie groups, Lectures given at Harvard University,  Benjamin 1965.

 \noindent [Wal] Waldspurger J.-L.,  La formule de Plancherel pour les groupes
$p$-adiques (d'aprs Harish-Chandra), J. Inst. Math. Jussieu  2  (2003),  235--333.

\printindex

\end{document}